\documentclass[<1p>]{elsarticle}

\journal{Journal of Pure and Applied Algebra}
\usepackage{amsmath}
\usepackage{amssymb}
\usepackage{amsthm}
\usepackage{enumerate}
\usepackage{wasysym}
\usepackage{dsfont}
\usepackage{comment}
\usepackage{tipa}
\usepackage{tikz}
\usepackage{tikz-cd}
\usepackage{extarrows}
\usepackage{mathrsfs}
\usepackage{quiver}
\usepackage{fullpage}

\newtheorem{defn}{Definition}[section]
\newtheorem{prop}[defn]{Proposition}
\newtheorem{theo}[defn]{Theorem}
\newtheorem{lem}[defn]{Lemma}
\newtheorem{cor}[defn]{Corollary}
\newtheorem{rmk}[defn]{Remark}
\newtheorem{claim}[defn]{Claim}

\newtheorem*{claim*}{Claim}
\newtheorem*{just*}{Justification}
\newtheorem*{assumption*}{Assumption}
\newtheorem*{lem*}{Lemma}
\newtheorem*{prop*}{Proposition}
\newtheorem*{thm*}{Theorem}
\newtheorem{sub-claim}[defn]{Sub-Claim}
\newtheorem*{theo*}{Theorem}

\newcommand{\T}{\mathbb{T}}
\newcommand{\la}{\left\langle}
\newcommand{\ra}{\right\rangle}
\newcommand{\C}{\mathbb{C}}
\newcommand{\Z}{\mathcal{Z}}
\newcommand{\B}{\mathcal{B}}

\newcommand{\PTmod}{\T\mathsf{mod}}

\newcommand{\J}{\mathcal{J}}
\newcommand{\SigmaJ}{\Sigma^\J}
\newcommand{\TJ}{\T^\J}

\newcommand{\PTJmod}{\TJ\mathsf{mod}}
\newcommand{\Sets}{\mathsf{Set}}

\newcommand{\Id}{\mathsf{Id}}
\newcommand{\x}{\mathsf{x}}
\newcommand{\Term}{\mathsf{Term}}
\newcommand{\Sort}{\mathsf{Sort}}

\newcommand{\id}{\mathsf{id}}
\newcommand{\Fun}{\mathsf{Fun}}
\newcommand{\Aut}{\mathsf{Aut}}

\newcommand{\Group}{\mathsf{Group}}

\newcommand{\Tmod}{\T\mathsf{mod}}
\newcommand{\TJmod}{\TJ\mathsf{mod}}

\newcommand{\Horn}{\mathsf{Horn}}
\newcommand{\Dom}{\mathsf{Dom}}
\newcommand{\ob}{\mathsf{ob}}
\newcommand{\mor}{\mathsf{mor}}
\newcommand{\inn}{\mathsf{inn}}
\newcommand{\cod}{\mathsf{cod}}
\newcommand{\limit}{\mathsf{lim}}

\begin{document}

\begin{frontmatter}

\title{Inner automorphisms of presheaves of groups}

\author[mymainaddress]{Jason Parker}


\ead{parkerj@brandonu.ca}

\address[mymainaddress]{Department of Mathematics and Computer Science, Brandon University, 270-18th Street, R7A 6A9, Brandon, Manitoba, Canada}

\begin{abstract}
It has been proven by Schupp and Bergman that the inner automorphisms of groups can be characterized purely \emph{categorically} as those group automorphisms that can be coherently extended along any outgoing homomorphism. One is thus motivated to define a notion of \emph{(categorical) inner automorphism} in an arbitrary category, as an automorphism that can be coherently extended along any outgoing morphism, and the theory of such automorphisms forms part of the theory of \emph{covariant isotropy}. In this paper, we prove that the categorical inner automorphisms in any category $\Group^\J$ of presheaves of groups can be characterized in terms of conjugation-theoretic inner automorphisms of the component groups, together with a natural automorphism of the identity functor on the index category $\J$. In fact, we deduce such a characterization from a much more general result characterizing the categorical inner automorphisms in any category $\Tmod^\J$ of presheaves of $\T$-models for a suitable first-order theory $\T$. 
\end{abstract}  

\begin{keyword}
inner automorphism \sep isotropy \sep quasi-equational theory \sep presheaf \sep conjugation 
\MSC[2020] 08A35 \sep 08C05 \sep 18C10 \sep 20A15 \sep 20J15 
\end{keyword}

\end{frontmatter}

\section{Introduction}

An automorphism $\alpha : G \xrightarrow{\sim} G$ of a group $G$ is said to be \emph{inner} if there is some element $s \in G$ with respect to which $\alpha$ is defined by conjugation, in the sense that $\alpha(g) = sgs^{-1}$ for any $g \in G$. In \cite[Theorem 1]{Bergman}, George Bergman proved that these inner automorphisms can be characterized purely \emph{categorically}, without reference to group elements or conjugation, as the automorphisms that can be \emph{coherently extended} along any morphisms out of their domains.\footnote{Earlier versions of this result were also proven by Pettet \cite{Pettet} and Schupp \cite{Schupp}.} More precisely, he showed that an automorphism $\alpha : G \xrightarrow{\sim} G$ is inner if and only if one can define a group automorphism $\alpha_f : H \xrightarrow{\sim} H$ of the codomain of any group homomorphism $f : G \to H$ out of $G$ in such a way that $\alpha_{\id_G} = \alpha$ and the resulting family of automorphisms $(\alpha_f)_f$ is \emph{coherent}. The latter condition means that if $f : G \to H$ and $f' : H \to K$ are any composable group homomorphisms out of $G$, then the following square must commute: 
\[\begin{tikzcd}[ampersand replacement=\&, row sep = huge, column sep = huge]
H \arrow{r}{\alpha_{f}}\arrow{d}[swap]{f'} \& H\arrow{d}{f'} \\
K\arrow{r}[swap]{\alpha_{f' \circ f}} \& K
\end{tikzcd}\] 
Such a family of automorphisms $(\alpha_f)_f$ is just a natural automorphism of the projection functor $G/\Group \to \Group$, which sends a morphism $f : G \to H$ to its codomain $H$. We refer to such a family of automorphisms as an \emph{extended inner automorphism} of $G$. Bergman's result can therefore be restated as saying that a group automorphism $\alpha : G \xrightarrow{\sim} G$ is inner iff there is an extended inner automorphism $(\alpha_f)_f \in \Z(G)$ with $\alpha_{\id_G} = \alpha$. These extended inner automorphisms form a group $\mathcal{Z}(G)$, which we call the \emph{covariant isotropy group} of $G$.   

We can now generalize these ideas from the category $\Group$ to an arbitrary category. For any category $\C$, one can define its \emph{covariant isotropy group (functor)} $\Z_\C : \C \to \Group$, which sends any object $C \in \ob\C$ to the group of natural automorphisms of the projection functor $C/\C \to \C$. More concretely, if we define $\mathsf{Dom}(C) := \left\{ f \in \mor(\C) : \mathsf{dom}(f) = C\right\}$, then an element $\pi \in \Z_\C(C)$ is a $\mathsf{Dom}(C)$-indexed family of automorphisms \[ \pi = \left(\pi_f : \mathsf{cod}(f) \xrightarrow{\sim} \mathsf{cod}(f)\right)_{f \in \mathsf{Dom}(C)} \] with the property that if $f : C \to C'$ and $f' : C' \to C''$ are any composable morphisms out of $C$, then $\pi_{f' \circ f} \circ f' = f' \circ \pi_f$ as in the commutative square above. Intuitively, an element $\pi \in \Z_\C(C)$ provides an automorphism $\pi_{\id_C} : C \xrightarrow{\sim} C$ that can be coherently or functorially \emph{extended} along any morphism out of $C$. The covariant isotropy group of a category $\C$ can therefore be regarded as encoding a notion of \emph{inner automorphism} or \emph{conjugation} for $\C$, abstracting from the initial case of $\C = \Group$. Motivated by the considerations above, we may then refer to the elements of $\Z_\C(C)$ as the \emph{extended inner automorphisms} of $C \in \ob(\C)$, while a \emph{(categorical) inner automorphism} of $C$ is an automorphism $f : C \xrightarrow{\sim} C$ for which there is some extended inner automorphism $\pi \in \Z_\C(C)$ with $\pi_{\id_C} = f$, i.e. a (categorical) inner automorphism is an automorphism that can be coherently extended along any outgoing morphism.\footnote{The general theory of categorical isotropy was introduced in \cite{Funk}.}    

In \cite{MFPS} and \cite{FSCD}, the present author and his collaborators studied the covariant isotropy group of the category $\Tmod$ of models of any finitary \emph{quasi-equational theory} $\T$ in the sense of \cite{Horn}.\footnote{Quasi-equational theories are an equivalent formulation of multi-sorted essentially algebraic theories.} We showed therein that the covariant isotropy group of a model $M$ of a finitary quasi-equational theory $\T$ can be \emph{logically} or \emph{syntactically} described in terms of the sort-indexed families in $\prod_C M\la \x_C \ra$ that are substitutionally invertible and commute generically with the operations of $\T$, where $M\la \x_C \ra$ is the $\T$-model obtained from $M$ by freely adjoining a new element $\x_C$ of sort $C$. Using this logical or syntactic characterization of the covariant isotropy group of $\Tmod$, we then provided explicit characterizations of the (extended) inner automorphisms in many prominent categories of mathematical interest, including the categories of (abelian) groups, (commutative) monoids, (commutative) rings with unit, lattices, racks and quandles (see also \cite{racks}), and strict monoidal categories.

In \cite[5.2]{FSCD} the authors also proved an explicit characterization of the covariant isotropy group of any presheaf category $\Sets^\J$ on a small category $\J$, and showed that it is the constant functor $\Sets^\J \to \Group$ with value $\Aut(\Id_\J)$, the group of natural automorphisms of the identity functor on $\J$. It is trivial to see that $\Sets$ is the category of models of a finitary quasi-equational theory, namely the single-sorted theory with no operations and no axioms. It is the primary purpose of the present paper to extend the characterization of covariant isotropy for presheaf categories $\Sets^\J$ to categories of the form $\Tmod^\J$ for arbitrary quasi-equational theories $\T$.\footnote{In fact, we will need to eventually impose some modest conditions on $\T$; see Definitions \ref{singleindeterminateisotropy} and \ref{singlesortednontotaloperations}.} In other words, we will explicitly characterize the (extended) inner automorphisms in such functor categories $\Tmod^\J$, and it will turn out that the (extended) inner automorphisms of a functor $F : \J \to \Tmod$ can be described in terms of the automorphism group $\Aut(\Id_\J)$ and the (extended) inner automorphisms of the component $\T$-models $F(i) \in \Tmod$ (for $i \in \ob\J$). From this general characterization, we will then extract an explicit characterization of the (extended) inner automorphisms in any category $\Group^\J$ of presheaves of groups. In particular, we will show that the (categorical) inner automorphisms of any functor $F : \J \to \Group$ can be characterized in terms of the automorphism group $\Aut(\Id_\J)$ and the conjugation-theoretic inner automorphisms of the component groups $F(i) \in \Group$.

We now provide a brief overview of the paper. In Section \ref{background} we first review some background material from \cite{Horn} on quasi-equational theories and their models, and we conclude by recalling the characterization of covariant isotropy for the categories of models of such theories that was proven in \cite{FSCD}. We then proceed in Section \ref{logical} to show that if $\T$ is a quasi-equational theory and $\J$ a small category, then the functor category $\Tmod^\J$ can be axiomatized as the category of models of a quasi-equational theory $\TJ$, and we then prove an explicit \emph{logical} characterization of the covariant isotropy group of $\TJmod \cong \Tmod^\J$. In Section \ref{categorical} we deduce from these results a \emph{categorical} characterization of the covariant isotropy group of $\Tmod^\J$, and hence of the (extended) inner automorphisms of functors $\J \to \Tmod$. In the final Section \ref{applications}, we deduce some important special cases of our main results, and conclude by providing an explicit characterization of the (extended) inner automorphisms in any category $\Group^\J$ of presheaves of groups. 

The content in this paper (especially from Section \ref{logical} onward) is largely based on the final chapter of the author's recent PhD thesis \cite{thesis}; in the interest of brevity, we have therefore chosen to refer the reader to this source for many of the (technical) proofs.   

\section{Covariant isotropy of locally finitely presentable categories}
\label{background}

In this section, we review the techniques developed in \cite{FSCD, thesis} for computing the covariant isotropy group of the category of models of any finitary quasi-equational theory, equivalently the covariant isotropy group of any locally finitely presentable category. The initial material in this section borrows heavily from \cite{Horn}. 

\begin{defn}
\label{signature}
{\em 
A \emph{(first-order) signature} $\Sigma$ is a pair of sets $\Sigma = (\Sigma_{\mathsf{Sort}}, \Sigma_{\mathsf{Fun}})$ such that $\Sigma_{\mathsf{Sort}}$ is the set of \emph{sorts} and $\Sigma_{\mathsf{Fun}}$ is the set of \emph{function/operation symbols}. Each element $f \in \Sigma_{\mathsf{Fun}}$ comes equipped with a pair $\left((A_1, \ldots, A_n), A\right)$, where $n \geq 0$ is a natural number and $A, A_i$ are sorts for all $1 \leq i \leq n$, which we write as $f : A_1 \times \ldots \times A_n \to A$. In case $n = 0$, we write $f : A$. \qed
}
\end{defn} 

\begin{defn}
\label{terms}
{\em Let $\Sigma$ be a signature. For every sort $A \in \Sigma_{\mathsf{Sort}}$, we assume that we have a countable set $V_A$ of variables of sort $A$. We now define the set $\mathsf{Term}(\Sigma)$ of \emph{terms} of $\Sigma$ recursively as follows, while simultaneously defining the \emph{sort} and the set $\mathsf{FV}(t)$ of \emph{free variables} of a term $t \in \mathsf{Term}(\Sigma)$:
\begin{itemize}
\item If $A \in \Sigma_{\mathsf{Sort}}$ and $x \in V_A$, then $x \in \mathsf{Term}(\Sigma)$ is of sort $A$ with $\mathsf{FV}(x) := \{x\}$.

\item If $f : A_1 \times \ldots \times A_n \to A$ is a function symbol of $\Sigma$ and $t_i \in \mathsf{Term}(\Sigma)$ with $t_i : A_i$ for each $1 \leq i \leq n$, then $f(t_1, \ldots, t_n) \in \mathsf{Term}(\Sigma)$ is of sort $A$, and $\mathsf{FV}\left(f(t_1, \ldots, t_n)\right) := \bigcup_{1 \leq i \leq n} \mathsf{FV}(t_i)$. In particular, if $c$ is a constant symbol of sort $A$, then $c$ is a term of sort $A$ with $\mathsf{FV}(c) = \emptyset$.
\end{itemize}

\noindent If $t \in \mathsf{Term}(\Sigma)$ and $\mathsf{FV}(t) = \emptyset$, then we call $t$ a \emph{closed} term. If $t \in \mathsf{Term}(\Sigma)$, then we write $t(\vec{x})$ to mean that $\mathsf{FV}(t) \subseteq \vec{x}$. 
\qed
}
\end{defn}

\begin{defn}
\label{Hornformulas}
{\em Let $\Sigma$ be a signature. We define the class $\mathsf{Horn}(\Sigma)$ of \emph{Horn formulas} over $\Sigma$ recursively as follows, while simultaneously defining the set $\mathsf{FV}(\varphi)$ of \emph{free variables} of a formula $\varphi \in \mathsf{Horn}(\Sigma)$:
\begin{itemize}
\item If $t_1, t_2 \in \mathsf{Term}(\Sigma)$ are terms of the same sort, then $t_1 = t_2 \in \mathsf{Horn}(\Sigma)$, and $\mathsf{FV}(t_1 = t_2) := \mathsf{FV}(t_1) \cup \mathsf{FV}(t_2)$.

\item $\top \in \mathsf{Horn}(\Sigma)$ (the empty conjunction), and $\mathsf{FV}(\top) := \emptyset$.

\item If $\varphi_1, \varphi_2 \in \mathsf{Horn}(\Sigma)$, then $\varphi_1 \wedge \varphi_2 \in \mathsf{Horn}(\Sigma)$, and $\mathsf{FV}\left(\varphi_1 \wedge \varphi_2\right) := \mathsf{FV}(\varphi_1) \cup \mathsf{FV}(\varphi_2)$.
\end{itemize}
If $\varphi \in \mathsf{Horn}(\Sigma)$ and $\mathsf{FV}(\varphi) = \emptyset$, then we will refer to $\varphi$ as a (Horn) \emph{sentence}. If $\varphi \in \mathsf{Horn}(\Sigma)$, then we will write $\varphi(\vec{x})$ to mean that $\mathsf{FV}(\varphi) \subseteq \vec{x}$. \qed
}
\end{defn}

\begin{defn}
\label{sequent}
{\em Let $\Sigma$ be a signature. A \emph{Horn sequent} over $\Sigma$ is an expression of the form $\varphi \vdash^{\vec{x}} \psi$, where $\varphi, \psi \in \mathsf{Horn}(\Sigma)$ and $\mathsf{FV}(\varphi), \mathsf{FV}(\psi) \subseteq \vec{x}$ with $\vec{x}$ finite. A \emph{quasi-equational theory} $\T$ is a set of Horn sequents over a signature $\Sigma$. \qed
}
\end{defn}

One can now set up a deduction system of \emph{partial Horn logic} for quasi-equational theories, wherein certain Horn sequents are designated as logical axioms, and there are logical inference rules allowing one to deduce certain Horn sequents from other Horn sequents. We refer the reader to \cite{Horn} for a list of all the specific logical axioms and inference rules of partial Horn logic. The main distinguishing feature of this deduction system is that equality of terms is \emph{not} assumed to be reflexive, i.e. if $t(\vec{x})$ is a term over a given signature, then $\top \vdash^{\vec{x}} t(\vec{x}) = t(\vec{x})$ is \emph{not} a logical axiom of partial Horn logic, unless $t$ is a variable. In other words, if we abbreviate the equation $t = t$ by $t \downarrow$ (read: $t$ \emph{is defined}), then unless $t$ is a variable, the sequent $\top \vdash^{\vec{x}} t \downarrow$ is \emph{not} a logical axiom of partial Horn logic. 

If $\T$ is a quasi-equational theory over a signature $\Sigma$ and $\varphi \vdash^{\vec{x}} \psi$ is a Horn sequent over $\Sigma$, then we say that the sequent $\varphi \vdash^{\vec{x}} \psi$ is \emph{provable} in $\T$ if there is a finite sequence of Horn sequents whose last member is $\varphi \vdash^{\vec{x}} \psi$, and each member of the sequence is either a logical axiom of partial Horn logic, an axiom of $\T$, or is obtained from previous elements of the sequence by an inference rule of partial Horn logic. We also say that $\T$ \emph{proves} the sequent $\varphi \vdash^{\vec{x}} \psi$, or that this sequent is a \emph{theorem} of $\T$. If $\T$ proves a Horn sequent of the form $\top \vdash^{\vec{x}} \varphi$, then we usually write this as $\T \vdash^{\vec{x}} \varphi$.    

We now review the set-theoretic semantics of partial Horn logic. We recall that if $A$ and $B$ are any sets, then a \emph{partial function} $f : A \rightharpoondown B$ is a total function $f : \mathsf{dom}(f) \to B$ with $\mathsf{dom}(f) \subseteq A$.  

\begin{defn}
{\em Let $\Sigma$ be a signature. A (set-based) \emph{partial} $\Sigma$\emph{-structure} $M$ is given by a set $M_A$ for each $A \in \Sigma_\Sort$ and a \emph{partial function} $f^M : M_{A_1} \times \ldots \times M_{A_n} \rightharpoondown M_{A}$ for each function symbol $f : A_1 \times \ldots \times A_n \to A$ of $\Sigma_\Fun$. \qed
}
\end{defn}

\begin{defn}
{\em Let $\Sigma$ be a signature, and let $M$ and $N$ be partial $\Sigma$-structures. A $\Sigma$\emph{-morphism} $h : M \to N$ is a $\Sigma_{\mathsf{Sort}}$-indexed sequence of \emph{total} functions $h = (h_A : M_A \to N_A)_{A \in \Sigma_\Sort}$ such that for any function symbol $f : A_1 \times \ldots \times A_n \to A$ in $\Sigma_\Fun$ and any $(m_i)_i \in \prod_{1 \leq i \leq n} M_{A_i}$, if $(m_i)_i \in \mathsf{dom}\left(f^M\right)$, then $\left(h_{A_i}(m_i)\right)_i \in \mathsf{dom}\left(f^N\right)$ and $h_A\left(f^M(m_1, \ldots, m_n)\right) = f^N\left(h_{A_1}(m_1), \ldots, h_{A_n}(m_n)\right) \in N_A$. \qed
}
\end{defn}

It is easy to verify that the (componentwise) composition of $\Sigma$-morphisms is a $\Sigma$-morphism, and that the sequence of identity functions $(\mathsf{id}_A : M_A \to M_A)_A$ is a $\Sigma$-morphism $\mathsf{id}_M : M \to M$ that is an identity for composition. So we can form the category $\mathsf{P}\Sigma\mathsf{Str}$ of partial $\Sigma$-structures and $\Sigma$-morphisms. 

Before we can define the notion of a (set-based) model of a quasi-equational theory, we must first define the interpretations of terms and Horn formulas in partial structures.

\begin{defn}
{\em Let $\Sigma$ be a signature. Let $t(x_1, \ldots, x_k) : A$ be an element of $\mathsf{Term}(\Sigma)$ with free variables among $x_i : A_i$ for $1 \leq i \leq k$. Let $M$ be a partial $\Sigma$-structure. We define the \emph{partial} function \linebreak $t(x_1, \ldots, x_k)^M : \prod_{1 \leq i \leq k} M_{A_i} \rightharpoondown M_A$ by induction on the structure of $t$: 
\begin{itemize}
\item If $t \equiv x_i : A_i$ for some $1 \leq i \leq k$, then we set $t^M := \pi_i : \prod_{1 \leq i \leq k} M_{A_i} \to M_{A_i}$, the (total) projection onto the $i^{\mathsf{th}}$ factor.

\item If $t \equiv f(t_1, \ldots, t_m) : B$ for some function symbol $f : B_1 \times \ldots \times B_m \to B$ of $\Sigma$ with $t_j(x_1, \ldots, x_k) \in \mathsf{Term}(\Sigma)$ and $t_j : B_j$ for each $1 \leq j \leq m$, we first set
\begin{align*}
\mathsf{dom}\left(t^M\right)	&:= \left\{ \vec{a} \in \bigcap_{1 \leq j \leq m} \mathsf{dom}\left(t_j^M\right) : \left(t_j^M(\vec{a})\right)_{j} \in \mathsf{dom}\left(f^M\right)\right\},
\end{align*} 
and for any $\vec{a} \in \mathsf{dom}\left(t^M\right)$ we set $t^M(\vec{a}) := f^M\left(t_1^M(\vec{a}), \ldots, t_m^M(\vec{a})\right) \in M_B$, which defines $t^M = f(t_1, \ldots, t_m)^M : \prod_{1 \leq i \leq k} M_{A_i} \rightharpoondown M_B$. \qed 
\end{itemize}
}
\end{defn}

\begin{defn}
{\em Let $\Sigma$ be a signature, and let $\varphi(x_1, \ldots, x_k)$ be a Horn formula over $\Sigma$ with free variables among $x_i : A_i$ for $1 \leq i \leq k$. Let $M$ be a partial $\Sigma$-structure. We define $\varphi(x_1, \ldots, x_k)^M \subseteq \prod_{1 \leq i \leq k} M_{A_i}$ by induction on the structure of $\varphi$: 
\begin{itemize}
\item If $\varphi \equiv t_1 = t_2$ for some terms $t_1, t_2$ of the same sort, then \[ \varphi^M = (t_1 = t_2)^M := \left\{ \vec{a} \in \mathsf{dom}\left(t_1^M\right) \cap \mathsf{dom}\left(t_2^M\right) : t_1^M(\vec{a}) = t_2^M(\vec{a})\right\}. \]

\item If $\varphi \equiv \top$, then $\top^M := \prod_{1 \leq i \leq k} M_{A_i}$.

\item If $\varphi \equiv \varphi_1 \wedge \varphi_2$ for some $\varphi_1, \varphi_2 \in \Horn(\Sigma)$, then $\left(\varphi_1 \wedge \varphi_2\right)^M := \varphi_1^M \cap \varphi_2^M \subseteq \prod_{1 \leq i \leq k} M_{A_i}$. \qed
\end{itemize}
}
\end{defn}

\begin{defn}
{\em Let $\Sigma$ be a signature, let $M$ be a partial $\Sigma$-structure, and let $\varphi(\vec{x}), \psi(\vec{x})$ be Horn formulas over $\Sigma$. Then $M$ \emph{models} or \emph{satisfies} the Horn sequent $\varphi \vdash^{\vec{x}} \psi$ if $\varphi(\vec{x})^M \subseteq \psi(\vec{x})^M$.

Let $\T$ be a quasi-equational theory over a signature $\Sigma$, and let $M$ be a partial $\Sigma$-structure. Then $M$ is a \emph{model} of $\T$ if $M$ satisfies every axiom of $\T$. \qed }
\end{defn}

\noindent For a quasi-equational theory $\T$ over a signature $\Sigma$, we now let $\Tmod$ be the full subcategory of $\mathsf{P}\Sigma\mathsf{Str}$ on the models of $\T$. 

In order to sketch the details of the Initial Model Theorem for quasi-equational theories (see \cite[Theorem 22]{Horn}), we first require the following definitions.

\begin{defn}
{\em Let $\Sigma$ be a signature and $M$ a partial $\Sigma$-structure. For every sort $A$, let $\sim_A$ be an equivalence relation on $M_A$. Then the $\Sigma_{\mathsf{Sort}}$-indexed family of equivalence relations $(\sim_A)_A$ is a \emph{partial congruence} on $M$ if for every function symbol $f : A_1 \times \ldots \times A_n \to A$ in $\Sigma_\Fun$ and all $\vec{a}, \vec{b} \in \prod_{1 \leq i \leq n} M_{A_i}$, if 
$a_i \sim_{A_i} b_i$ for all $1 \leq i \leq n$, then $\vec{a} \in \mathsf{dom}\left(f^M\right)$ iff $\vec{b} \in \mathsf{dom}\left(f^M\right)$ and $\vec{a}, \vec{b} \in \mathsf{dom}\left(f^M\right) \Longrightarrow f^M(\vec{a}) \sim_A f^M\left(\vec{b}\right)$. \qed
}
\end{defn}

\noindent We now have the following definition:  

\begin{defn}
\label{partialcongruencestructure}
{\em Let $\Sigma$ be a signature and $M$ a partial $\Sigma$-structure. Let $\sim \ = (\sim_A)_A$ be a partial congruence on $M$. We define the \emph{partial quotient} $\Sigma$\emph{-structure} $M/{\sim}$ as follows: 

\begin{itemize}
\item For every sort $A \in \Sigma$, we set $(M/{\sim})_A := M_A/{\sim_A}$, the set of equivalence classes of $M_A$ modulo the equivalence relation $\sim_A$.

\item For any function symbol $f : A_1 \times \ldots \times A_n \to A$ we set \[ \mathsf{dom}\left(f^{M/{\sim}}\right) := \left\{ \left([a_i]\right)_{i} \in \prod_{1 \leq i \leq n} M_{A_i}/{\sim_{A_i}} : (a_i)_{i} \in \mathsf{dom}\left(f^M\right) \right\}. \] Then for any $\left([a_i]\right)_{i} \in \mathsf{dom}\left(f^{M/{\sim}}\right)$, we set $f^{M/{\sim}}\left([a_1], \ldots, [a_n]\right) := \left[f^M(a_1, \ldots, a_n)\right]$.
\end{itemize}
Because $\sim$ is a partial congruence on $M$, it easily follows that $M/{\sim}$ is a well-defined partial $\Sigma$-structure. \qed }
\end{defn}

\noindent We now sketch the details of the Initial Model Theorem from \cite{Horn} that we will need for our purposes. First, given a quasi-equational theory $\T$ over a signature $\Sigma$, we define a specific partial $\Sigma$-structure $M^{\T}$.
  
\begin{defn}
{\em Let $\Sigma$ be a signature. First, let $\mathsf{Term}^c(\Sigma) := \{ t \in \mathsf{Term}(\Sigma) : \mathsf{FV}(t) = \emptyset \}$ be the set of \emph{closed terms} of $\mathsf{Term}(\Sigma)$. For any $A \in \Sigma_{\mathsf{Sort}}$, let
\[ \mathsf{Term}^c(\Sigma)_A := \{ t \in \mathsf{Term}^c(\Sigma) : t \ \text{is of sort} \ A \} \] be the set of closed $\Sigma$-terms of sort $A$. 

Now let $\T$ be a quasi-equational theory over $\Sigma$. We define a partial $\Sigma$-structure $M^{\T}$ as follows:
\begin{itemize}
\item For any sort $A \in \Sigma$, we set $M^{\T}_A := \{ t \in \mathsf{Term}^c(\Sigma)_A : \T \vdash t \downarrow \}$. 

\item For any function symbol $f : A_1 \times \ldots \times A_n \to A$ of $\Sigma$, we set \[ \mathsf{dom}\left(f^{M^{\T}}\right) := \left\{\vec{t} \in \prod_{1 \leq i \leq n} M^{\T}_{A_i} : \T \vdash f\left(\vec{t}\right)\downarrow\right\}, \] and if $\vec{t} \in \mathsf{dom}\left(f^{M^{\T}}\right)$, we set $f^{M^{\T}}\left(\vec{t}\right) := f\left(\vec{t}\right) \in M^\T_A$. \qed
\end{itemize}
}
\end{defn}

\noindent Now we define a partial congruence $\sim^\T$ on $M^{\T}$. For any sort $A \in \Sigma$, we set \[ \sim^\T_A  \: := \left\{ (t_1, t_2) \in M^{\T}_A \times M^{\T}_A : \T \vdash t_1 = t_2\right\}. \]

\noindent Using the rules of partial Horn logic, it is then straightforward to verify that $\sim^\T$ is in fact a partial congruence on $M^{\T}$. We now make the following definition:

\begin{defn}
{\em Let $\T$ be a quasi-equational theory over a signature $\Sigma$, and let $M^{\T}$ be the partial $\Sigma$-structure and $\sim^\T$ the partial congruence on $M^{\T}$ just defined. Applying Definition \ref{partialcongruencestructure}, we then define the following partial $\Sigma$-structure: $\mathsf{Free}(\T) := M^{\T}/{\sim^\T}$. \qed }
\end{defn}

\noindent The following theorem is then proven in \cite[Theorem 22]{Horn}:  

\begin{theo}
\label{initialmodelthm}
Let $\T$ be a quasi-equational theory over a signature $\Sigma$. Then the partial $\Sigma$-structure $\mathsf{Free}(\T)$ is an \textbf{initial model} of $\T$, i.e. an initial object of the category $\Tmod$. \qed
\end{theo}

\begin{rmk}
\label{explicitdescriptionoffreemodel}
{\em
For concreteness, we give the explicit description of $\mathsf{Free}(\T)$ for a quasi-equational theory $\T$ over a signature $\Sigma$. 
\begin{itemize}
\item For any sort $A \in \Sigma$, we have \[ \mathsf{Free}(\T)_A := M^{\T}_A/{\sim^\T_A} = \left\{ [t] : t \in \mathsf{Term}^c(\Sigma)_A \text{ and } \T \vdash t\downarrow \right\}, \]
where $[t]$ is the $\sim^\T_A$-congruence class of $t \in M^\T_A$ (so for any $s, t \in M^\T_A$, we have $[s] = [t]$ iff $\T \vdash s = t$).  

\item If $f : A_1 \times \ldots \times A_n \to A$ is a function symbol of $\Sigma$, then \[ \mathsf{dom}\left(f^{\mathsf{Free}(\T)}\right) = \left\{ \left([t_i]\right)_{i} \in \prod_{1 \leq i \leq n} \mathsf{Free}(\T)_{A_i} : \T \vdash f(t_1, \ldots, t_n)\downarrow \right\}, \]
and for any $\left([t_i]\right)_{i} \in \mathsf{dom}\left(f^{\mathsf{Free}(\T)}\right)$, we have $f^{\mathsf{Free}(\T)}\left([t_1], \ldots, [t_n]\right) = \left[f(t_1, \ldots, t_n)\right]$. \qed
\end{itemize}
}
\end{rmk}

To review the main results of \cite[Section 2.2]{thesis} and \cite{FSCD}, characterizing the covariant isotropy group of $\Tmod$ for a quasi-equational theory $\T$ in logical terms, we first recall the following notions. 

If $M \in \Tmod$ for a quasi-equational theory $\T$ over a signature $\Sigma$, then $\Sigma(M)$ is the \emph{diagram signature} of $M$, which extends $\Sigma$ by adding a new constant symbol $c_a : C$ for any sort $C \in \Sigma_\Sort$ and $a \in M_C$. The quasi-equational theory $\T(M)$ over the signature $\Sigma(M)$ then extends $\T$ by adding axioms expressing that each new constant $c_a$ is defined, and that the function symbols of $\T$ interact with these constants appropriately (for explicit details, see \cite[Definition 2.2.3]{thesis}). If $n \geq 1$ and $A_i \in \Sigma_\Sort$ for each $1 \leq i \leq n$, then $\Sigma(M, \x_1, \ldots, \x_n)$ is the signature that extends $\Sigma(M)$ by adding new pairwise distinct constant symbols $\x_i : A_i$ for all $1 \leq i \leq n$. The quasi-equational theory $\T(M, \x_1, \ldots, \x_n)$ over the signature $\Sigma(M, \x_1, \ldots, \x_n)$ then extends $\T(M)$ by adding axioms expressing that $\x_i$ is defined for each $1 \leq i \leq n$. Finally, if $C \in \Sigma_\Sort$, then $M\la \x_C \ra$ is defined to be the ($\Sigma$-reduct of) the initial model of $\T(M, \x_C)$, which therefore has the following explicit description (cf. Remark \ref{explicitdescriptionoffreemodel}): 

\begin{itemize}
\item For any $B \in \Sigma_\Sort$, \[ M\la \mathsf{x}_C \ra_B = \left\{ [t] : t \in \mathsf{Term}^c(\Sigma(M, \mathsf{x}_C))_B \wedge \T(M, \mathsf{x}_C) \vdash t\downarrow \right\}. \] 

\item If $f : A_1 \times \ldots \times A_n \to A$ is a function symbol of $\Sigma$, then \[ \mathsf{dom}\left(f^{M\la \x_C\ra}\right) = \left\{ \left([t_i]\right)_{i} \in \prod_{1 \leq i \leq n} M\la \x_C \ra_{A_i} : \T(M, \x_C) \vdash f(t_1, \ldots, t_n)\downarrow \right\}, \]
and for any $\left([t_i]\right)_{i} \in \mathsf{dom}\left(f^{M\la \x_C\ra}\right)$ we have $f^{M\la \x_C \ra}\left([t_1], \ldots, [t_n]\right) = \left[f(t_1, \ldots, t_n)\right]$.
\end{itemize}
   
\noindent We now recall \cite[Definition 2.2.47]{thesis} and \cite[Definition 6]{FSCD}. The notion of syntactic substitution used in the following definition is the standard/expected one; see \cite[Remark 2.2.21]{thesis} for an explicit definition.

\begin{defn}
\label{commutesgenericallydefn}
{\em Let $\T$ be a quasi-equational theory over a signature $\Sigma$, and let $M \in \Tmod$ and $([s_C])_C \in \prod_{C \in \Sigma_\Sort} M\la \mathsf{x}_C \ra_C$. 

\begin{itemize}

\item If $f : A_1 \times \ldots \times A_n \to A$ is a function symbol of $\Sigma$, then $([s_C])_C$ \emph{commutes generically with} $f$ if the Horn sequent \[ f(\x_1, \ldots, \x_n) \downarrow \: \vdash \ s_A[f(\x_1, \ldots, \x_n)/\mathsf{x}_A] = f\left(s_{A_1}[\mathsf{x}_1/\mathsf{x}_{A_1}], \ldots, s_{A_n}[\mathsf{x}_n/\mathsf{x}_{A_n}]\right) \] is provable in $\T\left(M, \x_1, \ldots, \x_n\right)$. 

\item We say that $([s_C])_C$ is \emph{(substitutionally) invertible} if for every $B \in \Sigma_\Sort$ there is some $\left[s_B^{-1}\right] \in M\la \mathsf{x}_B \ra_B$ with \[ \left[s_B\left[s_B^{-1}/\mathsf{x}_B\right]\right] = [\mathsf{x}_B] = \left[s_B^{-1}\left[s_B/\mathsf{x}_B\right]\right] \in M\la \mathsf{x}_B \ra_B, \] i.e. with $\T(M, \mathsf{x}_B) \vdash s_B\left[s_B^{-1}/\mathsf{x}_B\right] = \mathsf{x}_B = s_B^{-1}[s_B/\mathsf{x}_B]$.

\item We say that $([s_C])_C$ \emph{reflects definedness} if for every function symbol $f : A_1 \times \ldots \times A_n \to A$ in $\Sigma$, the sequent  
\[ f\left(s_{A_1}[\mathsf{x}_1/\mathsf{x}_{A_1}], \ldots, s_{A_n}[\mathsf{x}_n/\mathsf{x}_{A_n}]\right) \downarrow \ \ \vdash f(\x_1, \ldots \x_n) \downarrow \] is provable in $\T\left(M, \x_1, \ldots, \x_n\right)$. \qed
\end{itemize}
}
\end{defn}

\noindent As in \cite[Definition 2.2.36]{thesis}, we then have a functor $G_\T : \Tmod \to \Group$, with $G_\T(M)$ for $M \in \Tmod$ being the group of all elements $([s_C])_C \in \prod_{C \in \Sigma} M\la \mathsf{x}_C \ra_C$ that are substitutionally invertible and commute generically with and reflect definedness of every function symbol of $\Sigma$. The unit of this group is the element $([\x_C])_C$, the inverse of any element is obtained via the substitutional invertibility of each of its components (as in Definition \ref{commutesgenericallydefn}), and if $([s_C])_C, ([t_C])_C \in G_\T(M)$, then their product is obtained via substitution as $([s_C])_C \cdot ([t_C])_C := \left(\left[s_C\left[t_C/\x_C\right]\right]\right)_C$. For more details, see \cite[Propositions 2.2.35, 2.2.38]{thesis}. From \cite[Theorems 2.2.41, 2.2.53]{thesis} and \cite[Theorem 7]{FSCD} we then conclude that if $\T$ is a quasi-equational theory, then \[ \Z_{\Tmod} \cong G_\T : \Tmod \to \Group. \] (In what follows, we will generally write $\Z_\T$ in place of $\Z_{\Tmod}$.) In other words, the covariant isotropy group of $M \in \Tmod$, i.e. its group of extended inner automorphisms, is isomorphic to the group of all elements of $([s_C])_C \in \prod_{C \in \Sigma} M\la \mathsf{x}_C \ra_C$ that are substitutionally invertible and commute generically with and reflect definedness of all operations of $\T$ (naturally in $M$). 

\section{Logical characterization}
\label{logical}

For the remainder of this section, we fix a quasi-equational theory $\T$ over a signature $\Sigma$, as well as a small index category $\J$. At a certain point (see Proposition \ref{betaissurjective}) we will need to impose two modest assumptions on $\T$, but for the time being we can assume that $\T$ is arbitrary. We first show that the functor category $\Tmod^\J$ can be axiomatized as the category of models of a quasi-equational theory $\TJ$, and then we explicitly characterize $G_{\TJ} : \TJmod \to \Group$, which will yield an explicit characterization of the covariant isotropy group $\Z_{\Tmod^\J} : \Tmod^\J \to \Group$ in Section \ref{categorical}. 

We first define the signature $\Sigma^\J$ for the desired quasi-equational theory $\TJ$. Since $\J$ is small, we know that its classes $\ob\J$ and $\mor\J$ of objects and morphisms are both sets. 

\begin{defn}
\label{functorsignature}
{\em We define the signature $\Sigma^\J$ as follows: 
\begin{itemize}
\item If $i \in \ob\J$ and $A \in \Sigma_{\mathsf{Sort}}$, let $A^i \notin \Sigma_\Sort$ be a new sort. Then we set
\[ \Sigma^\J_{\mathsf{Sort}} := \left\{ A^i : i \in \ob\J, A \in \Sigma_{\mathsf{Sort}}\right\}. \]
\item For any morphism $f : i \to j$ in $\J$ and $A \in \Sigma_{\mathsf{Sort}}$, let $\alpha_f^A : A^i \to A^j$ be a new unary function symbol $\notin \Sigma_\Fun$. For any $i \in \ob\J$ and function symbol $g : A_1 \times \ldots \times A_n \to A$ in $\Sigma_\Fun$, let $g^i : A^i_1 \times \ldots \times A^i_n \to A^i$ be a new function symbol $\notin \Sigma_\Fun$. Then we set
\[ \Sigma^\J_{\mathsf{Fun}} := \left\{\alpha_f^A : f \in \mor\J, A \in \Sigma_{\mathsf{Sort}}\right\} \ \bigcup \ \left\{g^i : g \in \Sigma_{\mathsf{Fun}}, i \in \ob\J\right\}. \] \qed
\end{itemize}
}
\end{defn}

\noindent Given a partial $\Sigma^\J$-structure, we now show how to derive component $\Sigma$-structures from it, indexed by the objects of $\J$. 

\begin{defn}
\label{componentstructures}
{\em Let $M$ be a partial $\Sigma^\J$-structure and let $i \in \ob\J$. We define a partial $\Sigma$-structure $M^i$ as follows: for any $A \in \Sigma_{\mathsf{Sort}}$, we set $M^i_A := M_{A^i}$, and for any function symbol $g : A_1 \times \ldots \times A_n \to A$ of $\Sigma$, we set $g^{M^i} := \left(g^i\right)^M : M_{A^i_1} \times \ldots \times M_{A^i_n} = M^i_{A_1} \times \ldots \times M^i_{A_n} \rightharpoondown M^i_A = M_{A^i}$. \qed
}
\end{defn}

\noindent From a morphism of $\Sigma^\J$-structures we can also obtain morphisms of the component $\Sigma$-structures:

\begin{defn}
\label{componenthoms}
{\em Let $h : M \to N$ be a $\SigmaJ$-morphism and let $i \in \ob\J$. Then we have a $\Sigma$-morphism $h^i : M^i \to N^i$ given by $h^i_A := h_{A^i} : M^i_A = M_{A^i} \to N_{A^i} = N^i_A$ for every $A \in \Sigma_{\mathsf{Sort}}$. \qed } 
\end{defn}

\noindent For any $i \in \ob\J$, we now define an `inclusion' signature morphism $\rho^i : \Sigma \to \Sigma^\J$. Recall from \cite[Section 5]{Horn} that a \emph{signature morphism} $\rho : \Sigma_1 \to \Sigma_2$ from a signature $\Sigma_1$ to a signature $\Sigma_2$ assigns to each sort $A$ of $\Sigma_1$ a sort $\rho(A)$ of $\Sigma_2$ and to each function symbol $g : A_1 \times \ldots \times A_n \to A$ of $\Sigma_1$ a function symbol $\rho(g) : \rho(A_1) \times \ldots \times \rho(A_n) \to \rho(A)$ of $\Sigma_2$. 

\begin{defn}
\label{objectsignaturemorphisms}
{\em For any $i \in \ob\J$, we define a signature morphism $\rho^i : \Sigma \to \SigmaJ$ as follows: for any $A \in \Sigma_{\mathsf{Sort}}$, we set $\rho^i(A) := A^i \in \SigmaJ_{\mathsf{Sort}}$, and for any function symbol $g : A_1 \times \ldots \times A_n \to A$ in $\Sigma_{\mathsf{Fun}}$, we set $\rho^i(g) := g^i : A^i_1 \times \ldots \times A^i_n \to A^i$. \qed
}
\end{defn}

\noindent We can now define the quasi-equational theory $\TJ$ that will axiomatize $\PTmod^\J$.

\begin{defn}
\label{thefunctortheory}
{\em We define $\TJ$ to be the quasi-equational theory over the signature $\SigmaJ$ whose axioms are the following sequents:

\begin{enumerate}

\item For any $f : i \to j$ in $\mor\J$ and $A \in \Sigma_{\mathsf{Sort}}$, the axiom $\top \vdash^{x : A^i} \alpha_f^A(x) \downarrow$. \label{alphasaretotal}

\item For any $i \in \ob\J$ and $A \in \Sigma_{\mathsf{Sort}}$, the axiom $\top \vdash^{x : A^i} \alpha_{\mathsf{id}_i}^A(x) = x$. \label{alphaidworks}

\item For any $f : i \to j$ and $g : j \to k$ in $\mor\J$ and $A \in \Sigma_{\mathsf{Sort}}$, the axiom $\top \vdash^{x : A^i} \alpha_g^A\left(\alpha_f^A(x)\right) = \alpha_{g \circ f}^A(x)$. \label{alphacompworks}

\item For any $f : i \to j$ in $\mor\J$ and $g : A_1 \times \ldots \times A_n \to A$ in $\Sigma_{\mathsf{Fun}}$, the axiom
\[ g^i(\mathsf{x}_1, \ldots, \mathsf{x}_n) \downarrow \ \ \vdash^{\mathsf{x}_1 : A^i_1, \ldots, \mathsf{x}_n : A^i_n} \alpha_f^A\left(g^i(\mathsf{x}_1, \ldots, \mathsf{x}_n)\right) = g^j\left(\alpha_f^{A_1}(\mathsf{x}_1), \ldots, \alpha_f^{A_n}(\mathsf{x}_n)\right). \] \label{alphasarehoms}
\item For any $i \in \ob\J$ and any axiom $\varphi \vdash^{\vec{x}} \psi$ of $\T$, the axiom $\rho^i(\varphi) \vdash^{\rho^i(\vec{x})} \rho^i(\psi)$. \label{axiomsofT} \qed
\end{enumerate}
}
\end{defn}

\begin{rmk}
\label{objectsigmorphismsaretheorymorphisms}
{\em
Recall from \cite[Section 5]{Horn} that if $\T_1$ and $\T_2$ are quasi-equational theories over respective signatures $\Sigma_1$ and $\Sigma_2$, then a signature morphism $\rho : \Sigma_1 \to \Sigma_2$ is a \emph{theory morphism} from $\T_1$ to $\T_2$ if the $\rho$-translation $\rho(\varphi) \vdash^{\rho(\vec{x})} \rho(\psi)$ of any axiom $\varphi \vdash^{\vec{x}} \psi$ of $\T_1$ is provable in $\T_2$, which then entails that $\rho$ preserves provability of sequents (see \cite[Lemma 2.2.16]{thesis}). A first easy property of $\TJ$ is now that for any $i \in \ob\J$, the signature morphism $\rho^i : \Sigma \to \SigmaJ$ of Definition \ref{objectsignaturemorphisms} is also a \emph{theory} morphism $\T \to \TJ$, because $\TJ$ includes the axioms in Definition \ref{thefunctortheory}.\ref{axiomsofT}. \qed
}
\end{rmk}

\noindent To begin studying the models of $\TJ$, we first make the following easy observation:

\begin{lem}
\label{componentsaremodelsofT}
If $M$ is a partial $\SigmaJ$-structure with $M \models \TJ$, then for any $i \in \ob\J$, the partial $\Sigma$-structure $M^i$ of Definition \ref{componentstructures} is a model of $\T$. 
\end{lem}

\begin{proof}
Since $\rho^i : \T \to \TJ$ is a theory morphism by Remark \ref{objectsigmorphismsaretheorymorphisms}, it follows by \cite[Proposition 28]{Horn} that $U^i(M)$ is a model of $\T$, where $U^i : \PTJmod$ $\to \PTmod$ is the forgetful functor induced by the signature morphism $\rho^i$. However, it is trivial to observe that $U^i(M) = M^i$, so that $M^i$ is indeed a model of $\T$.
\end{proof}

\noindent We now have: 

\begin{prop}
\label{modelsarefunctors}
There is an isomorphism of categories $\PTJmod \cong \PTmod^\J$.
\end{prop}

\begin{proof}
We sketch the bijection between the objects of $\TJmod$ and $\Tmod^\J$ and refer the reader to \cite[Proposition 5.1.8]{thesis} for the remaining straightforward details. Given $M \in \TJmod$, we must define a corresponding functor $F^M : \J \to \Tmod$. For any $i \in \ob\J$, we let $F^M(i)$ be the $\T$-model $M^i$ of Lemma \ref{componentsaremodelsofT}. For any morphism $f : i \to j$ of $\J$, we define the $\Sigma$-morphism $F^M(f) : M^i \to M^j$ as follows: for any $A \in \Sigma_\Sort$, we define $F^M(f)_A : M^i_A = M_{A^i} \to M_{A^j} \to M^j_A$ as $F^M(f)_A := \left(\alpha_f^A\right)^M$. Because of Axiom \ref{thefunctortheory}.\ref{alphasaretotal}, it follows that each function $F^M(f)_A = \left(\alpha_f^A\right)^M$ is total (as needed). The functoriality of $F^M$ follows from Axioms \ref{thefunctortheory}.\ref{alphaidworks} and \ref{thefunctortheory}.\ref{alphacompworks}, and the fact that $F^M(f) : M^i \to M^j$ is a $\Sigma$-morphism follows from Axiom \ref{thefunctortheory}.\ref{alphasarehoms}. This proves that $F^M : \J \to \Tmod$ is a well-defined functor. Conversely, starting from a functor $F : \J \to \Tmod$, we define a model $M^F \in \TJmod$ as follows: for any $A \in \Sigma_\Sort$ and $i \in \ob\J$, we set $M^F_{A^i} := F(i)_A$, for any morphism $f : i \to j$ in $\J$, we set $\left(\alpha_f^A\right)^{M^F} := F(f)_A$, and for any function symbol $g : A_1 \times \ldots \times A_n \to A$ of $\Sigma$, we set $\left(g^i\right)^{M^F} := g^{F(i)}$. The functoriality of $F$ guarantees that $M^F$ satisfies Axioms \ref{thefunctortheory}.\ref{alphasaretotal}, \ref{thefunctortheory}.\ref{alphaidworks}, and \ref{thefunctortheory}.\ref{alphacompworks}, the fact that each $F(f)$ is a $\Sigma$-morphism guarantees that $M^F$ satisfies Axiom \ref{thefunctortheory}.\ref{alphasarehoms}, and the fact that each $F(i)$ is a $\T$-model guarantees that $M^F$ satisfies Axiom \ref{thefunctortheory}.\ref{axiomsofT}. So $M^F$ is indeed a $\TJ$-model, and it is now straightforward to observe that the assignments $M \mapsto F^M$ and $F \mapsto M^F$ are mutually inverse.    
\end{proof}

\noindent Before we can start to characterize the covariant isotropy group of $\TJ$, we first require the following purely group-theoretic fact, whose proof is a routine verification. 

\begin{lem}
\label{grouptheorylemma}
Let $F : \J \to \mathsf{Group}$ be a functor, and consider the product group $\prod_{i \in \ob\J} F(i)$. Then
\[ \left(\prod_{i \in \ob\J} F(i)\right)^F := \left\{ (g_i)_{i} \in \prod_{i \in \ob\J} F(i) \ \colon F(f)(g_j) = g_k \ \forall f : j \to k \in \mor\J \right\} \] is a subgroup of $\prod_{i \in \ob\J} F(i)$. Furthermore, this assignment is the object part of a functor \linebreak $\left(\prod_{i \in \ob\J} (-)(i)\right)^{(-)} = \mathsf{lim} : \mathsf{Group}^\J \to \mathsf{Group}$. \qed
\end{lem}

\noindent We can now begin to characterize the covariant isotropy groups of models of $\TJ$. If $M \in \PTJmod$, then by the proof of Proposition \ref{modelsarefunctors}, there is a corresponding functor $F^M : \J \to \PTmod$. If $G_{\T} : \PTmod \to \mathsf{Group}$ is the functor naturally isomorphic to the covariant isotropy group $\Z_{\T} : \Tmod \to \Group$ by Section \ref{background}, then we obtain the composite functor $G_{\T} \circ F^M : \J \to \mathsf{Group}$ with \[ \left(G_{\T} \circ F^M\right)(i) = G_{\T}\left(F^M(i)\right) = G_{\T}\left(M^i\right) \] for every $i \in \ob\J$. By Lemma \ref{grouptheorylemma}, it then follows that $\left(\prod_i G_{\T}\left(M^i\right)\right)^{G_{\T} \circ F^M}$ is a subgroup of $\prod_i G_{\T}\left(M^i\right)$, and hence in particular is a group. Let us denote this subgroup with the cleaner notation $\left(\prod_i G_{\T}\left(M^i\right)\right)^\J$. 

Next, we will need to define a certain group $\mathsf{Aut}(\mathsf{Id}_\J)^M$, where $\Aut(\Id_\J)$ is the group of natural automorphisms of the identity functor $\Id_\J : \J \to \J$. Its definition is somewhat subtle and unintuitive, so we ask the reader to bear with us until after we have defined it, at which point we will try to give some explanation for the technicalities in its definition.

For any $i \in \ob\J$ and $B \in \Sigma_{\mathsf{Sort}}$, we say that the diagram theory $\T\left(M^i\right)$ of the $\T$-model $M^i$ is \emph{trivial} for the sort $B$ if $\T\left(M^i\right) \vdash^{y, y'} y = y'$ for distinct variables $y, y' : B$. Otherwise, we say that $\T\left(M^i\right)$ is \emph{non-trivial} for the sort $B$. To say that $\T\left(M^i\right)$ is trivial for the sort $B$ is equivalent to saying (by \cite[Lemma 3.1.2]{thesis}) that for any $\T$-model $N$ for which there is a $\Sigma$-morphism $M^i \to N$, the carrier set $N_B$ has at most one element. For any $B \in \Sigma_{\mathsf{Sort}}$, we let $\J^M_B$ be the full subcategory of $\J$ on those objects $i \in \ob\J$ for which $\T\left(M^i\right)$ is non-trivial for the sort $B$. We then let $\mathsf{Aut}\left(\mathsf{Id}_{\J^M_B}\right)$ be the group of natural automorphisms of the identity functor $\mathsf{Id}_{\J^M_B} : \J^M_B \to \J^M_B$.

We will need to consider a certain subgroup of $\prod_{B \in \Sigma_{\mathsf{Sort}}} \mathsf{Aut}\left(\mathsf{Id}_{\J^M_B}\right)$, which we will call $\mathsf{Aut}(\mathsf{Id}_\J)^M$. To define this subgroup, we first require the following definition:

\begin{defn}
\label{degenerate}
{\em Let $M \in \PTJmod$, let $g : A_1 \times \ldots \times A_n \to A$ be a function symbol of $\Sigma$ with $n \geq 1$, and let $i \in \ob\J$. Then for any $1 \leq m \leq n$, we say that $g^{M^i}$ is \emph{degenerate in position} $m$ if 
\[ \T\left(M^i\right) \vdash^{y_1, \ldots, y_n, z_m} g(y_1, \ldots, y_n) = g(y_1, \ldots, y_n)[z_m/y_m], \] where $y_1, \ldots, y_n, z_m$ are pairwise distinct variables of the appropriate sorts. Otherwise, if $\T\left(M^i\right)$ does \emph{not} prove the above equation, we say that $g^{M^i}$ is \emph{non-degenerate in position} $m$.
\qed }
\end{defn}

\noindent Thus, to say that $g^{M^i}$ is degenerate in position $1 \leq m \leq n$ is equivalent to saying (again by \cite[Lemma 3.1.2]{thesis}) that for any $\T$-model $N$ for which there is a $\Sigma$-morphism $M^i \to N$ and any elements $a_1 \in N_{A_1}, \ldots, a_m, b_m \in N_{A_m}, \ldots, a_n \in N_{A_n}$, we have $g^N(a_1, \ldots, a_m, \ldots, a_n) = g^N(a_1, \ldots, b_m, \ldots, a_n)$ (i.e. the value of $g^N$ does not change when the $m$th coordinate of an input $n$-tuple changes). We can now define: 

\begin{defn}
\label{automorphismsubgroup}
{\em Let $M \in \PTJmod$. We denote an element of $\prod_{B \in \Sigma_{\mathsf{Sort}}} \mathsf{Aut}\left(\mathsf{Id}_{\J^M_B}\right)$ by $\psi = (\psi_B)_{B \in \Sigma}$, so that each $\psi_B$ is a natural automorphism of $\mathsf{Id}_{\J^M_B}$, with components $\psi_B(i) : i \xrightarrow{\sim} i$ for each $i \in \ob\J^M_B$.

We define \[ \mathsf{Aut}(\mathsf{Id}_\J)^M \subseteq \prod_{B \in \Sigma_{\mathsf{Sort}}} \mathsf{Aut}\left(\mathsf{Id}_{\J^M_B}\right) \] to consist of exactly those elements $\psi \in \prod_{B \in \Sigma_{\mathsf{Sort}}} \mathsf{Aut}\left(\mathsf{Id}_{\J^M_B}\right)$ with the following property: if \linebreak $g : A_1 \times \ldots \times A_n \to A$ is any function symbol of $\Sigma$ with $n \geq 1$, then for any $i \in \ob\J$ and $1 \leq m \leq n$ for which $g^{M^i}$ is non-degenerate in position $m$, we have $\psi_{A_m}(i) = \psi_A(i) : i \xrightarrow{\sim} i$. 

This property is well-defined, in the sense that if $g^{M^i}$ is non-degenerate in position $m$, then it easily follows that $\T\left(M^i\right)$ must be non-trivial for the sorts $A$ and $A_m$, so that $i$ must be an object of both $\J^M_A$ and $\J^M_{A_m}$, and hence $\psi_A(i)$ and $\psi_{A_m}(i)$ are both well-defined morphisms of $\J$. It is then trivial to verify that $\mathsf{Aut}(\mathsf{Id}_\J)^M$ is indeed a \emph{subgroup} of $\prod_{B \in \Sigma_{\mathsf{Sort}}} \mathsf{Aut}\left(\mathsf{Id}_{\J^M_B}\right)$, and hence is a group.\qed }
\end{defn}

Our ultimate goal in this section is now to show for any quasi-equational theory $\T$ (satisfying two mild conditions, see Proposition \ref{betaissurjective}), any small index category $\J$, and any $M \in \PTJmod$ that
\[ G_{\TJ}(M) \cong \left(\prod_{i \in \J} G_{\T}\left(M^i\right)\right)^\J \times \mathsf{Aut}(\mathsf{Id}_\J)^M, \] naturally in $M$. Specifically, we will construct a group isomorphism
\[ \beta_M : \left(\prod_{i \in \J} G_{\T}\left(M^i\right)\right)^\J \times \mathsf{Aut}(\mathsf{Id}_\J)^M \xrightarrow{\sim} G_{\TJ}(M) \] for each $M \in \PTJmod$. 

As promised, let us now attempt to give some explanation of the technicalities involved in the definition of $\mathsf{Aut}(\mathsf{Id}_\J)^M$. First, let us discuss why for each sort $B \in \Sigma_\Sort$ we needed to consider the full subcategory $\J_B^M$ of $\J$ on those objects $i \in \ob\J$ for which $\T\left(M^i\right)$ is non-trivial for the sort $B$, rather than just the whole category $\J$. Essentially, the reason is that if we considered $\J$ rather than $\J_B^M$, then the group homomorphism $\beta_M$ that we will define in Proposition \ref{grouphomexists} will \emph{not} be injective in general. Indeed, we show in detail on \cite[Page 151]{thesis} that if $\T$ is the single-sorted algebraic theory of commutative unital rings, $\J$ is any one-object category for which $\Aut(\Id_\J)$ is non-trivial, and $F : \J \to \Tmod = \mathsf{CRing}$ is the constant functor on the zero ring with corresponding $\TJ$-model $M$, then $\beta_M$ will \emph{not} be injective if we consider $\J$ rather than $\J_X^M$ (where $X$ is the unique sort of $\T$).  

This will hopefully help to convince the reader that we need to define $\Aut(\Id_\J)^M$ to be a subgroup of $\prod_{\B \in \Sigma_\Sort} \Aut\left(\Id_{\J^M_B}\right)$ rather than $\prod_{\B \in \Sigma_\Sort} \Aut\left(\Id_{\J}\right)$. Now let us try to motivate why we cannot just define $\Aut(\Id_\J)^M$ to be the full group $\prod_{\B \in \Sigma_\Sort} \Aut\left(\Id_{\J^M_B}\right)$ (when $\T$ is \emph{multi}-sorted). A first vague intuition is that if $\psi = (\psi_B)_{B \in \Sigma} \in \prod_{\B \in \Sigma_\Sort} \Aut\left(\Id_{\J^M_B}\right)$, then we need the distinct $\psi_B$'s to `interact' properly, if there are (non-degenerate) function symbols in $\Sigma_\Fun$ that `connect' distinct sorts. Indeed, we show in detail on \cite[Page 152]{thesis} that if $\J$ is the one-object category corresponding to the group $\mathbb{Z}_2$ and $\T$ is the quasi-equational theory with two sorts $X$ and $Y$ and one function symbol $f : X \to Y$ and the single axiom asserting that $f$ is always defined, then there is a model $M$ of $\TJ$ for which the group homomorphism $\beta_M$ defined in Proposition \ref{grouphomexists} will \emph{not} land in $G_{\TJ}(M)$, if we do not define $\Aut(\Id_\J)^M$ as we do in Definition \ref{automorphismsubgroup}.  

Now, towards constructing the group homomorphisms $\beta_M$ in Proposition \ref{grouphomexists}, we require the following technical definitions and lemmas.

\begin{defn}
\label{sigmorphismofM^i}
{\em Let $M \in \PTJmod$, $i \in \ob\J$, and $C \in \Sigma_{\mathsf{Sort}}$. We define a signature morphism \[ \rho_{M^i}^C : \Sigma\left(M^i, \mathsf{x}_C\right) \to \SigmaJ\left(M, \mathsf{x}_{C^i}\right) \] as follows (where $\mathsf{x}_C \notin \Sigma\left(M^i\right)$ and $\mathsf{x}_{C^i} \notin \SigmaJ(M)$ are new constants of sorts $C$ and $C^i$, respectively):

\begin{itemize}

\item On $\Sigma \subseteq \Sigma\left(M^i, \mathsf{x}_C\right)$, we stipulate that $\rho_{M^i}^C$ agrees with $\rho^i : \Sigma \to \SigmaJ$ (see Definition \ref{objectsignaturemorphisms}). 

\item If $s \in M^i_A = M_{A^i}$ for some $A \in \Sigma_{\mathsf{Sort}}$, then we set $\rho_{M^i}^C\left(c_{A, s}^{M^i}\right) := c_{A^i, s}^M \in \SigmaJ\left(M, \mathsf{x}_{C^i}\right)$. 

\item We set $\rho_{M^i}^C(\mathsf{x}_C) := \mathsf{x}_{C^i} \in \SigmaJ\left(M, \mathsf{x}_{C^i}\right)$. \qed
\end{itemize}
}
\end{defn}

\noindent We now have the following lemma, whose easy proof may be found in \cite[Lemma 5.1.13]{thesis}:

\begin{lem}
\label{theorymorphismofM^i}
For any $M \in \PTJmod$, $i \in \ob\J$, and $C \in \Sigma_{\mathsf{Sort}}$, the signature morphism $\rho_{M^i}^C$ is a theory morphism $\rho_{M^i}^C : \T\left(M^i, \mathsf{x}_C\right) \to \TJ\left(M, \mathsf{x}_{C^i}\right)$. \qed
\end{lem}

\noindent The proof of the next lemma is immediate from the definitions: 

\begin{lem}
\label{theorymorphismofM^isubst}
Let $M \in \PTJmod$, $i \in \ob\J$, and $C \in \Sigma_{\mathsf{Sort}}$. For any $u, v \in \mathsf{Term}^c\left(\Sigma\left(M^i, \mathsf{x}_C\right)\right)$ with $v : C$, we have $\rho_{M^i}^C\left(u\left[v/\mathsf{x}_C\right]\right) \equiv \rho_{M^i}^C(u)\left[\rho_{M^i}^C(v)/\mathsf{x}_{C^i}\right]$. \qed
\end{lem}

\noindent We will require the following signature morphisms indexed by the elements of $\mor\J$:

\begin{defn}
\label{sigmorphismsigma}
{\em Let $M \in \PTJmod$, $f : i \to j \in \mor\J$, and $C \in \Sigma_{\mathsf{Sort}}$. We define a signature morphism
\[ \sigma_{f}^C : \SigmaJ\left(M, \mathsf{x}_{C^j}\right) \to \SigmaJ\left(M, \mathsf{x}_{C^i}\right) \] as follows: on $\SigmaJ(M)$ we define $\sigma_{f}^C$ to be the inclusion into $\SigmaJ\left(M, \mathsf{x}_{C^i}\right)$, and we set $\sigma_{f}^C\left(\mathsf{x}_{C^j}\right) := \alpha_f^C\left(\mathsf{x}_{C^i}\right) : C^j$. \qed
}
\end{defn}

\noindent Since $\TJ\left(M, \mathsf{x}_{C^i}\right) \vdash \alpha_f^C(\mathsf{x}_{C^i}) \downarrow$, we then easily obtain: 
\begin{lem}
\label{theorymorphismsigma}
For any $M \in \PTJmod$, $f : i \to j \in \mor\J$, and $C \in \Sigma_{\mathsf{Sort}}$, the signature morphism $\sigma_{f}^C$ is a theory morphism $\sigma_{f}^C : \TJ\left(M, \mathsf{x}_{C^j}\right) \to \TJ\left(M, \mathsf{x}_{C^i}\right)$. \qed
\end{lem}

\noindent If $f : i \to j$ is a morphism in $\J$, let us write $f^M := F^M(f) : M^i \to M^j$ (see the proof of Proposition \ref{modelsarefunctors}). Then we have \[ f^M = F^M(f) = \left(\left(\alpha_f^A\right)^M : M^i_A \to M^j_A\right)_{A \in \Sigma}. \] Recall that for any $C \in \Sigma_{\mathsf{Sort}}$, the $\Sigma$-morphism $f^M : M^i \to M^j$ induces a canonical signature morphism $\rho_{f^M}^C : \Sigma\left(M^i, \mathsf{x}_C\right) \to \Sigma\left(M^j, \mathsf{x}_C\right)$ by \cite[Definition 2.2.17]{thesis} which is also a theory morphism $\T\left(M^i, \mathsf{x}_C\right) \to \T\left(M^j, \mathsf{x}_C\right)$ by \cite[Lemma 2.2.18]{thesis}.

\begin{defn}
\label{sigmorphismtau}
{\em For any $M \in \PTJmod$, any morphism $f : i \to j$ in $\J$, and any $C \in \Sigma_{\mathsf{Sort}}$, we define a signature morphism $\tau_{f}^C : \Sigma\left(M^i, \mathsf{x}_C\right) \to \SigmaJ\left(M, \mathsf{x}_{C^i}\right)$ as the composite
\[ \Sigma\left(M^i, \mathsf{x}_C\right) \xrightarrow{\rho_{f^M}^C} \Sigma(M^j, \mathsf{x}_C) \xrightarrow{\rho_{M^j}^C} \SigmaJ(M, \mathsf{x}_{C^j}) \xrightarrow{\sigma_{f}^C} \SigmaJ(M, \mathsf{x}_{C^i}). \]
Explicitly, $\tau_{f}^C$ is defined as follows:

\begin{itemize}

\item When restricted to $\Sigma \subseteq \Sigma\left(M^i, \mathsf{x}_C\right)$, $\tau_{f}^C$ agrees with $\rho^j : \Sigma \to \SigmaJ$. 

\item For any $s \in M^i_A = M_{A^i}$ (for any $A \in \Sigma_{\mathsf{Sort}}$), we have $\tau_{f}^C\left(c_{A, s}^{M^i}\right) \equiv c_{A^j, \left(\alpha_f^A\right)^M(s)}^M$. 

\item We have $\tau_{f}^C(\mathsf{x}_C) \equiv \alpha_f^C(\mathsf{x}_{C^i})$. \qed
\end{itemize}
}
\end{defn}

\noindent We will then need the following technical lemma about the signature morphism $\tau_f^C$, whose proof may be found in \cite[Lemma 5.1.19]{thesis}:

\begin{lem}
\label{taulemma}
Let $M \in \PTJmod$, let $f : i \to j$ be any morphism in $\J$, and let $C \in \Sigma_{\mathsf{Sort}}$. Then for any term $u \in \mathsf{Term}^c\left(\Sigma\left(M^i, \mathsf{x}_C\right)\right)$ with $\T\left(M^i, \mathsf{x}_C\right) \vdash u \downarrow$ and $u : A$, we have $\TJ\left(M, \mathsf{x}_{C^i}\right) \vdash \tau_{f}^C(u) = \alpha_f^A\left(\rho_{M^i}^C(u)\right)$. \qed
\end{lem}

\noindent We can now prove:

\begin{prop}
\label{grouphomexists}
For any $M \in \PTJmod$, there is a group homomorphism 
\[ \beta_M : \left(\prod_{i \in \J} G_{\T}\left(M^i\right)\right)^\J \times \mathsf{Aut}(\mathsf{Id}_\J)^M \to G_{\TJ}(M). \]
\end{prop}

\begin{proof}
Let $\gamma = (\gamma_i)_i \in \left(\prod_i G_{\T}\left(M^i\right)\right)^\J$ and $\psi = (\psi_B)_{B \in \Sigma} \in \mathsf{Aut}(\mathsf{Id}_\J)^M$. We must define $\beta_M(\gamma, \psi) \in G_{\TJ}(M)$, with $G_{\TJ}(M)$ being the group of all $\SigmaJ_{\mathsf{Sort}}$-indexed sequences $\left([t_{C^i}]\right)_{i, C} \in \prod_{i \in \J, C \in \Sigma} M\la \mathsf{x}_{C^i}\ra_{C^i}$ that are invertible, commute generically with all function symbols of $\SigmaJ$, and reflect definedness. Each $t_{C^i} \in \mathsf{Term}^c\left(\SigmaJ\left(M, \mathsf{x}_{C^i}\right)\right)$ is a closed term of sort $C^i$ with $\TJ\left(M, \mathsf{x}_{C^i}\right) \vdash t_{C^i} \downarrow$. 

So let $i \in \ob\J$ and $C \in \Sigma_{\mathsf{Sort}}$, and let us define $\beta_M(\gamma, \psi)_{C^i} \in M\la \mathsf{x}_{C^i} \ra_{C^i}$. 
Since $\gamma_i \in G_{\T}\left(M^i\right)$, we know that $\gamma_i^C = \left[s^i_C\right] \in M^i\la \mathsf{x}_C \ra_C$. So $s^i_C \in \mathsf{Term}^c\left(\Sigma\left(M^i, \mathsf{x}_C\right)\right)_C$ is a closed term of sort $C$ with $\T\left(M^i, \mathsf{x}_C\right) \vdash s_C^i \downarrow$. Then $\rho_{M^i}^C\left(s^i_C\right) \in \mathsf{Term}^c\left(\SigmaJ\left(M, \mathsf{x}_{C^i}\right)\right)_{C^i}$ and $\TJ\left(M, \mathsf{x}_{C^i}\right) \vdash \rho_{M^i}^C\left(s^i_C\right) \downarrow$, since $\rho_{M^i}^C : \T\left(M^i, \mathsf{x}_C\right) \to \TJ\left(M, \mathsf{x}_{C^i}\right)$ is a theory morphism by Lemma \ref{theorymorphismofM^i}. 

Suppose first that $\T\left(M^i\right)$ is \emph{non-trivial} for the sort $C$. Then $i \in \ob\J^M_C$ and $\psi_C(i) : i \xrightarrow{\sim} i$ is an isomorphism in $\J$. So then $\alpha_{\psi_C(i)}^C : C^i \to C^i$ is a function symbol of $\SigmaJ$ and $\TJ\left(M, \mathsf{x}_{C^i}\right) \vdash \alpha_{\psi_C(i)}^C(\mathsf{x}_{C^i}) \downarrow$, and it then follows by \cite[Lemma 2.2.24]{thesis} that $\TJ\left(M, \mathsf{x}_{C^i}\right) \vdash \rho_{M^i}^C\left(s_{C}^i\right)\left[\alpha_{\psi_C(i)}^C(\mathsf{x}_{C^i})/\mathsf{x}_{C^i}\right] \downarrow$. So then $\left[\rho_{M^i}^C\left(s_{C}^i\right)\left[\alpha_{\psi_C(i)}^C(\mathsf{x}_{C^i})/\mathsf{x}_{C^i}\right]\right] \in M\la \mathsf{x}_{C^i} \ra_{C^i}$, and we therefore set \[ \beta_M(\gamma, \psi)_{C^i} := \left[\rho_{M^i}^C\left(s_{C}^i\right)\left[\alpha_{\psi_C(i)}^C(\mathsf{x}_{C^i})/\mathsf{x}_{C^i}\right]\right] \in M\la \mathsf{x}_{C^i} \ra_{C^i}. \] If $\T\left(M^i\right)$ is \emph{trivial} for the sort $C$, then we simply set $\beta_M(\gamma, \psi)_{C^i} := [\mathsf{x}_{C^i}] \in M\la \mathsf{x}_{C^i} \ra_{C^i}$.
It is then shown in the proof of \cite[Proposition 5.1.20]{thesis} that $\beta_M$ is well-defined. We now prove in a series of claims that $\beta_M(\gamma, \psi) \in G_{\TJ}(M)$.
\begin{claim}
$\beta_M(\gamma, \psi)$ is invertible.
\end{claim} 
\begin{proof}
Let $i \in \ob\J$ and $C \in \Sigma_{\mathsf{Sort}}$. The result is trivial to verify if $\T\left(M^i\right)$ is trivial for the sort $C$, so assume otherwise. Since $\gamma_i \in G_{\T}\left(M^i\right)$, there is some $\left[\left(s_C^i\right)^{-1}\right] \in M^i \la \mathsf{x}_C \ra_C$ with
\[ \left[s_C^i\left[\left(s_C^i\right)^{-1}/\mathsf{x}_C\right]\right] = [\mathsf{x}_C] = \left[\left(s_C^i\right)^{-1}\left[s_C^i/\mathsf{x}_C\right]\right] \in M^i\la \mathsf{x}_C \ra_C, \] i.e.
\[ \T\left(M^i, \mathsf{x}_C\right) \vdash s_C^i\left[\left(s_C^i\right)^{-1}/\mathsf{x}_C\right] = \mathsf{x}_C = \left(s_C^i\right)^{-1}\left[s_C^i/\mathsf{x}_C\right]. \] Now consider $\rho_{M^i}^C\left(\left(s_C^i\right)^{-1}\right) \in \mathsf{Term}^c\left(\SigmaJ\left(M, \mathsf{x}_{C^i}\right)\right)_{C^i}$: since $\T\left(M^i, \mathsf{x}_C\right) \vdash \left(s_C^i\right)^{-1} \downarrow$, it follows from Lemma \ref{theorymorphismofM^i} that $\TJ\left(M, \mathsf{x}_{C^i}\right) \vdash \rho_{M^i}^C\left(\left(s_C^i\right)^{-1}\right) \downarrow$. Then because $\alpha_{\psi_C(i)^{-1}}^C : C^i \to C^i$ is provably total in $\TJ$, we obtain $\TJ\left(M, \mathsf{x}_{C^i}\right) \vdash \alpha_{\psi_C(i)^{-1}}^C\left(\rho_{M^i}^C\left(\left(s_C^i\right)^{-1}\right)\right) \downarrow$, so that \[ \left[\alpha_{\psi_C(i)^{-1}}^C\left(\rho_{M^i}^C\left(\left(s_C^i\right)^{-1}\right)\right)\right] \in M\la \mathsf{x}_{C^i} \ra_{C^i}. \]
So we set 
\[ \beta_M(\gamma, \psi)_{C^i}^{-1} := \left[\alpha_{\psi_C(i)^{-1}}^C\left(\rho_{M^i}^C\left(\left(s_C^i\right)^{-1}\right)\right)\right], \] and it is now straightforward to show (as in the proof of \cite[Proposition 5.1.20]{thesis}) that this is a substitutional inverse of $\beta_M(\gamma, \psi)_{C^i}$, which proves that $\beta_M(\gamma, \psi)$ is invertible. 
\end{proof}

\begin{claim}
$\beta_M(\gamma, \psi)$ commutes generically with all function symbols of $\SigmaJ$. 
\end{claim}
\begin{proof}
First let $i \in \ob\J$ and let $g : A_1 \times \ldots \times A_n \to A$ be a function symbol of $\Sigma$. We must show that $\beta_M(\gamma, \psi)$ commutes generically with the function symbol $g^i : A_1^i \times \ldots \times A_n^i \to A^i$ of $\SigmaJ$. Assume without loss of generality that $\T\left(M^i\right)$ is non-trivial for each of the sorts $A_1, \ldots, A_n, A$; if this is \emph{not} the case, then the argument required is a simpler version of the one we are about to give. 

We must show that the sequent
\[ g^i\left(\mathsf{x}_{A_1^i}, \ldots, \mathsf{x}_{A_n^i}\right) \downarrow \ \ \vdash \left\{\rho_{M^i}^{A}\left(s_A^i\right)\left[\alpha_{\psi_A(i)}^A(\mathsf{x}_{A^i})/\mathsf{x}_{A^i}\right]\right\} \left[g^i\left(\mathsf{x}_{A_1^i}, \ldots, \mathsf{x}_{A_n^i}\right)/\mathsf{x}_{A^i}\right] \] \[ = g^i\left(\rho_{M^i}^{A_1}\left(s_{A_1}^i\right)\left[\alpha_{\psi_{A_1}(i)}^{A_1}\left(\mathsf{x}_{A^i_1}\right)/\mathsf{x}_{A^i_1}\right], \ldots, \rho_{M^i}^{A_n}\left(s_{A_n}^i\right)\left[\alpha_{\psi_{A_n}(i)}^{A_n}\left(\mathsf{x}_{A^i_n}\right)/\mathsf{x}_{A^i_n}\right]\right) \] 
is provable in the theory $\TJ\left(M, \mathsf{x}_{A_1^i}, \ldots, \mathsf{x}_{A_n^i}\right)$ (technically, we need to ensure that the indeterminates on the right side of the equation are pairwise distinct (see Definition \ref{commutesgenericallydefn}), but we will ignore this subtlety here and elsewhere in the proof to increase readability). Since $\gamma_i \in G_{\T}\left(M^i\right)$, we know that the sequent
\[ g\left(\mathsf{x}_{A_1}, \ldots, \mathsf{x}_{A_n}\right) \downarrow \ \vdash s_A^i\left[g\left(\mathsf{x}_{A_1}, \ldots, \mathsf{x}_{A_n}\right)/\mathsf{x}_A\right] = g\left(s_{A_1}^i, \ldots, s_{A_n}^i\right) \tag{$*$} \] is provable in the theory $\T\left(M^i, \mathsf{x}_{A_1}, \ldots, \mathsf{x}_{A_n}\right)$. As in Definition \ref{sigmorphismofM^i} and Lemma \ref{theorymorphismofM^i}, we can define a signature morphism $\rho_{M^i}^{\vec{A}} : \Sigma\left(M^i, \mathsf{x}_{A_1}, \ldots, \mathsf{x}_{A_n}\right) \to \SigmaJ\left(M, \mathsf{x}_{A_1^i}, \ldots, \mathsf{x}_{A_n^i}\right)$ that will be a theory morphism \[ \rho_{M^i}^{\vec{A}} : \T\left(M^i, \mathsf{x}_{A_1}, \ldots, \mathsf{x}_{A_n}\right) \to \TJ\left(M, \mathsf{x}_{A_1^i}, \ldots, \mathsf{x}_{A_n^i}\right); \] on $\Sigma\left(M^i\right)$, we define $\rho_{M^i}^{\vec{A}}$ as in Definition \ref{sigmorphismofM^i}, and for any $1 \leq j \leq n$ we set $\rho_{M^i}^{\vec{A}}(\mathsf{x}_{A_j}) := \mathsf{x}_{A_j^i}$ (here $\vec{A} = A_1, \ldots, A_n$). Then it is obvious that for any $1 \leq j \leq n$, the signature morphism $\rho_{M^i}^{\vec{A}}$ agrees with the signature morphism $\rho_{M^i}^{A_j} : \Sigma\left(M^i, \mathsf{x}_{A_j}\right) \to \SigmaJ\left(M, \mathsf{x}_{A_j^i}\right)$ when restricted to $\Sigma\left(M^i, \mathsf{x}_{A_j}\right)$, which implies that $\rho_{M^i}^{\vec{A}}\left(s_{A_j}^i\right) \equiv \rho_{M^i}^{A_j}\left(s_{A_j}^i\right)$ for all $1 \leq j \leq n$. Also (by Lemma \ref{theorymorphismofM^isubst}), we have
\[ \rho_{M^i}^{A}\left(s_A^i\right)\left[g^i\left(\mathsf{x}_{A_1^i}, \ldots, \mathsf{x}_{A_n^i}\right)/\mathsf{x}_{A^i}\right] \equiv \rho_{M^i}^{\vec{A}}\left(s_A^i\left[g(\mathsf{x}_{A_1}, \ldots, \mathsf{x}_{A_n})/\mathsf{x}_A\right]\right). \] Now, since $\rho_{M^i}^{\vec{A}} : \T\left(M^i, \mathsf{x}_{A_1}, \ldots, \mathsf{x}_{A_n}\right) \to \TJ\left(M, \mathsf{x}_{A_1^i}, \ldots, \mathsf{x}_{A_n^i}\right)$ is a theory morphism, it follows that the $\rho_{M^i}^{\vec{A}}$-translation of the aforementioned sequent ($*$) provable in $\T\left(M^i, \mathsf{x}_{A_1}, \ldots, \mathsf{x}_{A_n}\right)$ will be provable in $\TJ\left(M, \mathsf{x}_{A_1^i}, \ldots, \mathsf{x}_{A_n^i}\right)$. In other words, the following sequent is provable in $\TJ\left(M, \mathsf{x}_{A_1^i}, \ldots, \mathsf{x}_{A_n^i}\right)$:
\[ g^i\left(\mathsf{x}_{A_1^i}, \ldots, \mathsf{x}_{A_n^i}\right) \downarrow \ \ \vdash \rho_{M^i}^{A}\left(s_A^i\right)\left[g^i\left(\mathsf{x}_{A_1^i}, \ldots, \mathsf{x}_{A_n^i}\right)/\mathsf{x}_{A^i}\right] = g^i\left(\rho_{M^i}^{A_1}\left(s_{A_1}^i\right), \ldots, \rho_{M^i}^{A_n}\left(s_{A_n}^i\right)\right). \]
Now, let us reason in the theory $\TJ\left(M, \mathsf{x}_{A_1^i}, \ldots, \mathsf{x}_{A_n^i}\right) \cup \left\{ \top \vdash g^i\left(\mathsf{x}_{A_1^i}, \ldots, \mathsf{x}_{A_n^i}\right) \downarrow\right\}$ (referred to as the `expanded theory' for the rest of this argument), one of whose theorems is therefore the preceding equation. By substituting $\alpha_{\psi_A(i)}^{A_1}\left(\mathsf{x}_{A^i_1}\right)$ for $\mathsf{x}_{A^i_1}$, $\ldots$, $\alpha_{\psi_A(i)}^{A_n}\left(\mathsf{x}_{A^i_n}\right)$ for $\mathsf{x}_{A^i_n}$, the following equation is then provable in the expanded theory: 
\[ \rho_{M^i}^A\left(s_A^i\right)\left[g^i\left(\alpha_{\psi_A(i)}^{A_1}\left(\mathsf{x}_{A^i_1}\right), \ldots, \alpha_{\psi_A(i)}^{A_n}\left(\mathsf{x}_{A^i_n}\right)\right)/\mathsf{x}_{A^i}\right] \] \[ = g^i\left(\rho_{M^i}^{A_1}\left(s_{A_1}^i\right)\left[\alpha_{\psi_{A}(i)}^{A_1}\left(\mathsf{x}_{A^i_1}\right)/\mathsf{x}_{A^i_1}\right], \ldots, \rho_{M^i}^{A_n}\left(s_{A_n}^i\right)\left[\alpha_{\psi_{A}(i)}^{A_n}\left(\mathsf{x}_{A^i_n}\right)/\mathsf{x}_{A^i_n}\right]\right). \] 
Since the expanded theory (because of Axiom \ref{thefunctortheory}.\ref{alphasarehoms}) proves the equation
\[ g^i\left(\alpha_{\psi_A(i)}^{A_1}\left(\mathsf{x}_{A^i_1}\right), \ldots, \alpha_{\psi_A(i)}^{A_n}\left(\mathsf{x}_{A^i_n}\right)\right) = \alpha_{\psi_A(i)}^A\left(g^i\left(\mathsf{x}_{A^i_1}, \ldots, \mathsf{x}_{A^i_n}\right)\right), \] it follows that the expanded theory proves the equation
\[ \rho_{M^i}^A\left(s_A^i\right)\left[\alpha_{\psi_A(i)}^A\left(g^i\left(\mathsf{x}_{A^i_1}, \ldots, \mathsf{x}_{A^i_n}\right)\right)/\mathsf{x}_{A^i}\right] = g^i\left(\rho_{M^i}^{A_1}\left(s_{A_1}^i\right)\left[\alpha_{\psi_{A}(i)}^{A_1}\left(\mathsf{x}_{A^i_1}\right)/\mathsf{x}_{A^i_1}\right], \ldots, \rho_{M^i}^{A_n}\left(s_{A_n}^i\right)\left[\alpha_{\psi_{A}(i)}^{A_n}\left(\mathsf{x}_{A^i_n}\right)/\mathsf{x}_{A^i_n}\right]\right), \] i.e. the expanded theory proves the equation 
\[ \left(\rho_{M^i}^{A}\left(s_A^i\right)\left[\alpha_{\psi_A(i)}^A\left(\mathsf{x}_{A^i}\right)/\mathsf{x}_{A^i}\right]\right) \left[g^i\left(\mathsf{x}_{A_1^i}, \ldots, \mathsf{x}_{A_n^i}\right)/\mathsf{x}_{A^i}\right] \] \[ = g^i\left(\rho_{M^i}^{A_1}\left(s_{A_1}^i\right)\left[\alpha_{\psi_{A}(i)}^{A_1}\left(\mathsf{x}_{A^i_1}\right)/\mathsf{x}_{A^i_1}\right], \ldots, \rho_{M^i}^{A_n}\left(s_{A_n}^i\right)\left[\alpha_{\psi_{A}(i)}^{A_n}\left(\mathsf{x}_{A^i_n}\right)/\mathsf{x}_{A^i_n}\right]\right). \] So to complete the argument, it remains to show (by the deduction theorem in \cite[Theorem 10]{Horn}) that the expanded theory proves the equation
\[ g^i\left(\rho_{M^i}^{A_1}\left(s_{A_1}^i\right)\left[\alpha_{\psi_{A_1}(i)}^{A_1}\left(\mathsf{x}_{A^i_1}\right)/\mathsf{x}_{A^i_1}\right], \ldots, \rho_{M^i}^{A_n}\left(s_{A_n}^i\right)\left[\alpha_{\psi_{A_n}(i)}^{A_n}\left(\mathsf{x}_{A^i_n}\right)/\mathsf{x}_{A^i_n}\right]\right) \] \[ = g^i\left(\rho_{M^i}^{A_1}\left(s_{A_1}^i\right)\left[\alpha_{\psi_{A}(i)}^{A_1}\left(\mathsf{x}_{A^i_1}\right)/\mathsf{x}_{A^i_1}\right], \ldots, \rho_{M^i}^{A_n}\left(s_{A_n}^i\right)\left[\alpha_{\psi_{A}(i)}^{A_n}\left(\mathsf{x}_{A^i_n}\right)/\mathsf{x}_{A^i_n}\right]\right) \] (the difference in the two terms being the $\psi$-subscripts). It suffices to show that for any position $1 \leq m \leq n$, we can `swap' $\rho_{M^i}^{A_m}\left(s_{A_m}^i\right)\left[\alpha_{\psi_{A_m}(i)}^{A_m}\left(\mathsf{x}_{A^i_m}\right)\right]$ for $\rho_{M^i}^{A_m}\left(s_{A_m}^i\right)\left[\alpha_{\psi_A(i)}^{A_m}\left(\mathsf{x}_{A^i_m}\right)\right]$ within position $m$ in $g^i$ (modulo the expanded theory). If $g^{M^i}$ is degenerate in position $m$, then this easily follows by the definition of `degenerate' (see Definition \ref{degenerate}): specifically, if $\T\left(M^i\right)$ proves the equation in Definition \ref{degenerate} for $g$, then it follows from Lemma \ref{theorymorphismofM^i} that $\TJ(M)$ will prove the corresponding equation for $g^i$.

Otherwise, if $g^{M^i}$ is \emph{non-degenerate} in position $m$, then since $\psi \in \mathsf{Aut}(\mathsf{Id}_\J)^M$, it follows that $\psi_{A_m}(i) = \psi_A(i) : i \xrightarrow{\sim} i$, which again easily yields the desired result. This completes the proof that $\beta_M(\gamma, \psi)$ commutes generically with the function symbol $g^i$ of $\SigmaJ$. 

\bigskip Now let $B \in \Sigma_{\mathsf{Sort}}$ and let $f : i \to j$ be an arbitrary morphism in $\J$. We must show that $\beta_M(\gamma, \psi)$ commutes generically with the function symbol $\alpha_f^B : B^i \to B^j$ of $\SigmaJ$. Suppose first that both $\T\left(M^i\right)$ and $\T\left(M^j\right)$ are non-trivial for the sort $B$. Then we must show that the equation
\[ \left\{\rho_{M^j}^B\left(s_B^j\right)\left[\alpha_{\psi_B(j)}^B\left(\mathsf{x}_{B^j}\right)/\mathsf{x}_{B^j}\right]\right\}\left[\alpha_f^B\left(\mathsf{x}_{B^i}\right)/\mathsf{x}_{B^j}\right] = \alpha_f^B\left(\rho_{M^i}^B\left(s_B^i\right)\left[\alpha_{\psi_B(i)}^B\left(\mathsf{x}_{B^i}\right)/\mathsf{x}_{B^i}\right]\right) \] 
is provable in $\TJ\left(M, \mathsf{x}_{B^i}\right)$ (since $\alpha_f^B$ is provably total in $\TJ$). Since $\gamma \in \left(\prod_i G_{\T}\left(M^i\right)\right)^\J$, we know that $G_{\T}\left(F^M(f)\right)(\gamma_i) = \gamma_j$, i.e. 
\[ G_{\T}\left(F^M(f)\right)\left(\left(\left[s_C^i\right]\right)_{C \in \Sigma}\right) = \left(\left[s_C^j\right]\right)_{C \in \Sigma}. \] Recalling our earlier convention that $f^M := F^M(f) : M^i \to M^j$, this equality means that \[ \left(\left[\rho_{f^M}^C\left(s_C^i\right)\right]\right)_{C \in \Sigma} = \left(\left[s_C^j\right]\right)_{C \in \Sigma} \in G_{\T}\left(M^j\right) \] (see \cite[Definition 2.2.36]{thesis}). In particular, for our fixed sort $B$ we have $\left[\rho_{f^M}^B\left(s_B^i\right)\right] = \left[s_B^j\right] \in M^j\la \mathsf{x}_B \ra_B$, which means that $\T\left(M^j, \mathsf{x}_B\right) \vdash \rho_{f^M}^B\left(s_B^i\right) = s_B^j$. Since $\rho_{M^j}^B : \T\left(M^j, \mathsf{x}_B\right) \to \TJ\left(M, \mathsf{x}_{B^j}\right)$ is a theory morphism by Lemma \ref{theorymorphismofM^i}, we then have 
\[ \TJ\left(M, \mathsf{x}_{B^j}\right) \vdash \rho_{M^j}^B\left(\rho_{f^M}^B\left(s_B^i\right)\right) = \rho_{M^j}^B\left(s_B^j\right). \] And since $\sigma_{f}^B : \TJ\left(M, \mathsf{x}_{B^j}\right) \to \TJ\left(M, \mathsf{x}_{B^i}\right)$ is a theory morphism by Lemma \ref{theorymorphismsigma}, we obtain
\[ \TJ\left(M, \mathsf{x}_{B^i}\right) \vdash \sigma_{f}^B\left(\rho_{M^j}^B\left(\rho_{f^M}^B\left(s_B^i\right)\right)\right) = \sigma_{f}^B\left(\rho_{M^j}^B\left(s_B^j\right)\right), \] i.e. (see Definition \ref{sigmorphismtau}) \[ \TJ\left(M, \mathsf{x}_{B^i}\right) \vdash \tau_{f}^B\left(s_B^i\right) = \sigma_{f}^B\left(\rho_{M^j}^B\left(s_B^j\right)\right). \] Also, since $\sigma_{f}^B : \SigmaJ\left(M, \mathsf{x}_{B^j}\right) \to \SigmaJ\left(M, \mathsf{x}_{B^i}\right)$ is the identity except for the fact that $\sigma_{f}^B(\mathsf{x}_{B^j}) := \alpha_f^B(\mathsf{x}_{B^i})$, it easily follows that \[ \sigma_{f}^B\left(\rho_{M^j}^B\left(s_B^j\right)\right) \equiv \rho_{M^j}^B\left(s_B^j\right)\left[\alpha_f^B(\mathsf{x}_{B^i})/\mathsf{x}_{B^j}\right]. \] So we have \[ \TJ\left(M, \mathsf{x}_{B^i}\right) \vdash \tau_{f}^B\left(s_B^i\right) = \rho_{M^j}^B\left(s_B^j\right)\left[\alpha_f^B(\mathsf{x}_{B^i})/\mathsf{x}_{B^j}\right]. \] Finally, since $\T\left(M^i, \mathsf{x}_B\right) \vdash s_B^i \downarrow$, it follows from Lemma \ref{taulemma} that
\[ \TJ\left(M, \mathsf{x}_{B^i}\right) \vdash \tau_{f}^B\left(s_B^i\right) = \alpha_f^B\left(\rho_{M^i}^B\left(s_B^i\right)\right). \] Combining this equation with the previous one, we then have \[ \TJ\left(M, \mathsf{x}_{B^i}\right) \vdash \rho_{M^j}^B\left(s_B^j\right)\left[\alpha_f^B(\mathsf{x}_{B^i})/\mathsf{x}_{B^j}\right] = \alpha_f^B\left(\rho_{M^i}^B\left(s_B^i\right)\right). \] Substituting $\alpha_{\psi_B(i)}^B(\mathsf{x}_{B^i})$ for $\mathsf{x}_{B^i}$ and applying \cite[Lemma 2.2.24]{thesis}, $\TJ\left(M, \mathsf{x}_{B^i}\right)$ then proves the equation
\[ \rho_{M^j}^B\left(s_B^j\right)\left[\alpha_f^B\left(\alpha_{\psi_B(i)}^B(\mathsf{x}_{B^i})\right)/\mathsf{x}_{B^j}\right] = \alpha_f^B\left(\rho_{M^i}^B\left(s_B^i\right)\left[\alpha_{\psi_B(i)}^B(\mathsf{x}_{B^i})/\mathsf{x}_{B^i}\right]\right). \] So to complete the argument, it remains to prove that $\TJ\left(M, \mathsf{x}_{B^i}\right)$ proves the equation
\[ \rho_{M^j}^B\left(s_B^j\right)\left[\alpha_f^B\left(\alpha_{\psi_B(i)}^B(\mathsf{x}_{B^i})\right)/\mathsf{x}_{B^j}\right] = \left\{\rho_{M^j}^B\left(s_B^j\right)\left[\alpha_{\psi_B(j)}^B(\mathsf{x}_{B^j})/\mathsf{x}_{B^j}\right]\right\}[\alpha_f^B(\mathsf{x}_{B^i})/\mathsf{x}_{B^j}]. \] But the following sequence of equations is provable in $\TJ\left(M, \mathsf{x}_{B^i}\right)$, as desired:
\begin{align*}
&\ \ \ \left\{\rho_{M^j}^B\left(s_B^j\right)\left[\alpha_{\psi_B(j)}^B\left(\mathsf{x}_{B^j}\right)/\mathsf{x}_{B^j}\right]\right\}\left[\alpha_f^B\left(\mathsf{x}_{B^i}\right)/\mathsf{x}_{B^j}\right] \\
&\equiv \rho_{M^j}^B\left(s_B^j\right)\left[\alpha_{\psi_B(j)}^B\left(\alpha_f^B\left(\mathsf{x}_{B^i}\right)\right)/\mathsf{x}_{B^j}\right] \\
&= \rho_{M^j}^B\left(s_B^j\right)\left[\alpha_{\psi_B(j) \circ f}^B\left(\mathsf{x}_{B^i}\right)/\mathsf{x}_{B^j}\right] \\
&= \rho_{M^j}^B\left(s_B^j\right)\left[\alpha_{f \circ \psi_B(i)}^B\left(\mathsf{x}_{B^i}\right)/\mathsf{x}_{B^j}\right] \\
&= \rho_{M^j}^B\left(s_B^j\right)\left[\alpha_f^B\left(\alpha_{\psi_B(i)}^B\left(\mathsf{x}_{B^i}\right)\right)/\mathsf{x}_{B^j}\right];
\end{align*}
the second equality follows by Axiom \ref{thefunctortheory}.\ref{alphacompworks}, the third by naturality of $\psi_B \in \mathsf{Aut}\left(\mathsf{Id}_{\J^M_B}\right)$, and the last by Axiom \ref{thefunctortheory}.\ref{alphacompworks} again. 

Now suppose that $\T\left(M^i\right)$ is trivial for the sort $B$, which implies that $\T\left(M^i, \mathsf{x}_B\right)$ is also trivial for the sort $B$. Given the morphism $f : i \to j$, we have the induced $\Sigma$-morphism $f^M : M^i \to M^j$, which in turn induces the theory morphism $\rho_{f^M}^B : \T\left(M^i, \mathsf{x}_B\right) \to \T\left(M^j, \mathsf{x}_B\right)$, the existence of which implies that $\T\left(M^j, \mathsf{x}_B\right)$ is also trivial for the sort $B$. But by Lemma \ref{theorymorphismofM^i}, it then easily follows that $\TJ\left(M, \mathsf{x}_{B^j}\right)$ is trivial for the sort $B^j$, and hence will prove all equations between terms of this sort, which clearly yields the desired result. And if $\T\left(M^j\right)$ is trivial for the sort $B$, then $\T\left(M^j, \mathsf{x}_B\right)$ is trivial for the sort $B$, which then also yields the desired result, as just explained. This completes the proof that $\beta_M(\gamma, \psi)$ commutes generically with $\alpha_f^B$, which in turn completes the proof that $\beta_M(\gamma, \psi)$ commutes generically with all function symbols of $\SigmaJ$. 
\end{proof}

\begin{claim}
$\beta_M(\gamma, \psi)$ reflects definedness. 
\end{claim}
\begin{proof}
This can be proven analogously to the previous claim; we refer the reader to \cite[Claim 5.1.23]{thesis} for the details. 
\end{proof} 

\noindent With the preceding claims, we have now proved that \[ \beta_M : \left(\prod_i G_{\T}\left(M^i\right)\right)^\J \times \mathsf{Aut}(\mathsf{Id}_\J)^M \to G_{\TJ}(M) \] is a well-defined function. To complete the proof of Proposition \ref{grouphomexists}, it remains to show that $\beta_M$ preserves the group multiplication, which is not too difficult; we refer the reader to the proof of \cite[Proposition 5.1.20]{thesis} for this verification. 
\end{proof}

Our next step is to show that the group homomorphism $\beta_M$ is \emph{bijective}. To motivate our proof of this, let us assume for simplicity that $\T$ has just one sort $X$, so that $\Aut(\Id_\J)^M = \Aut\left(\Id_{\J_X^M}\right)$ for any $M \in \TJmod$. Assume also for simplicity that $M \in \TJmod$ is such that each $M^i \in \Tmod$ for $i \in \ob\J$ is non-trivial for the unique sort $X$, so that $\Aut\left(\Id_{\J_X^M}\right) = \Aut(\Id_\J)$ and the group homomorphism 
\[ \beta_M : \left(\prod_i G_{\T}\left(M^i\right)\right)^\J \times \Aut(\Id_\J) \to G_{\TJ}(M) \] is defined as in the proof of Proposition \ref{grouphomexists} by
\[ \left([s_i]_{i \in \ob\J}, \psi\right) \mapsto \left(\left[\rho_{M^i}\left(s_i\right)\left[\alpha_{\psi(i)}(\x_i)/\x_i\right]\right]\right)_{i \in \ob\J}, \] where we have suppressed the subscripts and superscripts for the unique sort $X$ of $\T$. To show that this assignment is bijective, we will essentially reason as follows. First, we show in Lemma \ref{mainalphalemma} that any closed term $u \in \Term^c\left(\SigmaJ(M, \x_i)\right)$ for $i \in \ob\J$ has a `normal form' $u'$ in which all function symbols of $\SigmaJ$ of the form $\alpha_f$ for $f \in \mor\J$ are `pushed inside as far as possible', and we call any term in this normal form an $\alpha$\emph{-restricted} term. In Definition \ref{theta} we then show that we can take any $\alpha$-restricted term $u \in \Term^c\left(\SigmaJ(M, \x_i)\right)$, replace any subterm in it of the form $\alpha_f(\x_i)$ by a new constant symbol $\x_f : X$, and thereby obtain an induced term $\theta(u)$ of the signature $\Sigma\left(M^i\right)$ augmented by these new constants $\x_f$ indexed by $\mor\J$. We then show in Proposition \ref{thetaprop} that this process preserves the provability of equations in $\TJ(M, \x_i)$. In Definition \ref{theta^*} we also show that we can erase the various morphism subscripts from these new constants $\x_f$ to obtain from $\theta(u)$ a term $\theta^*(u)$ over the more familiar signature $\Sigma\left(M^i, \x\right)$, and we show in Lemma \ref{theta^*lemma} that the mapping $\theta^*$ preserves the provability of a certain kind of sequent in $\TJ(M, \x_i)$, and in Lemma \ref{secondtheta^*lemma} that it preserves the provability of equations. After some further technical lemmas regarding $\theta$ and $\theta^*$, we finally prove in Proposition \ref{betaisinjective} that $\beta_M$ is injective. The idea behind this proof is roughly as follows: if $\left([s_i]_{i \in \ob\J}, \psi\right)$ is an element of the domain of $\beta_M$ for which $\left[\rho_{M^i}\left(s_i\right)\left[\alpha_{\psi(i)}(\x_i)/\x_i\right]\right] = [\x_i]$ holds in $M\la \x_i\ra_i$, i.e. for which $\TJ(M, \x_i) \vdash \rho_{M^i}\left(s_i\right)\left[\alpha_{\psi(i)}(\x_i)/\x_i\right] = \x_i$ for each $i \in \ob\J$, then we can essentially show using the aforementioned results that the equation $s_i\left[\x_{\psi(i)}/\x\right] = \x_{\id_i}$ is provable in the theory $\T\left(M^i\right)$ augmented by the new constants $\x_f$ for $f \in \mor\J$, which then (by Lemma \ref{occurrenceofidentityconstant}) forces $\psi(i) = \id_i$ and thereby entails $\T\left(M^i, \x\right) \vdash s_i = \x$, as desired. 

To prove in Proposition \ref{betaissurjective} that $\beta_M$ is surjective, we need to impose two conditions on $\T$ in Definitions \ref{singleindeterminateisotropy} and \ref{singlesortednontotaloperations}. Given an arbitrary element $\left([s_i]\right)_{i \in \ob\J} \in G_{\TJ}(M)$, we can assume without loss of generality that each $s_i$ is in $\alpha$-restricted normal form. We then apply $\theta^*$ to each $s_i$ to obtain $\left[\theta^*(s_i)\right] \in M^i\la \x\ra$ for each $i \in \ob\J$, and we show using the aforementioned results that $\left(\left[\theta^*(s_i)\right]\right)_{i \in \ob\J} \in \left(\prod_i G_\T\left(M^i\right)\right)^\J$. To construct an appropriate natural automorphism $\psi : \Id_\J \xrightarrow{\sim} \Id_\J$, we use the two aforementioned assumptions on $\T$. The assumption of Definition \ref{singleindeterminateisotropy} guarantees that each $s_i$ can be assumed to have exactly \emph{one} occurrence of $\x_i$, and hence exactly one subterm of the form $\alpha_f(\x_i)$ for an endomorphism $f : i \to i$, which we choose to be $\psi(i)$. We then show that $\psi$, so defined, is a natural automorphism of $\Id_\J$. Using the second assumption on $\T$ in Definition \ref{singlesortednontotaloperations}, we then show that if $\T$ is not single-sorted, then the various natural automorphisms $\psi_B : \Id_{\J_B^M} \xrightarrow{\sim} \Id_{\J_B^M}$ for $B \in \Sigma_\Sort$ are compatible with each other in the sense of Definition \ref{automorphismsubgroup}, so that $\left(\psi_B\right)_{B \in \Sigma} \in \Aut(\Id_\J)^M$. Let us now embark on providing the details. 

For any category $\C$ and object $C \in \ob\C$, we let $\Dom(C)$ be the class of all morphisms in $\C$ with domain $C$ (which is certainly a set if $\C$ is small).

\begin{defn}
\label{alpharestricted}
{\em If $M \in \PTJmod$ and $u \in \mathsf{Term}^c\left(\SigmaJ\left(M, \mathsf{x}_{A^i}\right)\right)$ for some $A \in \Sigma_\Sort$ and $i \in \ob\J$, then we say that $u$ is $\alpha$-\emph{restricted} if the only subterms of $u$ of the form $\alpha_f^C(v)$ are those with $C = A$ and $v \equiv \mathsf{x}_{A^i}$ and $\mathsf{dom}(f) = i$.   

In other words, the term $u \in \mathsf{Term}^c\left(\SigmaJ\left(M, \mathsf{x}_{A^i}\right)\right)$ is $\alpha$-restricted if all `$\alpha$-subterms' of $u$ have the form $\alpha_f^A(\mathsf{x}_{A^i})$ for some $f \in \mathsf{Dom}(i)$. \qed }
\end{defn}

\noindent Essentially, an $\alpha$-restricted term is a term in which all of the $\alpha$-function symbols have been `pushed inside as far as possible'. In order to prove that every (provably defined) term has an $\alpha$-restricted equivalent, we require the following lemma, whose proof may be found in \cite[Lemma 5.1.26]{thesis}: 
 
\begin{lem}
\label{firstalphalemma}
Let $M \in \PTJmod$ and let $u \in \mathsf{Term}^c\left(\SigmaJ\left(M, \mathsf{x}_{A^i}\right)\right)$ be $\alpha$-restricted, where $A \in \Sigma_\Sort$ and $i \in \ob\J$. If $u : C^j$ for some $j \in \ob\J$ and $C \in \Sigma_{\mathsf{Sort}}$, then for any morphism $f : j \to \mathsf{cod}(f)$ in $\J$, there is an $\alpha$-restricted term $u^f \in \mathsf{Term}^c\left(\SigmaJ\left(M, \mathsf{x}_{A^i}\right)\right)$ with $u^f : C^{\mathsf{cod}(f)}$ and $\TJ(M, \mathsf{x}_{A^i})$ proves the sequent $u \downarrow \ \vdash \alpha_f^C(u) = u^f$. \qed
\end{lem}

\noindent With the help of Lemma \ref{firstalphalemma} we can now show that any term has an $\alpha$-restricted equivalent; the proof may be found in \cite[Lemma 5.1.27]{thesis}:

\begin{lem}
\label{mainalphalemma}
If $M \in \PTJmod$ and $u \in \mathsf{Term}^c\left(\SigmaJ\left(M, \mathsf{x}_{A^i}\right)\right)$ for some $A \in \Sigma_\Sort$ and $i \in \ob\J$, then there is an $\alpha$-restricted term $u' \in \mathsf{Term}^c\left(\SigmaJ\left(M, \mathsf{x}_{A^i}\right)\right)$ of the same sort such that $\TJ(M, \mathsf{x}_{A^i})$ proves the sequent $u \downarrow \ \vdash u = u'$. \qed
\end{lem}

\noindent It is trivial to verify (from the proof of Lemma \ref{mainalphalemma}) that if $u \in \mathsf{Term}^c\left(\SigmaJ\left(M, \mathsf{x}_{A^i}\right)\right)$ is already $\alpha$-restricted, then $u \equiv u'$. We now discuss the augmentation of the various signatures $\Sigma\left(M^k\right)$ (for $M \in \TJmod$ and $k \in \ob\J$) by new constants indexed by $\mor\J$:

\begin{defn} 
{\em Let $M \in \PTJmod$.
\begin{itemize}
\item For any $k \in \ob\J$, let $\mathsf{Cod}(k) := \left\{f \in \mor\J : \mathsf{cod}(f) = k\right\}$. For any $k \in \ob\J$ and $B \in \Sigma_\Sort$, let $\Sigma\left(M^k, \overline{\mathsf{x}^B_{\mathsf{Cod}(k)}}\right)$ be the signature obtained from $\Sigma\left(M^k\right)$ by adding pairwise distinct new constant symbols $\mathsf{x}_f^B : B$ for every $f \in \mathsf{Cod}(k)$.

\item For any $k \in \ob\J$ and $B \in \Sigma_\Sort$, let $\T\left(M^k, \overline{\mathsf{x}^B_{\mathsf{Cod}(k)}}\right)$ be the quasi-equational theory over the signature $\Sigma\left(M^k, \overline{\mathsf{x}^B_{\mathsf{Cod}(k)}}\right)$ obtained from $\T\left(M^k\right)$ by adding the axioms $\top \vdash \mathsf{x}_f^B \downarrow$ for every $f \in \mathsf{Cod}(k)$. \qed 
\end{itemize}
}
\end{defn}  

\begin{defn}
\label{theta}
{\em Let $M \in \PTJmod$ and $A \in \Sigma_\Sort$ and $i \in \ob\J$, and let $\mathsf{Term}^c\left(\SigmaJ\left(M, \mathsf{x}_{A^i}\right)\right)^*$ be the subset of $\mathsf{Term}^c\left(\SigmaJ\left(M, \mathsf{x}_{A^i}\right)\right)$ consisting of the $\alpha$-\emph{restricted} terms.
We define a map 
\[ \theta : \mathsf{Term}^c\left(\SigmaJ\left(M, \mathsf{x}_{A^i}\right)\right)^* \to \bigcup_{k \in \ob\J} \mathsf{Term}^c\left(\Sigma\left(M^k, \overline{\mathsf{x}^A_{\mathsf{Cod}(k)}}\right)\right) \] with the property that if $u : C^k$ for $k \in \ob\J$ and $C \in \Sigma_{\mathsf{Sort}}$, then $\theta(u) \in \mathsf{Term}^c\left(\Sigma\left(M^k, \overline{\mathsf{x}^A_{\mathsf{Cod}(k)}}\right)\right)$ with $\theta(u) : C$. We define $\theta$ by induction on the structure of $u \in \mathsf{Term}^c\left(\SigmaJ\left(M, \mathsf{x}_{A^i}\right)\right)^*$:

\begin{itemize}

\item For any $f : i \to k$ in $\J$, we set $\theta(\mathsf{x}_{A^i}) := \mathsf{x}_{\mathsf{id}_i}^A : A$ and $\theta\left(\alpha_f^A(\mathsf{x}_{A^i})\right) := \mathsf{x}_f^A : A$ 
(note that $\alpha_f^A(\mathsf{x}_{A^i}) : A^{k}$ and $\mathsf{x}_f^A \in \mathsf{Term}^c\left(\Sigma\left(M^{k}, \overline{\mathsf{x}^A_{\mathsf{Cod}(k)}}\right)\right)$).

\item For any $k \in \ob\J$, $C \in \Sigma_{\mathsf{Sort}}$, and $s \in M_{C^k} = M^k_C$, we set $\theta\left(c_{C^k, s}^M\right) := c_{C, s}^{M^k}$.

\item For any $k \in \ob\J$, any function symbol $g : C_1 \times \ldots \times C_n \to C$ in $\Sigma$, and any $u_1, \ldots, u_n \in \mathsf{Term}^c\left(\SigmaJ\left(M, \mathsf{x}_{A^i}\right)\right)^*$ with $u_\ell : C^k_\ell$ for all $1 \leq \ell \leq n$, we set $\theta\left(g^k(u_1, \ldots, u_n)\right) := g\left(\theta(u_1), \ldots, \theta(u_n)\right)$. \qed
\end{itemize}
}
\end{defn}

\noindent The idea behind the map $\theta$ is that it takes an $\alpha$-restricted term $t$ and replaces all of the subterms in $t$ of the form $\alpha_f^A(\x_{A^i})$ by constant symbols $\x_f^A$. The next crucial result now states that $\theta$ preserves provability of equations; its proof may be found in \cite[Proposition 5.1.29]{thesis}. 

\begin{prop}
\label{thetaprop}
Let $M \in \PTJmod$ and let $s, t \in \mathsf{Term}^c\left(\SigmaJ\left(M, \mathsf{x}_{A^i}\right)\right)^*$ for some $i \in \ob\J$ and $A \in \Sigma_\Sort$, with $s, t : C^j$ for some $C \in \Sigma_{\mathsf{Sort}}$ and $j \in \ob\J$. If $\TJ(M, \x_{A_i}) \vdash s = t$, then $\T\left(M^j, \overline{\mathsf{x}^A_{\mathsf{Cod}(j)}}\right) \vdash \theta(s) = \theta(t)$. \qed
\end{prop}

\noindent We will also need the following technical lemma, whose proof may be found in \cite[Lemma 5.1.31]{thesis}. 

\begin{lem}
\label{occurrenceofidentityconstant}
Let $M \in \PTJmod$ and $k \in \ob\J$ and $B \in \Sigma_\Sort$, and suppose that $u \in \mathsf{Term}^c\left(\Sigma\left(M^k, \overline{\mathsf{x}^B_{\mathsf{Cod}(k)}}\right)\right)$ is of sort $B$ and $\T\left(M^k, \overline{\mathsf{x}^B_{\mathsf{Cod}(k)}}\right) \vdash u = \mathsf{x}_{\mathsf{id}_k}^B$. If $\T\left(M^k\right)$ is non-trivial for the sort $B$, then $u$ contains at least one occurrence of $\mathsf{x}_{\mathsf{id}_k}^B$. \qed   
\end{lem}

\noindent We will also need the following map $\theta^*$, which essentially takes an $\alpha$-restricted term $t$, applies $\theta$ to it, and then erases all of the morphism subscripts from the indeterminates of the form $\x_f^A$ in $\theta(t)$:

\begin{defn}
\label{theta^*}
{\em Let $M \in \PTJmod$ and $A \in \Sigma_\Sort$ and $i \in \ob\J$. We define a map 
\[ \theta^* : \mathsf{Term}^c\left(\SigmaJ\left(M, \mathsf{x}_{A^i}\right)\right)^* \to \bigcup_{k \in \ob\J} \mathsf{Term}^c\left(\Sigma\left(M^k, \mathsf{x}_A\right)\right) \] with the property that if $u : C^k$ for $k \in \ob\J$ and $C \in \Sigma_{\mathsf{Sort}}$, then $\theta^*(u) \in \mathsf{Term}^c\left(\Sigma\left(M^k, \mathsf{x}_A\right)\right)$ with $\theta^*(u) : C$. To define $\theta^*$, we first define for each $k \in \ob\J$ a signature morphism \[ \lambda_k : \Sigma\left(M^k, \overline{\mathsf{x}^A_{\mathsf{Cod}(k)}}\right) \to \Sigma\left(M^k, \mathsf{x}_A\right) \] as follows: $\lambda_k$ is the identity on $\Sigma\left(M^k\right)$, and $\lambda_k\left(\mathsf{x}_f^A\right) := \mathsf{x}_A$ for any $f \in \mathsf{Cod}(k)$. By a slight abuse of notation, we also denote the induced function on closed terms as 
\[ \lambda_k : \mathsf{Term}^c\left(\Sigma\left(M^k, \overline{\mathsf{x}^A_{\mathsf{Cod}(k)}}\right)\right) \to \mathsf{Term}^c\left(\Sigma\left(M^k, \mathsf{x}_A\right)\right). \] Finally, we set
\[ \lambda := \bigcup_{k \in \ob\J} \lambda_k : \bigcup_{k \in \ob\J} \mathsf{Term}^c\left(\Sigma\left(M^k, \overline{\mathsf{x}^A_{\mathsf{Cod}(k)}}\right)\right) \to \bigcup_{k \in \ob\J} \mathsf{Term}^c\left(\Sigma\left(M^k, \mathsf{x}_A\right)\right), \] and we then define $\theta^*$ to be the composite
\[ \mathsf{Term}^c\left(\SigmaJ\left(M, \mathsf{x}_{A^i}\right)\right)^* \xrightarrow{\theta} \bigcup_{k \in \ob\J} \mathsf{Term}^c\left(\Sigma\left(M^k, \overline{\mathsf{x}^A_{\mathsf{Cod}(k)}}\right)\right) \xrightarrow{\lambda} \bigcup_{k \in \ob\J} \mathsf{Term}^c\left(\Sigma\left(M^k, \mathsf{x}_A\right)\right), \] and it is easy to see that $\theta^*$ indeed has the stated property. \qed }
\end{defn}

\noindent Before showing that $\theta^*$ preserves the provability of a certain restricted kind of sequent, we require the following technical concepts. 

\begin{defn}
\label{alphacommuting}
{\em Let $M \in \PTJmod$. 
\begin{itemize}
\item If $u \in \mathsf{Term}^c\left(\SigmaJ\left(M, \mathsf{x}_{A^i}\right)\right)^*$ for some $A \in \Sigma_\Sort$ and $i \in \ob\J$, then we say that $u$ is $i$-\emph{local} if for any subterm $v$ of $u$, there is some sort $C \in \Sigma_\Sort$ such that $v : C^i$. In particular, if $u$ is $i$-local, then $u : B^i$ for some sort $B$, and every $\alpha$-subterm of $u$ has the form $\alpha_f^A(\mathsf{x}_{A^i})$ for some endomorphism $f : i \to i$.   

\item Let $f : j \to i$ have codomain $i$. If $u \in \mathsf{Term}^c\left(\SigmaJ\left(M, \mathsf{x}_{A^i}\right)\right)^*$ is $i$-local, we define \[ u[f] \in \mathsf{Term}^c\left(\SigmaJ\left(M, \mathsf{x}_{A^j}\right)\right)^* \] (note the change from $x_{A^i}$ to $x_{A^j}$) to be the term of the same sort defined as follows:
\begin{itemize}
\item If $u \equiv \mathsf{x}_{A^i} : A^i$, then we set $u[f] := \alpha_f^A\left(\mathsf{x}_{A^j}\right) : A^i$.
\item If $u \equiv \alpha_g^A(\mathsf{x}_{A^i}) : A^i$ for some $g : i \to i$ (since $u$ is $i$-local), then we set $u[f] := \alpha_{g \circ f}^A\left(\mathsf{x}_{A^j}\right) : A^i$. 
\item If $u \equiv c_{B^i, s}^M : B^i$ for some $B \in \Sigma_{\mathsf{Sort}}$ and $s \in M_{B^i}$, then we set $u[f] := u : B^i$.

\item If $u \equiv g^i(u_1, \ldots, u_n) : B^i$ for some function symbol $g : B_1 \times \ldots \times B_n \to B$ in $\Sigma$ and $i$-local terms $u_1, \ldots, u_n \in \mathsf{Term}^c\left(\SigmaJ\left(M, \mathsf{x}_{A^i}\right)\right)^*$ with $u_\ell : B^i_\ell$ for each $1 \leq \ell \leq n$, then we set
\[ u[f] := g^i\left(u_1[f], \ldots, u_n[f]\right) : B^i. \]
\end{itemize}
In general, $u[f]$ will \emph{not} be the same term as $u^f$ from Lemma \ref{firstalphalemma}. 

\item If $u \in \mathsf{Term}\left(\SigmaJ\left(M, \mathsf{x}_{A^i}\right)\right)^*$ is $i$-local with $u : B^i$ for some $B \in \Sigma_\Sort$, then we say that $u$ \emph{commutes generically with} an endomorphism $f : i \to i$ if $\TJ(M, \mathsf{x}_{A^i}) \vdash \alpha_f^B(u) = u[f]$. \qed
\end{itemize}
}
\end{defn}

\noindent For future reference, we note the following obvious result: if $u \in \mathsf{Term}^c\left(\SigmaJ\left(M, \mathsf{x}_{A^i}\right)\right)$ is $\alpha$-restricted and $i$-local and $f : i \to i$ is an endomorphism, then \[ \TJ(M, \mathsf{x}_{A^i}) \vdash u \downarrow \ \ \Longrightarrow \ \ \TJ\left(M, \mathsf{x}_{A^i}\right) \vdash u[f] = u\left[\alpha_f^A(\mathsf{x}_{A^i})/\mathsf{x}_{A^i}\right]. \]

\noindent We can now prove that $\theta^*$ preserves provability of a certain restricted kind of sequent; the proof may be found in \cite[Lemma 5.1.34]{thesis}.

\begin{lem}
\label{theta^*lemma}
Let $M \in \PTJmod$. Let $u, s, t \in \mathsf{Term}^c\left(\SigmaJ\left(M, \mathsf{x}_{A^i}\right)\right)^*$ for some $A \in \Sigma_\Sort$ and $i \in \ob\J$, with $u : C^i$ and $s, t : D^i$ for some $C, D \in \Sigma_{\mathsf{Sort}}$. Suppose that $u \equiv h^i(u_1, \ldots, u_m)$ for some function symbol $h : C_1 \times \ldots \times C_m \to C$ of $\Sigma$ and $i$-local terms $u_1, \ldots, u_m \in \mathsf{Term}^c\left(\SigmaJ\left(M, \mathsf{x}_{A^i}\right)\right)^*$ with $u_\ell : C_\ell^i$ and $\TJ(M, \mathsf{x}_{A^i}) \vdash u_\ell \downarrow$ for each $1 \leq \ell \leq m$, and assume that $u_\ell$ commutes generically with each endomorphism $f : i \to i$ in $\J$. If $\TJ(M, \mathsf{x}_{A^i})$ proves the sequent $u \downarrow \ \vdash s = t$, then $\T\left(M^i, \mathsf{x}_A\right)$ proves the sequent $\theta^*(u) \downarrow \ \vdash \theta^*(s) = \theta^*(t)$. \qed
\end{lem}

\noindent We also have that $\theta^*$ preserves the provability of \emph{equations}; the simple proof may be found in \cite[Lemma 5.1.35]{thesis}. 

\begin{lem}
\label{secondtheta^*lemma}
Let $M \in \PTJmod$ and $A \in \Sigma_\Sort$ and $i \in \ob\J$. For any $s, t \in \mathsf{Term}^c\left(\SigmaJ\left(M, \mathsf{x}_{A^i}\right)\right)^*$ with $s, t : C^k$ for some $k \in \ob\J$ and $C \in \Sigma_{\mathsf{Sort}}$, if $\TJ(M, \mathsf{x}_{A^i})$ proves the sequent $\top \vdash s = t$, then $\T\left(M^k, \mathsf{x}_A\right)$ proves the sequent $\top \vdash \theta^*(s) = \theta^*(t)$. \qed
\end{lem}

\noindent The proofs of the following three lemmas may be found in \cite[Lemmas 5.1.36, 5.1.37, 5.1.38]{thesis}.

\begin{lem}
\label{smallthetalemma}
Let $M \in \PTJmod$, let $u \in \mathsf{Term}^c\left(\SigmaJ\left(M, \mathsf{x}_{A^i}\right)\right)^*$ be $i$-local for some $i \in \ob\J$ and $A \in \Sigma_\Sort$, and let $f \in \mathsf{Cod}(i)$. Then $\theta^*(u) \equiv \theta^*(u[f]) \in \mathsf{Term}^c\left(\Sigma\left(M^i, \mathsf{x}_A\right)\right)$. \qed
\end{lem}

\noindent The following lemma says that $\theta^*$ interacts properly with substitution, provided that the term being substituted commutes generically with certain morphisms of $\J$:

\begin{lem}
\label{substthetalemma}
Let $M \in \PTJmod$, let $u, v \in \mathsf{Term}^c\left(\SigmaJ\left(M, \mathsf{x}_{A^i}\right)\right)^*$ for some $A \in \Sigma_\Sort$ and $i \in \ob\J$ with $v : A^i$, and suppose that $\TJ(M, \mathsf{x}_{A^i}) \vdash u, v \downarrow$. Suppose also that $u, v$ are $i$-local, and that $v$ commutes generically with every endomorphism $f : i \to i$ in $\J$.
Then
\[ \T\left(M^i, \mathsf{x}_A\right) \vdash \theta^*\left(u[v/\mathsf{x}_{A^i}]'\right) = \theta^*(u)[\theta^*(v)/\mathsf{x}_A], \] where $u[v/\mathsf{x}_{A^i}]'$ is the $\alpha$-restricted variant of $u[v/\mathsf{x}_{A^i}]$ from Lemma \ref{mainalphalemma}. \qed
\end{lem}

\noindent We will need the following technical lemma to prove that each group homomorphism $\beta_M$ is surjective: 

\begin{lem}
\label{thetacommuteswiththeorymorphism}
Let $M \in \PTJmod$ and $A \in \Sigma_\Sort$ and $i \in \ob\J$, and let $u \in \mathsf{Term}^c\left(\SigmaJ\left(M, \mathsf{x}_{A^i}\right)\right)^*$ be an $\alpha$-restricted, $i$-local term of sort $B^i$ for some $B \in \Sigma_{\mathsf{Sort}}$ with $\TJ(M, \mathsf{x}_{A^i}) \vdash u \downarrow$. Let $f : i \to \ell$ be an arbitrary morphism of $\J$ with domain $i$. Then $\alpha_f^B(u)$ has an $\alpha$-restricted variant $\alpha_f^B(u)'$ by Lemma \ref{mainalphalemma}, and $\alpha_f^B(u)' : B^\ell$, so that $\theta^*\left(\alpha_f^B(u)'\right) \in \mathsf{Term}^c\left(\Sigma\left(M^\ell, \mathsf{x}_A\right)\right)$. And $\theta^*(u) \in \mathsf{Term}^c\left(\Sigma\left(M^i, \mathsf{x}_A\right)\right)$, so that $\rho_{f^M}^A(\theta^*(u)) \in \mathsf{Term}^c\left(\Sigma\left(M^\ell, \mathsf{x}_A\right)\right)$, where $\rho_{f^M}^A : \T\left(M^i, \mathsf{x}_A\right) \to \T\left(M^\ell,\mathsf{x}_A\right)$ is the theory morphism induced by the $\Sigma$-morphism $f^M := F^M(f) : M^i \to M^\ell$. Then $\T\left(M^\ell, \mathsf{x}_A\right) \vdash \theta^*\left(\alpha_f^B(u)'\right) = \rho_{f^M}^{A}(\theta^*(u))$. \qed
\end{lem}

\noindent We will also require the following technical results regarding the map $\theta$, whose proofs involve straightforward inductions on terms:

\begin{lem}
\label{theta^flemma}
Let $M \in \PTJmod$ and $A \in \Sigma_\Sort$ and $i \in \ob\J$, and let $v \in \mathsf{Term}^c\left(\SigmaJ\left(M, \mathsf{x}_{A^i}\right)\right)^*$ be of sort $B^k$ for some $B \in \Sigma_{\mathsf{Sort}}$ and $k \in \ob\J$. Fix an endomorphism $f : k \to k$ in $\J$. By Lemma \ref{firstalphalemma}, there is a term $v^f \in \mathsf{Term}^c\left(\SigmaJ\left(M, \mathsf{x}_{A^i}\right)\right)^*$ with $v^f : B^k$. Then $\theta(v), \theta\left(v^f\right) \in \mathsf{Term}^c\left(\Sigma\left(M^k, \overline{\mathsf{x}^A_{\mathsf{Cod}(k)}}\right)\right)$, and for any $g \in \mathsf{Cod}(k)$,  
$\mathsf{x}_g^A$ occurs in $\theta(v)$ iff $\mathsf{x}_{f \circ g}^A$ occurs in $\theta\left(v^f\right)$. \qed 
\end{lem}

\begin{lem}
\label{lastthetalemma}
Let $M \in \PTJmod$, let $u \in \mathsf{Term}^c\left(\SigmaJ\left(M, \mathsf{x}_{B^k}\right)\right)$ for some $B \in \Sigma_{\mathsf{Sort}}$ and $k \in \ob\J$, and suppose that $u$ is $\alpha$-restricted and $k$-local. Then it is easy to see that every indeterminate in $\theta(u) \in \mathsf{Term}\left(\Sigma\left(M^k, \overline{\mathsf{x}^B_{\mathsf{Cod}(k)}}\right)\right)$ has the form $\mathsf{x}_f^B$ for some endomorphism $f : k \to k$.

Suppose that the indeterminates occurring in $\theta(u)$ are $\mathsf{x}_{f_1}^B, \ldots, \mathsf{x}_{f_n}^B$, with $f_1, \ldots, f_n : k \to k$. Then for any $v \in \mathsf{Term}^c\left(\SigmaJ\left(M, \mathsf{x}_{A^i}\right)\right)^*$ for some $A \in \Sigma_\Sort$ and $i \in \ob\J$ with $v : B^k$, we know that $u[v/\mathsf{x}_{B^k}] \in \mathsf{Term}^c\left(\SigmaJ\left(M, \mathsf{x}_{A^i}\right)\right)$ has an $\alpha$-restricted variant $u[v/\mathsf{x}_{B^k}]' \in \mathsf{Term}^c\left(\SigmaJ\left(M, \mathsf{x}_{A^i}\right)\right)^*$ by Lemma \ref{mainalphalemma}. We then have
\[ \theta\left(u[v/\mathsf{x}_{B^k}]'\right) \equiv \theta(u)\left[\theta\left(v^{f_1}\right)/\mathsf{x}_{f_1}^B, \ldots, \theta\left(v^{f_n}\right)/\mathsf{x}_{f_n}^B\right] \in \mathsf{Term}^c\left(\Sigma\left(M^k, \overline{\mathsf{x}^A_{\mathsf{Cod}(k)}}\right)\right) \]
(recall from Lemma \ref{firstalphalemma} that, for each $1 \leq i \leq n$, $v^{f_i} \in \mathsf{Term}^c\left(\SigmaJ\left(M, \mathsf{x}_{A^i}\right)\right)^*$ is a term of sort $B^{\mathsf{cod}(f_i)} = B^k$). \qed
\end{lem}
\noindent Finally, we require the following notion of `$\alpha$-free variant':

\begin{defn}
\label{alphafreevariant}
{\em Let $M \in \PTJmod$ and $A \in \Sigma_\Sort$ and $i \in \ob\J$. For any $u \in \mathsf{Term}^c\left(\SigmaJ\left(M, \mathsf{x}_{A^i}\right)\right)^*$ that is $i$-local, we define a term \[ u^{-\alpha} \in \mathsf{Term}^c\left(\SigmaJ\left(M, \mathsf{x}_{A^i}\right)\right)^* \] of the same sort, which we call the $\alpha$-\emph{free variant} of $u$:

\begin{itemize}

\item If $u \equiv \mathsf{x}_{A^i} : A^i$, then $u^{-\alpha} := u : A^i$.

\item If $u \equiv \alpha_f^A(\mathsf{x}_{A^i}) : A^i$ for some endomorphism $f : i \to i$ (since $u$ is $i$-local), then $u^{-\alpha} := \mathsf{x}_{A^i} : A^i$.

\item If $u \equiv c_{B^i, s}^{M} : B^i$ for some $B \in \Sigma_{\mathsf{Sort}}$ and $s \in M_{B^i}$, then $u^{-\alpha} := u : B^i$.

\item If $u \equiv g^i(u_1, \ldots, u_n) : B^i$ for some function symbol $g : B_1 \times \ldots \times B_n \to B$ of $\Sigma$ and $i$-local terms $u_i \in \mathsf{Term}^c\left(\SigmaJ\left(M, \mathsf{x}_{A^i}\right)\right)^*$ of sort $B_j^i$ for each $1 \leq j \leq n$, then 
$u^{-\alpha} := g^i\left(u_1^{-\alpha}, \ldots, u_n^{-\alpha}\right) : B^i$. \qed
\end{itemize}
}
\end{defn}

\noindent Essentially, the $\alpha$-free variant $u^{-\alpha}$ is obtained from $u$ by `erasing' all of the $\alpha$-function symbols in $u$ (and since $u$ is $i$-local, it is possible to do this and obtain a well-defined term of the same sort). We then have the following technical lemma, whose proof is a straightforward induction on terms:

\begin{lem}
\label{alphafreevariantlemma}
Let $M \in \PTJmod$ and $A \in \Sigma_{\mathsf{Sort}}$ and $i \in \ob\J$. For any $\alpha$-restricted and $i$-local $u \in \mathsf{Term}^c\left(\SigmaJ\left(M, \mathsf{x}_{A^i}\right)\right)^*$ we have $\rho_{M^i}^A(\theta^*(u)) \equiv u^{-\alpha}$, where $\rho_{M^i}^A : \Sigma\left(M^i, \mathsf{x}_A\right) \to \SigmaJ(M, \mathsf{x}_{A^i})$ is the signature morphism from Definition \ref{sigmorphismofM^i}. \qed
\end{lem}

\noindent We can now finally prove that the group homomorphism $\beta_M : \left(\prod_i G_{\T}\left(M^i\right)\right)^\J \times \mathsf{Aut}(\mathsf{Id}_\J)^M \to G_{\TJ}(M)$ is injective:

\begin{prop}
\label{betaisinjective}
For any $M \in \PTJmod$, the group homomorphism \[ \beta_M : \left(\prod_i G_{\T}\left(M^i\right)\right)^\J \times \mathsf{Aut}(\mathsf{Id}_\J)^M \to G_{\TJ}(M) \] is injective. 
\end{prop}

\begin{proof}
Let $\gamma = (\gamma_i)_i \in \left(\prod_i G_{\T}\left(M^i\right)\right)^\J$ with $\gamma_i = \left(\left[s_C^i\right]\right)_{C \in \Sigma}$ for each $i \in \ob\J$, let $\psi \in \mathsf{Aut}(\mathsf{Id}_\J)^M$, and suppose that $\beta_M(\gamma, \psi) = ([\mathsf{x}_A])_{A \in \SigmaJ}$, the unit element of the group $G_{\TJ}(M)$. We must show that each $\gamma_i$ is the unit of the group $G_{\T}\left(M^i\right)$, i.e. we must show that $\gamma_i = ([\mathsf{x}_C])_C$ for all $i \in \ob\J$, and we must also show that $\psi$ is the unit element of $\mathsf{Aut}(\mathsf{Id}_\J)^M$. So fix $i \in \ob\J$ and $B \in \Sigma_{\mathsf{Sort}}$, and suppose first that $\T\left(M^i\right)$ is non-trivial for the sort $B$. The hypothesis implies in particular that $\beta_M(\gamma, \psi)_{B^i} = [\mathsf{x}_{B^i}]$, i.e. that \[ \left[\rho_{M^i}^B\left(s_B^i\right)\left[\alpha_{\psi_B(i)}^B(\mathsf{x}_{B^i})/\mathsf{x}_{B^i}\right]\right] = [\mathsf{x}_{B^i}] \in M\la \mathsf{x}_{B^i} \ra_{B^i}, \] which means that \[ \TJ(M, \mathsf{x}_{B^i}) \vdash \rho_{M^i}^B\left(s_B^i\right)\left[\alpha_{\psi_B(i)}^B(\mathsf{x}_{B^i})/\mathsf{x}_{B^i}\right] = \mathsf{x}_{B^i}. \] 
We first show that $\psi_B(i) = \mathsf{id}_i$. Note that $\rho_{M^i}^B\left(s_B^i\right)\left[\alpha_{\psi_B(i)}^B(\mathsf{x}_{B^i})/\mathsf{x}_{B^i}\right] \in \mathsf{Term}\left(\SigmaJ\left(M, \mathsf{x}_{B^i}\right)\right)$ is $\alpha$-restricted (because $\rho_{M^i}^B\left(s_B^i\right)$ does not contain any $\alpha$-function symbols). Then by Proposition \ref{thetaprop}, we obtain
\[ \T\left(M^i, \overline{\mathsf{x}^B_{\mathsf{Cod}(i)}}\right) \vdash \theta\left(\rho_{M^i}^B\left(s_B^i\right)\left[\alpha_{\psi_B(i)}^B(\mathsf{x}_{B^i})/\mathsf{x}_{B^i}\right]\right) = \theta(\mathsf{x}_{B^i}). \] 
Now, it is trivial to see that the only indeterminate that occurs in $\theta\left(\rho_{M^i}^B\left(s_B^i\right)\right) \in \mathsf{Term}\left(\Sigma\left(M^i, \overline{\mathsf{x}^B_{\mathsf{Cod}(i)}}\right)\right)$ is $\mathsf{x}_{\mathsf{id}_i}^B$. Then since $\rho_{M^i}^B\left(s_B^i\right) \in \mathsf{Term}\left(\SigmaJ\left(M, \mathsf{x}_{B^i}\right)\right)$ is also $\alpha$-restricted and clearly $i$-local, and since \linebreak $\rho_{M^i}^B\left(s_B^i\right)\left[\alpha_{\psi_B(i)}^B(\mathsf{x}_{B^i})/\mathsf{x}_{B^i}\right]$ is $\alpha$-restricted, it follows by Lemma \ref{lastthetalemma} that 
\[ \theta\left(\rho_{M^i}^B\left(s_B^i\right)\left[\alpha_{\psi_B(i)}^B(\mathsf{x}_{B^i})/\mathsf{x}_{B^i}\right]\right) \equiv \theta\left(\rho_{M^i}^B\left(s_B^i\right)\right)\left[\theta\left(\alpha_{\psi_B(i)}^B(\mathsf{x}_{B^i})^{\mathsf{id}_i}\right)/\mathsf{x}_{\mathsf{id}_i}^B\right], \] and hence 
\[ \theta\left(\rho_{M^i}^B\left(s_B^i\right)\left[\alpha_{\psi_B(i)}^B(\mathsf{x}_{B^i})/\mathsf{x}_{B^i}\right]\right) \equiv \theta\left(\rho_{M^i}^B\left(s_B^i\right)\right)\left[\mathsf{x}_{\psi_B(i)}^B/\mathsf{x}_{\mathsf{id}_i}^B\right], \] because
\[ \theta\left(\alpha_{\psi_B(i)}^B(\mathsf{x}_{B^i})^{\mathsf{id}_i}\right) \equiv \theta\left(\alpha_{\mathsf{id}_i \circ \psi_B(i)}^B(\mathsf{x}_{B^i})\right) \equiv \theta\left(\alpha_{\psi_B(i)}^B(\mathsf{x}_{B^i})\right) \equiv \mathsf{x}_{\psi_B(i)}^B. \] Also, it is easy to see that $\theta\left(\rho_{M^i}^B\left(s_B^i\right)\right) \equiv s_B^i\left[\mathsf{x}_{\mathsf{id}_i}^B/\mathsf{x}_B\right]$, and so we obtain
\[ \theta\left(\rho_{M^i}^B\left(s_B^i\right)\left[\alpha_{\psi_B(i)}^B(\mathsf{x}_{B^i})/\mathsf{x}_{B^i}\right]\right) \equiv s_B^i\left[\mathsf{x}_{\psi_B(i)}^B/\mathsf{x}_B\right]. \] Since $\theta(\mathsf{x}_{B^i}) \equiv \mathsf{x}_{\mathsf{id}_i}^B$, from the fact that 
\[ \T\left(M^i, \overline{\mathsf{x}^B_{\mathsf{Cod}(i)}}\right) \vdash \theta\left(\rho_{M^i}^B\left(s_B^i\right)\left[\alpha_{\psi_B(i)}^B(\mathsf{x}_{B^i})/\mathsf{x}_{B^i}\right]\right) = \theta(\mathsf{x}_{B^i}) \] we finally deduce that
\[ \T\left(M^i, \overline{\mathsf{x}^B_{\mathsf{Cod}(i)}}\right) \vdash s_B^i\left[\mathsf{x}_{\psi_B(i)}^B/\mathsf{x}_B\right] = \mathsf{x}_{\mathsf{id}_i}^B. \]
Since $\T\left(M^i\right)$ is non-trivial for the sort $B$, it then follows from Lemma \ref{occurrenceofidentityconstant} that $\mathsf{x}_{\mathsf{id}_i}^B$ occurs in $s_B^i\left[\mathsf{x}_{\psi_B(i)}^B/\mathsf{x}_{B}\right]$, and moreover it follows by \cite[Lemma 2.2.56]{thesis} that $\x_B$ occurs in $s_B^i$, so that $\x_{\psi_B(i)}^B$ occurs in $s_B^i\left[\mathsf{x}_{\psi_B(i)}^B/\mathsf{x}_{B}\right]$. But this forces $\mathsf{x}_{\psi_B(i)}^B \equiv \mathsf{x}_{\mathsf{id}_i}^B$, because $\mathsf{x}_{\psi_B(i)}^B$ is the only indeterminate occurring in $s_B^i\left[\mathsf{x}_{\psi_B(i)}^B/\mathsf{x}_{B}\right]$, and hence we deduce $\psi_B(i) = \mathsf{id}_i$, as desired. So we may now infer that
\[ \T\left(M^i, \overline{\mathsf{x}^B_{\mathsf{Cod}(i)}}\right) \vdash s_B^i\left[\mathsf{x}_{\mathsf{id}_i}^B/\mathsf{x}_{B^i}\right] = \mathsf{x}_{\mathsf{id}_i}^B. \] Since $\lambda_i : \T\left(M^i, \overline{\mathsf{x}^B_{\mathsf{Cod}(i)}}\right) \to \T\left(M^i, \mathsf{x}_B\right)$ is a theory morphism by the proof of Lemma \ref{theta^*lemma}, we then obtain
\[ \T\left(M^i, \mathsf{x}_B\right) \vdash \lambda_i\left(s_B^i\left[\mathsf{x}_{\mathsf{id}_i}^B/\mathsf{x}_{B^i}\right]\right) = \lambda_i\left(\mathsf{x}_{\mathsf{id}_i}^B\right). \] Since $\lambda_i$ is the identity except on the indeterminates of $\Sigma\left(M^i, \overline{\mathsf{x}^B_{\mathsf{Cod}(i)}}\right)$, it then follows that $\T\left(M^i, \mathsf{x}_B\right) \vdash s_B^i = \mathsf{x}_B$. This shows that if $\T\left(M^i\right)$ is non-trivial for the sort $B$, then $\gamma_i^B = \left[s_B^i\right] = [\mathsf{x}_B]$ and $\psi_B(i) = \mathsf{id}_i$, so that $\psi_B$ is the identity natural automorphism of $\mathsf{Id}_{\J^M_B}$ and hence $\psi$ is the unit element of $\mathsf{Aut}(\mathsf{Id}_\J)^M$.  

It remains to show that if $\T\left(M^i\right)$ is \emph{trivial} for the sort $B$, then $\gamma_i^B = \left[s_B^i\right] = [\mathsf{x}_B]$ in this case as well. But if $\T\left(M^i\right)$ is trivial for the sort $B$, then $\T\left(M^i, \mathsf{x}_B\right)$ is trivial for the sort $B$ as well, which implies that $\T(M^i, \mathsf{x}_{B}) \vdash s_B^i = \mathsf{x}_{B}$, because $s_B^i, \mathsf{x}_B : B$. This completes the proof that each $\gamma_i$ is the unit element of $G_\T\left(M^i\right)$, which completes the proof that $\beta_M$ is injective. 
\end{proof}

\noindent Since we will need to impose two assumptions on $\T$ in order to prove that each $\beta_M$ is surjective, let us now record what we have proven so far:

\begin{prop}
\label{arbitrarytheoryinjectivegrouphom}
Let $\T$ be an \textbf{arbitrary} quasi-equational theory and $\J$ a small index category. Then for any $M \in \PTmod$, there is an injective group homomorphism $\beta_M : \left(\prod_i G_{\T}\left(M^i\right)\right)^\J \times \mathsf{Aut}(\mathsf{Id}_\J)^M \to G_{\TJ}(M)$. \qed
\end{prop}

\noindent To prove that each $\beta_M$ is surjective, we will need to assume that $\T$ satisfies the conditions in the following two definitions: 

\begin{defn}
\label{singleindeterminateisotropy}
{\em Let $\T$ be a quasi-equational theory over a signature $\Sigma$. We say that $\T$ has \emph{single-indeterminate isotropy} if for any $M \in \PTmod$ and $([s_C])_{C \in \Sigma} \in G_{\T}(M)$ and $C \in \Sigma_{\mathsf{Sort}}$, we can assume without loss of generality that $s_C$ contains \emph{exactly one occurrence} of $\x_C$. \qed
}
\end{defn}

\noindent In other words, $\T$ has single-indeterminate isotropy if every component of every element of isotropy of every $\T$-model can be assumed to have exactly one occurrence of the indeterminate. This is not an \emph{overly} restrictive condition, because (apart from the theories of racks and quandles, see \cite{racks}) every example theory considered in \cite[Section 4]{MFPS}, \cite[Chapter 3]{thesis}, and \cite{FSCD} has single-indeterminate isotropy (in particular, the theory of groups does). 

\begin{defn}
\label{singlesortednontotaloperations}
{\em Let $\T$ be a quasi-equational theory over a signature $\Sigma$. If $g : A_1 \times \ldots \times A_n \to A$ is a function symbol of $\Sigma$, then we say that $g$ is \emph{totally defined in} $\T$ if $\T$ proves the sequent $\top \vdash^{y_1, \ldots, y_n} g(y_1, \ldots, y_n) \downarrow$, where $y_1, \ldots, y_n$ are pairwise distinct variables with $y_i : A_i$ for each $1 \leq i \leq n$. 

We then say that $\T$ has \emph{single-sorted non-total operations} if for any function symbol $g : A_1 \times \ldots \times A_n \to A$ of $\Sigma$ that is \emph{not} totally defined in $\T$, we have $A_i = A$ for each $1 \leq i \leq n$. \qed 
}
\end{defn} 

\noindent Again, this is not an \emph{overly} restrictive condition, because every example theory considered in the previous sources has single-sorted non-total operations (and in particular the theory of groups). In Remark \ref{notsingleindeterminateisotropy} below, we will indicate how the failure of $\T$ to satisfy the conditions in Definitions \ref{singleindeterminateisotropy} and \ref{singlesortednontotaloperations} can result in the failure of the surjectivity of $\beta_M$. We can now prove:  

\begin{prop}
\label{betaissurjective}
Let $\T$ be a quasi-equational theory with single-indeterminate isotropy and single-sorted non-total operations, and let $\J$ be a small index category. For any $M \in \PTJmod$, the group homomorphism $\beta_M : \left(\prod_i G_{\T}\left(M^i\right)\right)^\J \times \mathsf{Aut}(\mathsf{Id}_\J)^M \to G_{\TJ}(M)$ is surjective.
\end{prop}

\begin{proof}
Let $([s_{C^i}])_{i \in \ob\J, C \in \Sigma} \in G_{\TJ}(M)$. So for any $i \in \ob\J$ and $C \in \Sigma_{\mathsf{Sort}}$, we know that \linebreak $s_{C^i} \in \mathsf{Term}^c\left(\SigmaJ\left(M, \mathsf{x}_{C^i}\right)\right)$ is a closed term of sort $C^i$ with $\TJ(M, \mathsf{x}_{C^i}) \vdash s_{C^i} \downarrow$. Moreover, the $\SigmaJ_{\mathsf{Sort}}$-indexed sequence $([s_{C^i}])_{i, C}$ is invertible, commutes generically with all function symbols of $\SigmaJ$, and reflects definedness. By Lemma \ref{mainalphalemma}, we may assume without loss of generality that for each $i \in \ob\J$ and $C \in \Sigma_{\mathsf{Sort}}$, the term $s_{C^i} \in \mathsf{Term}^c\left(\SigmaJ\left(M, \mathsf{x}_{C^i}\right)\right)_{C^i}$ is $\alpha$-restricted, i.e. $s_{C^i} \in \mathsf{Term}^c\left(\SigmaJ\left(M, \mathsf{x}_{C^i}\right)\right)^*$. Then for every $C \in \Sigma_{\mathsf{Sort}}$ and $i \in \ob\J$, it follows by Lemma \ref{secondtheta^*lemma} that $\theta^*\left(s_{C^i}\right) \in \mathsf{Term}^c\left(\Sigma\left(M^i,\mathsf{x}_C\right)\right)_C$ and $\T\left(M^i, \mathsf{x}_C\right) \vdash \theta^*\left(s_{C^i}\right) \downarrow$,  so that $\left[\theta^*\left(s_{C^i}\right)\right] \in M^i \la \mathsf{x}_C \ra_C$. 
We now define $\gamma \in \left(\prod_i G_{\T}\left(M^i\right)\right)^\J$. For any $i \in \ob\J$, we define \[ \gamma_i \in G_{\T}\left(M^i\right) \subseteq \prod_{C \in \Sigma} M^i\la \mathsf{x}_C\ra_C \] as follows: for any $C \in \Sigma_{\mathsf{Sort}}$, we set
\[ \gamma_i^C := \left[\theta^*\left(s_{C^i}\right)\right] \in M^i \la \mathsf{x}_C \ra_C. \] Our goal is now to show that $\gamma \in \left(\prod_i G_{\T}\left(M^i\right)\right)^\J$, which we achieve via the following series of claims.
\begin{claim}
For any $i \in \ob\J$, $\gamma_i$ is invertible.
\end{claim}
\begin{proof}
Let $B \in \Sigma_{\mathsf{Sort}}$. We must show that there is some $\left[t^i_B\right] \in M^i \la \mathsf{x}_B \ra_B$ with \[ \T\left(M^i, \mathsf{x}_B\right) \vdash \theta^*\left(s_{B^i}\right)\left[t^i_B/\mathsf{x}_B\right] = \mathsf{x}_B = t^i_B\left[\theta^*\left(s_{B^i}\right)/\mathsf{x}_B\right]. \] Since $([s_{C^j}])_{j, C} \in G_{\TJ}(M)$, there is some $\left[s_{B^i}^{-1}\right] \in M\la \mathsf{x}_{B^i} \ra_{B^i}$ with \[ \TJ(M, \mathsf{x}_{B^i}) \vdash s_{B^i}\left[s_{B^i}^{-1}/\mathsf{x}_{B^i}\right] = \mathsf{x}_{B^i} = s_{B^i}^{-1}[s_{B^i}/\mathsf{x}_{B^i}]. \] Since $s_{B^i}^{-1} \in \mathsf{Term}^c\left(\SigmaJ\left(M, \mathsf{x}_{B^i}\right)\right)_{B^i}$ and $\TJ(M, \mathsf{x}_{B^i}) \vdash s_{B^i}^{-1} \downarrow$, we may assume without loss of generality that $s_{B^i}^{-1}$ is $\alpha$-restricted by Lemma \ref{mainalphalemma}. So $\theta^*\left(s_{B^i}^{-1}\right) \in \mathsf{Term}^c\left(\Sigma\left(M^i, \mathsf{x}_B\right)\right)$ has the property that $\T\left(M^i, \mathsf{x}_B\right) \vdash \theta^*\left(s_{B^i}^{-1}\right) \downarrow$ by Lemma \ref{secondtheta^*lemma}. Hence, we have $\left[\theta^*\left(s_{B^i}^{-1}\right)\right] \in M^i \la \mathsf{x}_B \ra_B$, so we set \[ \left[t_B^i\right] := \left[\theta^*\left(s_{B^i}^{-1}\right)\right]. \] Since $([s_{C^j}])_{j, C} \in G_{\TJ}(M)$, it follows that $([s_{C^j}])_{j, C}$ commutes generically with the function symbol $\alpha_f^B : B^i \to B^i$ for any endomorphism $f : i \to i$ in $\J$. In other words, for any endomorphism $f : i \to i \in \J$ we have
\[ \TJ(M, \mathsf{x}_{B^i}) \vdash \alpha_f^B\left(s_{B^i}\right) = s_{B^i}\left[\alpha_f^B(\mathsf{x}_{B^i})/\mathsf{x}_{B^i}\right] = s_{B^i}[f], \] the latter equality being provable by the remark after Definition \ref{alphacommuting} (since $s_{B^i}$ is $i$-local, because it is $\alpha$-restricted and of sort $B^i$ and only contains the indeterminate $\x_{B^i}$). In the same way, we also have 
\[ \TJ(M, \mathsf{x}_{B^i}) \vdash \alpha_f^B\left(s_{B^i}^{-1}\right) = s_{B^i}^{-1}\left[\alpha_f^B(\mathsf{x}_{B^i})/\mathsf{x}_{B^i}\right] = s_{B^i}^{-1}[f] \] for any endomorphism $f : i \to i$ in $\J$. Since $s_{B^i}^{-1}$ is $i$-local, this means that $s_{B^i}^{-1}$ also commutes generically with every endomorphism $f : i \to i$. Then by Lemma \ref{substthetalemma} we obtain
\[ \T\left(M^i, \mathsf{x}_B\right) \vdash \theta^*\left(s_{B^i}\left[s_{B^i}^{-1}/\mathsf{x}_{B^i}\right]'\right) = \theta^*\left(s_{B^i}\right)\left[\theta^*\left(s_{B^i}^{-1}\right)/\mathsf{x}_B\right], \tag{$*$} \] where $s_{B^i}\left[s_{B^i}^{-1}/\mathsf{x}_{B^i}\right]'$ is the $\alpha$-restricted variant of $s_{B^i}\left[s_{B^i}^{-1}/\mathsf{x}_{B^i}\right]$ with \[ \TJ(M, \mathsf{x}_{B^i}) \vdash s_{B^i}\left[s_{B^i}^{-1}/\mathsf{x}_{B^i}\right] = s_{B^i}\left[s_{B^i}^{-1}/\mathsf{x}_{B^i}\right]' \] by Lemma \ref{mainalphalemma}. From this latter equation and the defining property of $s_{B^i}^{-1}$ we obtain \[ \TJ(M, \mathsf{x}_{B^i}) \vdash s_{B^i}\left[s_{B^i}^{-1}/\mathsf{x}_{B^i}\right]' = \mathsf{x}_{B^i}. \] By Lemma \ref{secondtheta^*lemma} we then obtain \[ \T\left(M^i, \mathsf{x}_B\right) \vdash \theta^*\left(s_{B^i}\left[s_{B^i}^{-1}/\mathsf{x}_{B^i}\right]'\right) = \theta^*(\mathsf{x}_{B^i}) \equiv \mathsf{x}_B. \] Combining this with ($*$), we finally have $\T\left(M^i, \mathsf{x}_B\right) \vdash  \theta^*\left(s_{B^i}\right)\left[\theta^*(s_{B^i}^{-1})/\mathsf{x}_B\right] = \mathsf{x}_B$, as desired. The converse equality is proven analogously, which completes the proof that $\gamma_i$ is invertible. 
\end{proof}

\begin{claim}
For any $i \in \ob\J$, $\gamma_i$ commutes generically with all function symbols of $\Sigma$.
\end{claim} 
\begin{proof}
Let $g : B_1 \times \ldots \times B_n \to B$ be a function symbol of $\Sigma$. We must show that the sequent \[ g(\mathsf{x}_{B_1}, \ldots, \mathsf{x}_{B_n}) \downarrow \ \vdash \theta^*\left(s_{B^i}\right)\left[g(\mathsf{x}_{B_1}, \ldots, \mathsf{x}_{B_n})/\mathsf{x}_B\right] = g\left(\theta^*\left(s_{B^i_1}\right), \ldots, \theta^*\left(s_{B^i_n}\right)\right) \] is provable in the theory $\T\left(M^i, \mathsf{x}_{B_1}, \ldots, \mathsf{x}_{B_n}\right)$ (as in the proof of Proposition \ref{grouphomexists}, we technically need to ensure that the indeterminates on the right side of the above equation are pairwise distinct (see Definition \ref{commutesgenericallydefn}), but we will ignore this subtlety here and elsewhere in the proof to increase readability). Since $([s_{C^j}])_{j, C} \in G_{\TJ}(M)$, we know that $([s_{C^j}])_{j, C}$ commutes generically with the function symbol $g^i : B^i_1 \times \ldots B^i_n \to B^i$ of $\SigmaJ$, which means that the sequent
\[ g^i\left(\mathsf{x}_{B^i_1}, \ldots, \mathsf{x}_{B^i_n}\right) \downarrow \ \vdash s_{B^i}\left[g^i\left(\mathsf{x}_{B^i_1}, \ldots, \mathsf{x}_{B^i_n}\right)/\mathsf{x}_{B^i}\right] = g^i\left(s_{B^i_1}, \ldots, s_{B^i_n}\right) \] is provable in the theory $\TJ\left(M, \mathsf{x}_{B^i_1}, \ldots, \mathsf{x}_{B^i_n}\right)$. Since the terms $s_{B^i_1}, \ldots, s_{B^i_n}$ are all $\alpha$-restricted, it easily follows that the terms $g^i\left(\mathsf{x}_{B^i_1}, \ldots, \mathsf{x}_{B^i_n}\right)$ and $g^i\left(s_{B^i_1}, \ldots, s_{B^i_n}\right)$ are $\alpha$-restricted. By a simple extension of Lemma \ref{mainalphalemma}, there is an $\alpha$-restricted variant $s_{B^i}\left[g^i\left(\mathsf{x}_{B^i_1}, \ldots, \mathsf{x}_{B^i_n}\right)/\mathsf{x}_{B^i}\right]'$ of $s_{B^i}\left[g^i\left(\mathsf{x}_{B^i_1}, \ldots, \mathsf{x}_{B^i_n}\right)/\mathsf{x}_{B^i}\right]$ such that the sequent
\[ g^i\left(\mathsf{x}_{B^i_1}, \ldots, \mathsf{x}_{B^i_n}\right) \downarrow \ \vdash s_{B^i}\left[g^i\left(\mathsf{x}_{B^i_1}, \ldots, \mathsf{x}_{B^i_n}\right)/\mathsf{x}_{B^i}\right] = s_{B^i}\left[g^i\left(\mathsf{x}_{B^i_1}, \ldots, \mathsf{x}_{B^i_n}\right)/\mathsf{x}_{B^i}\right]' \] is provable in $\TJ\left(M, \mathsf{x}_{B^i_1}, \ldots, \mathsf{x}_{B^i_n}\right)$. For each $1 \leq m \leq n$, the indeterminate $\x_{B^i_m}$ is clearly $i$-local, we have $\TJ\left(M, \mathsf{x}_{B^i_1}, \ldots, \mathsf{x}_{B^i_n}\right) \vdash \x_{B^i_m} \downarrow$, and for each endomorphism $f : i \to i$ in $\J$ we have
\[ \TJ\left(M, \mathsf{x}_{B^i_1}, \ldots, \mathsf{x}_{B^i_n}\right) \vdash \alpha_f^{B_m}\left(\x_{B^i_m}\right) = \x_{B^i_m}[f], \] which means that $\x_{B^i_m}$ commutes generically with $f$. Then by a simple extension of Lemma \ref{theta^*lemma} and the assumption that \[ g^i\left(\mathsf{x}_{B^i_1}, \ldots, \mathsf{x}_{B^i_n}\right) \downarrow \ \vdash s_{B^i}\left[g^i\left(\mathsf{x}_{B^i_1}, \ldots, \mathsf{x}_{B^i_n}\right)/\mathsf{x}_{B^i}\right] = g^i\left(s_{B^i_1}, \ldots, s_{B^i_n}\right) \] is provable in the theory $\TJ\left(M, \mathsf{x}_{B^i_1}, \ldots, \mathsf{x}_{B^i_n}\right)$, we obtain that \[ g(\mathsf{x}_{B_1}, \ldots, \mathsf{x}_{B_n}) \downarrow \ \vdash \theta^*\left(s_{B^i}\left[g^i\left(\mathsf{x}_{B^i_1}, \ldots, \mathsf{x}_{B^i_n}\right)/\mathsf{x}_{B^i}\right]'\right) = g\left(\theta^*\left(s_{B^i_1}\right), \ldots, \theta^*\left(s_{B^i_n}\right)\right) \] is provable in the theory $\T\left(M^i, \mathsf{x}_{B_1}, \ldots, \mathsf{x}_{B_n}\right)$. Now, we will be done if we can show that
$\T\left(M^i, \mathsf{x}_{B_1}, \ldots, \mathsf{x}_{B_n}\right)$ proves the sequent 
\[ g(\mathsf{x}_{B_1}, \ldots, \mathsf{x}_{B_n}) \downarrow \ \vdash \theta^*\left(s_{B^i}\left[g^i\left(\mathsf{x}_{B^i_1}, \ldots, \mathsf{x}_{B^i_n}\right)/\mathsf{x}_{B^i}\right]'\right) = \theta^*\left(s_{B^i}\right)\left[g\left(\mathsf{x}_{B_1}, \ldots, \mathsf{x}_{B_n}\right)/\mathsf{x}_B\right]. \] Since $\theta^*\left(g^i\left(\mathsf{x}_{B^i_1}, \ldots, \mathsf{x}_{B^i_n}\right)\right) \equiv g(\mathsf{x}_{B_1}, \ldots, \mathsf{x}_{B_n})$, it suffices by a simple extension of Lemma \ref{substthetalemma} to show that $g^i\left(\mathsf{x}_{B^i_1}, \ldots, \mathsf{x}_{B^i_n}\right)$ commutes generically with every endomorphism $f : i \to i$ in $\J$, i.e. it suffices to show that the sequent
\[ g^i\left(\mathsf{x}_{B^i_1}, \ldots, \mathsf{x}_{B^i_n}\right) \downarrow \ \vdash \alpha_f^B\left(g^i\left(\mathsf{x}_{B^i_1}, \ldots, \mathsf{x}_{B^i_n}\right)\right) = g^i\left(\mathsf{x}_{B^i_1}, \ldots, \mathsf{x}_{B^i_n}\right)[f], \] i.e. the sequent
\[ g^i\left(\mathsf{x}_{B^i_1}, \ldots, \mathsf{x}_{B^i_n}\right) \downarrow \ \vdash \alpha_f^B\left(g^i\left(\mathsf{x}_{B^i_1}, \ldots, \mathsf{x}_{B^i_n}\right)\right) = g^i\left(\alpha_f^{B_1}\left(\mathsf{x}_{B^i_1}\right), \ldots, \alpha_f^B\left(\mathsf{x}_{B^i_n}\right)\right) \]
is provable in the theory $\TJ\left(M, \mathsf{x}_{B^i_1}, \ldots, \mathsf{x}_{B^i_n}\right)$ for each endomorphism $f : i \to i$ in $\J$, which is true by Axiom \ref{thefunctortheory}.\ref{alphasarehoms}. This completes the proof that $\gamma_i$ commutes generically with all function symbols of $\Sigma$. 
\end{proof}

\begin{claim}
For any $i \in \ob\J$, $\gamma_i$ reflects definedness.
\end{claim}
\begin{proof}
Let the function symbol $g \in \Sigma$ be as above (with $n \geq 1$). We must show that the sequent
\[ g\left(\theta^*\left(s_{B^i_1}\right), \ldots, \theta^*\left(s_{B^i_n}\right)\right) \downarrow \ \vdash g(\mathsf{x}_{B_1}, \ldots, \mathsf{x}_{B_n}) \downarrow \] is provable in the theory $\T\left(M^i, \mathsf{x}_{B_1}, \ldots, \mathsf{x}_{B_n}\right)$. Since $([s_{C^j}])_{j, C} \in G_{\TJ}(M)$, we know that $([s_{C^j}])_{j, C}$ reflects definedness, which implies that the sequent
\[ g^i\left(s_{B^i_1}, \ldots, s_{B^i_n}\right) \downarrow \ \vdash g^i\left(\mathsf{x}_{B^i_1}, \ldots, \mathsf{x}_{B^i_n}\right) \downarrow \] is provable in the theory $\TJ\left(M, \mathsf{x}_{B^i_1}, \ldots, \mathsf{x}_{B^i_n}\right)$. As remarked above, the terms in the latter sequent are both $\alpha$-restricted. For each $1 \leq m \leq n$, the term $s_{B^i_m}$ is $i$-local and satisfies $\TJ\left(M, \mathsf{x}_{B^i_1}, \ldots, \mathsf{x}_{B^i_n}\right) \vdash s_{B^i_m} \downarrow$. The term $s_{B^i_m}$ also commutes generically with every endomorphism $f : i \to i$ in $\J$, because $([s_{C^j}])_{j, C} \in G_{\TJ}(M)$ and thus commutes generically with the function symbol $\alpha_f^{B_m} : B^i_m \to B^i_m$. So by a simple extension of Lemma \ref{theta^*lemma}, it follows that $\T\left(M^i, \mathsf{x}_{B_1}, \ldots, \mathsf{x}_{B_n}\right)$ proves the sequent
\[ \theta^*\left(g^i\left(s_{B^i_1}, \ldots, s_{B^i_n}\right)\right) \downarrow \ \vdash \theta^*\left(g^i\left(\mathsf{x}_{B^i_1}, \ldots, \mathsf{x}_{B^i_n}\right)\right) \downarrow. \] But this is the desired sequent by definition of $\theta^*$. 
\end{proof}

So $\gamma \in \prod_i G_\T\left(M^i\right)$ by the previous three claims, and now we must show that $\gamma \in \left(\prod_i G_\T\left(M^i\right)\right)^\J$. To show this, let $f : i \to k$ be an arbitrary morphism of $\J$. We must show that $G_{\T}\left(f^M\right)(\gamma_i) = \gamma_k$ (recall that $f^M := F^M(f) : M^i \to M^k$). Unravelling the definitions, this means that we must show for any $B \in \Sigma_{\mathsf{Sort}}$ that
\[ \left[\rho_{f^M}^B\left(\theta^*\left(s_{B^i}\right)\right)\right] = \left[\theta^*\left(s_{B^k}\right)\right] \] holds in $M^k\la \mathsf{x}_B \ra_B$, i.e. that
\[ \T\left(M^k, \mathsf{x}_B\right) \vdash \rho_{f^M}^B\left(\theta^*\left(s_{B^i}\right)\right) = \theta^*\left(s_{B^k}\right), \]
where $\rho_{f^M}^B : \T\left(M^i, \mathsf{x}_B\right) \to \T\left(M^k, \mathsf{x}_B\right)$ is the theory morphism induced by the $\Sigma$-morphism $f^M : M^i \to M^k$ by \cite[Definition 2.2.17]{thesis}.

Since $([s_{C^j}])_{j, C} \in G_{\TJ}(M)$, we know that $([s_{C^j}])_{j, C}$ commutes generically with the function symbol $\alpha_f^B : B^i \to B^k$ of $\SigmaJ$, which means that
\[ \TJ(M, \mathsf{x}_{B^i}) \vdash s_{B^k}\left[\alpha_f^B(\mathsf{x}_{B^i})/\mathsf{x}_{B^k}\right] = \alpha_f^B\left(s_{B^i}\right) \] (since $\TJ(M, \mathsf{x}_{B^i}) \vdash \alpha_f^B(\mathsf{x}_{B^i}) \downarrow$). By Lemma \ref{mainalphalemma}, there are $\alpha$-restricted variants $s_{B^k}\left[\alpha_f^B(\mathsf{x}_{B^i})/\mathsf{x}_{B^k}\right]', \alpha_f^B\left(s_{B^i}\right)' \in \mathsf{Term}^c\left(\SigmaJ\left(M, \mathsf{x}_{B^i}\right)\right)$ of these terms. So we have \[ \TJ(M, \mathsf{x}_{B^i}) \vdash s_{B^k}\left[\alpha_f^B(\mathsf{x}_{B^i})/\mathsf{x}_{B^k}\right]' = \alpha_f^B\left(s_{B^i}\right)'. \] Then by Lemma \ref{secondtheta^*lemma} we obtain
\[ \T\left(M^k, \mathsf{x}_B\right) \vdash \theta^*\left(s_{B^k}\left[\alpha_f^B(\mathsf{x}_{B^i})/\mathsf{x}_{B^k}\right]'\right) = \theta^*\left(\alpha_f^B\left(s_{B^i}\right)'\right), \] since both of the arguments of $\theta^*$ are of sort $B^k$. By Lemma \ref{smallthetalemma}, since $s_{B^k}$ is $k$-local and $f \in \mathsf{Cod}(k)$, we have $\theta^*\left(s_{B^k}\right) \equiv \theta^*\left(s_{B^k}[f]\right)$. We also have (by the observation following Definition \ref{alphacommuting})
\[ \TJ(M, \mathsf{x}_{B^i}) \vdash s_{B^k}[f] = s_{B^k}\left[\alpha_f^B(\mathsf{x}_{B^i})/\mathsf{x}_{B^k}\right], \] and hence
\[ \TJ(M, \mathsf{x}_{B^i}) \vdash s_{B^k}[f] = s_{B^k}\left[\alpha_f^B(\mathsf{x}_{B^i})/\mathsf{x}_{B^k}\right]'. \] Then by Lemma \ref{secondtheta^*lemma} we deduce
\[ \T\left(M^k, \mathsf{x}_B\right) \vdash \theta^*\left(s_{B^k}\right) \equiv \theta^*\left(s_{B^k}[f]\right) = \theta^*\left(s_{B^k}\left[\alpha_f^B(\mathsf{x}_{B^i})/\mathsf{x}_{B^k}\right]'\right) = \theta^*\left(\alpha_f^B\left(s_{B^i}\right)'\right). \] So to obtain our desired result, it suffices to show that \[ \T\left(M^k, \mathsf{x}_B\right) \vdash \theta^*\left(\alpha_f^B\left(s_{B^i}\right)'\right) = \rho_{f^M}^B\left(\theta^*\left(s_{B^i}\right)\right); \] but this is true by Lemma \ref{thetacommuteswiththeorymorphism}, given that $s_{B^i}$ is $\alpha$-restricted and $i$-local and $\TJ(M, \mathsf{x}_{B^i}) \vdash s_{B^i} \downarrow$. This completes the proof that $\gamma \in \left(\prod_i G_\T\left(M^i\right)\right)^\J$. 

To complete the proof that $\beta_M$ is surjective, we must now produce an element $\psi \in \mathsf{Aut}(\mathsf{Id}_\J)^M$ and then show that $\beta_M(\gamma, \psi) = ([s_{C^j}])_{j, C}$. So for every sort $B \in \Sigma$, we must construct a natural automorphism $\psi_B : \mathsf{Id}_{\J^M_B} \to \mathsf{Id}_{\J^M_B}$. Let $i$ be any object of $\J^M_B$. Then by definition of $\J^M_B$, it follows that the theory $\T\left(M^i\right)$ is non-trivial for the sort $B$. We now wish to define an endomorphism $\psi_B(i) : i \to i$ (which will turn out to be an isomorphism). 

We have shown that $\gamma_i := \left(\left[\theta^*\left(s_{C^i}\right)\right]\right)_{C \in \Sigma} \in G_{\T}\left(M^i\right)$. Then because $\T$ has single-indeterminate isotropy, we can assume without loss of generality that $\theta^*\left(s_{B^i}\right) \in \mathsf{Term}^c\left(\Sigma\left(M^i, \mathsf{x}_B\right)\right)_B$ has exactly one occurrence of the indeterminate $\mathsf{x}_B$. From this, it then follows that $s_{B^i} \in \mathsf{Term}^c\left(\SigmaJ\left(M, \mathsf{x}_{B^i}\right)\right)_{B^i}$ has exactly one occurrence of the indeterminate $\mathsf{x}_{B^i}$ (because distinct occurrences of $\mathsf{x}_{B^i}$ in $s_{B^i}$ correspond to distinct occurrences of $\mathsf{x}_B$ in $\theta^*\left(s_{B^i}\right)$). More precisely, because $\theta^*\left(s_{B^i}\right)$ has exactly one occurrence of $\mathsf{x}_B$, it follows that $\rho_{M^i}^B\left(\theta^*\left(s_{B^i}\right)\right)$ has exactly one occurrence of $\mathsf{x}_{B^i}$. But by Lemma \ref{alphafreevariantlemma} we know that $\rho_{M^i}^B\left(\theta^*\left(s_{B^i}\right)\right) \equiv s_{B^i}^{-\alpha}$, and so the $\alpha$-free variant $s_{B^i}^{-\alpha}$ of $s_{B^i}$ has exactly one occurrence of $\mathsf{x}_{B^i}$, which implies that $s_{B^i}$ has exactly one occurrence of $\mathsf{x}_{B^i}$. 

Now consider $\theta\left(s_{B^i}\right) \in \mathsf{Term}^c\left(\Sigma\left(M^i, \overline{\mathsf{x}^B_{\mathsf{Cod}(i)}}\right)\right)$. Since $s_{B^i}$ has exactly one occurrence of $\mathsf{x}_{B^i}$, it follows that $\theta\left(s_{B^i}\right)$ has exactly one indeterminate from $\Sigma\left(M^i, \overline{\mathsf{x}^B_{\mathsf{Cod}(i)}}\right)$, and moreover the subscript of this indeterminate will be an endomorphism of $i$. We thus define $\psi_B(i) : i \to i$ to be this endomorphism. In other words, we define $\psi_B(i) : i \to i$ so that $\mathsf{x}_{\psi_B(i)}^B$ is the unique indeterminate occurring in $\theta\left(s_{B^i}\right)$. We now prove:

\begin{claim}
$\psi_B(i) : i \to i$ is an isomorphism. 
\end{claim}
\begin{proof}
From the proof that $\gamma_i := (\left[\theta^*\left(s_{C^i}\right)\right])_{C \in \Sigma} \in G_{\T}\left(M^i\right)$, it follows that $\gamma_i^{-1} = \left(\left[\theta^*\left(s_{C^i}^{-1}\right)\right]\right)_{C \in \Sigma} \in G_{\T}\left(M^i\right)$. Then, as for $s_{B^i}$, it follows that $s_{B^i}^{-1}$ has exactly one occurrence of the indeterminate $\mathsf{x}_{B^i}$, and so we define $\psi_B(i)^{-1} : i \to i$ from $\theta\left(s_{B^i}^{-1}\right)$ in the same way that we defined $\psi_B(i)$ from $\theta\left(s_{B^i}\right)$. We now need to verify that $\psi_B(i)$ and $\psi_B(i)^{-1}$ are in fact mutually inverse endomorphisms of $i$. First, we know that
\[ \TJ(M, \mathsf{x}_{B^i}) \vdash s_{B^i}\left[s_{B^i}^{-1}/\mathsf{x}_{B^i}\right] = \mathsf{x}_{B^i} = s_{B^i}^{-1}[s_{B^i}/\mathsf{x}_{B^i}]. \] By Lemma \ref{mainalphalemma}, there is an $\alpha$-restricted variant $s_{B^i}\left[s_{B^i}^{-1}/\mathsf{x}_{B^i}\right]'$ of $s_{B^i}\left[s_{B^i}^{-1}/\mathsf{x}_{B^i}\right]$, so we obtain $\TJ(M, \mathsf{x}_{B^i}) \vdash s_{B^i}\left[s_{B^i}^{-1}/\mathsf{x}_{B^i}\right]' = \mathsf{x}_{B^i}$. Then by Proposition \ref{thetaprop} we have
\[ \T\left(M^i, \overline{\mathsf{x}^B_{\mathsf{Cod}(i)}}\right) \vdash \theta\left(s_{B^i}\left[s_{B^i}^{-1}/\mathsf{x}_{B^i}\right]'\right) = \theta(\mathsf{x}_{B^i}) \equiv \mathsf{x}_{\mathsf{id}_i}^B. \] 
Since the unique indeterminate that $\theta\left(s_{B^i}\right) \in \mathsf{Term}^c\left(\Sigma\left(M^i, \overline{\mathsf{x}^B_{\mathsf{Cod}(i)}}\right)\right)$ contains is $\mathsf{x}_{\psi_B(i)}^B$, it follows by Lemma \ref{lastthetalemma} that 
\[ \theta\left(s_{B^i}\left[s_{B^i}^{-1}/\mathsf{x}_{B^i}\right]'\right) \equiv \theta\left(s_{B^i}\right)\left[\theta\left(\left(s_{B^i}^{-1}\right)^{\psi_B(i)}\right)/\mathsf{x}_{\psi_B(i)}^B\right]. \]
So then the unique indeterminate that occurs in $\theta\left(s_{B^i}\left[s_{B^i}^{-1}/\mathsf{x}_{B^i}\right]'\right)$ will be the unique indeterminate that occurs in $\theta\left(\left(s_{B^i}^{-1}\right)^{\psi_B(i)}\right)$. But since the unique indeterminate that occurs in $\theta\left(s_{B^i}^{-1}\right)$ is $\mathsf{x}_{\psi_B(i)^{-1}}^B$, it follows by Lemma \ref{theta^flemma} that the unique indeterminate that occurs in $\theta\left(\left(s_{B^i}^{-1}\right)^{\psi_B(i)}\right)$ is $\mathsf{x}_{\psi_B(i) \circ \psi_B(i)^{-1}}^B$. In summary, the unique indeterminate that occurs in $\theta\left(s_{B^i}\left[s_{B^i}^{-1}/\mathsf{x}_{B^i}\right]'\right)$ is $\mathsf{x}_{\psi_B(i) \circ \psi_B(i)^{-1}}^B$.

Now, since $\T\left(M^i\right)$ is non-trivial for the sort $B$ and \[ \T\left(M^i, \overline{\mathsf{x}^B_{\mathsf{Cod}(i)}}\right) \vdash \theta\left(s_{B^i}\left[s_{B^i}^{-1}/\mathsf{x}_{B^i}\right]'\right) = \mathsf{x}_{\mathsf{id}_i}^B, \] it follows by Lemma \ref{occurrenceofidentityconstant} that $\mathsf{x}_{\mathsf{id}_i}^B$ occurs in $\theta\left(s_{B^i}\left[s_{B^i}^{-1}/\mathsf{x}_{B^i}\right]'\right)$. But then we must have $\mathsf{x}_{\psi_B(i) \circ \psi_B(i)^{-1}}^B \equiv \mathsf{x}_{\mathsf{id}_i}^B$, which forces $\psi_B(i) \circ \psi_B(i)^{-1} = \mathsf{id}_i$. The proof that $\psi_B(i)^{-1} \ \circ \ \psi_B(i) = \mathsf{id}_i$ is analogous, which completes the proof that $\psi_B(i) : i \to i$ is an isomorphism.
\end{proof}

Next, we prove:

\begin{claim}
$\psi_B$ is a natural automorphism of $\mathsf{Id}_{\J^M_B}$. 
\end{claim}
\begin{proof}
Let $i, k \in \ob\J^M_B$, which means that the theories $\T\left(M^i\right)$ and $\T\left(M^k\right)$ are both non-trivial for the sort $B$, and let $h : i \to k$ be an arbitrary morphism of $\J$. We must show that $h \circ \psi_B(i) = \psi_B(k) \circ h : i \to k$. We know that $\alpha_h^B : B^i \to B^k$ is a function symbol of $\SigmaJ$, so because $([s_{C^j}])_{j, C} \in G_{\TJ}(M)$, it follows that $([s_{C^j}])_{j, C}$ commutes generically with this function symbol, which means that
\[ \TJ(M, \mathsf{x}_{B^i}) \vdash \alpha_h^B\left(s_{B^i}\right) = s_{B^k}\left[\alpha_h^B(\mathsf{x}_{B^i})/\mathsf{x}_{B^k}\right]. \] By Lemma \ref{mainalphalemma}, the righthand term in the above equation has an $\alpha$-restricted variant$\ s_{B^k}\left[\alpha_h^B(\mathsf{x}_{B^i})/\mathsf{x}_{B^k}\right]'$, and by Lemma \ref{firstalphalemma}, since $s_{B^i} : B^i$ is $\alpha$-restricted and $\TJ(M, \mathsf{x}_{B^i}) \vdash s_{B^i} \downarrow$, we know that $s_{B^i}^h : B^k$ is an $\alpha$-restricted term with $\TJ(M, \mathsf{x}_{B^i}) \vdash \alpha_h^B\left(s_{B^i}\right) = s_{B^i}^h$. Altogether, we then have
\[ \TJ(M, \mathsf{x}_{B^i}) \vdash s_{B^i}^h = s_{B^k}\left[\alpha_h^B(\mathsf{x}_{B^i})/\mathsf{x}_{B^k}\right]', \] with both terms $\alpha$-restricted. By Proposition \ref{thetaprop}, we then obtain
\[ \T\left(M^k, \overline{\mathsf{x}^B_{\mathsf{Cod}(k)}}\right) \vdash \theta\left(s_{B^i}^h\right) = \theta\left(s_{B^k}\left[\alpha_h^B(\mathsf{x}_{B^i})/\mathsf{x}_{B^k}\right]'\right). \] Since the unique indeterminate that occurs in $\theta\left(s_{B^i}\right)$ is $\mathsf{x}_{\psi_B(i)}^B$, it follows by Lemma \ref{theta^flemma} that the unique indeterminate that occurs in $\theta\left(s_{B^i}^h\right)$ is $\mathsf{x}_{h \circ \psi_B(i)}^B$. Also, we know by Lemma \ref{lastthetalemma} that 
\[ \theta\left(s_{B^k}\left[\alpha_h^B(\mathsf{x}_{B^i})/\mathsf{x}_{B^k}\right]'\right) \equiv \theta\left(s_{B^k}\right)\left[\theta\left(\alpha_h^B(\mathsf{x}_{B^i})^{\psi_B(k)}\right)/\mathsf{x}_{\psi_B(k)}^B\right], \] since $\mathsf{x}_{\psi_B(k)}^B$ is the unique indeterminate that occurs in $\theta\left(s_{B^k}\right)$. But (by the proof of Lemma \ref{firstalphalemma}) we have 
\[ \theta\left(\alpha_h^B(\mathsf{x}_{B^i})^{\psi_B(k)}\right) \equiv \theta\left(\alpha_{\psi_B(k) \circ h}^B(\mathsf{x}_{B^i})\right) \equiv \mathsf{x}_{\psi_B(k) \circ h}^B, \] which means that the unique indeterminate that occurs in $\theta\left(s_{B^k}\left[\alpha_h^B(\mathsf{x}_{B^i})/\mathsf{x}_{B^k}\right]'\right)$ is $\mathsf{x}_{\psi_B(k) \circ h}^B$. Now, suppose towards a contradiction that $h \circ \psi_B(i) \neq \psi_B(k) \circ h$. Then we would have $\mathsf{x}_{h \circ \psi_B(i)}^B \not\equiv \mathsf{x}_{\psi_B(k) \circ h}^B$, since distinct morphisms with codomain $k$ correspond to distinct indeterminates in $\Sigma\left(M^k, \overline{\mathsf{x}^B_{\mathsf{Cod}(k)}}\right)$. By the preceding discussion, it would then follow that in the equation
\[ \T\left(M^k, \overline{\mathsf{x}^B_{\mathsf{Cod}(k)}}\right) \vdash \theta\left(s_{B^i}^h\right) = \theta\left(s_{B^k}\left[\alpha_h^B(\mathsf{x}_{B^i})/\mathsf{x}_{B^k}\right]'\right), \] i.e.
\[ \T\left(M^k, \overline{\mathsf{x}^B_{\mathsf{Cod}(k)}}\right) \vdash \theta\left(s_{B^i}^h\right) = \theta\left(s_{B^k}\right)\left[\mathsf{x}_{\psi_B(k) \circ h}^B/\mathsf{x}_{\psi_B(k)}^B\right], \]
the two terms have no indeterminates in common. From the previous line, we can infer  
\[ \T\left(M^k, \mathsf{x}_{h \circ \psi_B(i)}^B, \mathsf{x}_{\psi_B(k) \circ h}^B\right) \vdash \theta\left(s_{B^i}^h\right) = \theta\left(s_{B^k}\right)\left[\mathsf{x}_{\psi_B(k) \circ h}^B/\mathsf{x}_{\psi_B(k)}^B\right] \] by \cite[Lemma 5.1.30]{thesis}. Now let $y, y'$ be distinct variables of sort $B$. Then by \cite[Theorem 10]{Horn}, we may conclude
\[ \T\left(M^k\right) \vdash^{y, y'} \theta\left(s_{B^i}^h\right)\left[y/\mathsf{x}_{h \circ \psi_B(i)}^B\right] = \theta\left(s_{B^k}\right)\left[y'/\mathsf{x}_{\psi_B(k)}^B\right], \tag{$*$} \] since the indeterminates $\mathsf{x}_{h \circ \psi_B(i)}^B, \mathsf{x}_{\psi_B(k) \circ h}^B : B$ are distinct. Now, in the proof of the claim that $\psi_B(k)$ is an isomorphism we showed that
\[ \theta\left(s_{B^k}\left[s_{B^k}^{-1}/\mathsf{x}_{B^k}\right]'\right) \equiv \theta\left(s_{B^k}\right)\left[\theta\left(\left(s_{B^k}^{-1}\right)^{\psi_B(k)}\right)/\mathsf{x}_{\psi_B(k)}^B\right]. \] Hence, by substituting $\theta\left(\left(s_{B^k}^{-1}\right)^{\psi_B(k)}\right)$ for $y'$ in ($*$), we obtain
\[ \T\left(M^k, \mathsf{x}_{\mathsf{id}_k}^B\right) \vdash^{y} \theta\left(s_{B^i}^h\right)\left[y/\mathsf{x}_{h \circ \psi_B(i)}^B\right] = \theta\left(s_{B^k}\right)\left[\theta\left(\left(s_{B^k}^{-1}\right)^{\psi_B(k)}\right)/\mathsf{x}_{\psi_B(k)}^B\right] \equiv \theta\left(s_{B^k}\left[s_{B^k}^{-1}/\mathsf{x}_{B^k}\right]'\right).  \] But we also know (from the same proof) that 
\[ \T\left(M^k, \mathsf{x}_{\mathsf{id}_k}^B\right) \vdash \theta\left(s_{B^k}\left[s_{B^k}^{-1}/\mathsf{x}_{B^k}\right]'\right) = \mathsf{x}_{\mathsf{id}_k}^B, \] so we finally obtain
\[ \T\left(M^k, \mathsf{x}_{\mathsf{id}_k}^B\right) \vdash^{y} \theta\left(s_{B^i}^h\right)\left[y/\mathsf{x}_{h \circ \psi_B(i)}^B\right] = \mathsf{x}_{\mathsf{id}_k}^B. \] Since $\mathsf{x}_{\mathsf{id}_k}^B$ does not appear in $\theta\left(s_{B^i}^h\right)\left[y/\mathsf{x}_{h \circ \psi_B(i)}^B\right]$, it then follows from \cite[Theorem 10]{Horn} that if $y' : B$ is a variable distinct from $y$, then
\[ \T\left(M^k\right) \vdash^{y, y'} \theta\left(s_{B^i}^h\right)\left[y/\mathsf{x}_{h \circ \psi_B(i)}^B\right] = y', \] and $y'$ does not appear in $\theta\left(s_{B^i}^h\right)\left[y/\mathsf{x}_{h \circ \psi_B(i)}^B\right]$. But then if $y'' : B$ is a variable distinct from both $y$ and $y'$, we also obtain
\[ \T\left(M^k\right) \vdash^{y, y''} \theta\left(s_{B^i}^h\right)\left[y/\mathsf{x}_{h \circ \psi_B(i)}^B\right] = y''. \] Finally, we deduce that $\T\left(M^k\right) \vdash^{y', y''} y' = y''$, which contradicts the assumption that $\T\left(M^k\right)$ is non-trivial for the sort $B$. So we must have $h \circ \psi_B(i) = \psi_B(k) \circ h$, as desired. This completes the argument that $\psi_B$ is a natural automorphism of $\mathsf{Id}_{\J^M_B}$.
\end{proof} 

\noindent To complete the proof that $\psi \in \mathsf{Aut}(\mathsf{Id}_\J)^M$, we must also verify: 
\begin{claim}
If $g : B_1 \times \ldots \times B_n \to B$ is a function symbol of $\Sigma$ with $n \geq 1$, then for any $i \in \ob\J$ and $1 \leq m \leq n$ such that $g^{M^i}$ is non-degenerate in position $m$ we have $\psi_{B_m}(i) = \psi_B(i) : i \to i$. 
\end{claim}
\begin{proof}
For simplicity, we will let $g : A \times B \to C$ be a binary function symbol of $\Sigma$ with $A \neq B$, and let $i \in \ob\J$. Suppose that
$g^{M^i}$ is non-degenerate in position $1$ (the argument for position $2$ being analogous), which means that 
\[ \T\left(M^i\right) \nvdash^{y_1, y_1', y_2} g(y_1, y_2) = g(y_1', y_2) \] for pairwise distinct variables $y_1, y_1' : A, y_2 : B$. Then (as remarked in the definition of $\mathsf{Aut}(\mathsf{Id}_\J)^M$) this implies that $\T\left(M^i\right)$ is non-trivial for the sort $C$, so that $i \in \ob\J^M_C$ and hence $\psi_{A}(i), \psi_C(i) : i \to i$ are defined. We must now show that $\psi_{A}(i) = \psi_C(i) : i \to i$. If $g$ is \emph{not} totally defined in $\T$, then the assumption that $\T$ has single-sorted non-total operations implies that $A = B = C$, which obviously entails the desired result. So suppose that $g$ \emph{is} totally defined in $\T$. If the sorts $A$ and $C$ are identical, then the desired result trivially follows, so suppose that $A \neq C$. Suppose towards a contradiction that $\psi_{A}(i) \neq \psi_C(i) : i \to i$. Since $([s_{C^j}])_{j, C} \in G_{\TJ}(M)$, it follows that $([s_{C^j}])_{j, C}$ commutes generically with the function symbol $g^i : A^i \times B^i \to C^i$ of $\SigmaJ$. Since $g$ is totally defined in $\T$, this entails that
\[ \TJ(M, \mathsf{x}_{A^i}, \mathsf{x}_{B^i}) \vdash g^i(s_{A^i}, s_{B^i}) = s_{C^i}\left[g^i(\mathsf{x}_{A^i}, \mathsf{x}_{B^i})/\mathsf{x}_{C^i}\right]. \] Let $s_{C^i}\left[g^i(\mathsf{x}_{A^i}, \mathsf{x}_{B^i})/\mathsf{x}_{C^i}\right]'$ be the $\alpha$-restricted variant of $s_{C^i}\left[g^i(\mathsf{x}_{A^i}, \mathsf{x}_{B^i})/\mathsf{x}_{C^i}\right]$ obtained from (a simple extension of) Lemma \ref{mainalphalemma}. Then by (simple extensions of) Lemma \ref{mainalphalemma} and Proposition \ref{thetaprop} and the definition of $\theta$ we have 
\[ \T\left(M^i, \overline{\mathsf{x}^A_{\mathsf{Cod}(i)}}, \overline{\mathsf{x}^B_{\mathsf{Cod}(i)}}\right) \vdash g\left(\theta(s_{A^i}), \theta\left(s_{B^i}\right)\right) = \theta\left(s_{C^i}[g^i(\mathsf{x}_{A^i}, \mathsf{x}_{B^i})/\mathsf{x}_{C^i}]'\right). \] By a simple extension of Lemma \ref{lastthetalemma}, we have
\[ \theta\left(s_{C^i}\left[g^i(\mathsf{x}_{A^i}, \mathsf{x}_{B^i})/\mathsf{x}_{C^i}\right]'\right) \equiv \theta\left(s_{C^i}\right)\left[\theta\left(g^i(\mathsf{x}_{A^i}, \mathsf{x}_{B^i})^{\psi_C(i)}\right)/\mathsf{x}_{\psi_C(i)}^C\right]. \] We also have (by a simple extension of Lemma \ref{firstalphalemma})
\begin{align*}
\theta\left(g^i(\mathsf{x}_{A^i}, \mathsf{x}_{B^i})^{\psi_C(i)}\right)	&\equiv \theta\left(g^i\left(\mathsf{x}_{A^i}^{\psi_C(i)}, \mathsf{x}_{B^i}^{\psi_C(i)}\right)\right) \\
								&\equiv \theta\left(g^i\left(\alpha_{\psi_C(i)}^A(\mathsf{x}_{A^i}), \alpha_{\psi_C(i)}^B(\mathsf{x}_{B^i})\right)\right) \\
								&\equiv g\left(\theta\left(\alpha_{\psi_C(i)}^A(\mathsf{x}_{A^i}), \theta(\alpha_{\psi_C(i)}^B(\mathsf{x}_{B^i})\right)\right) \\
								&\equiv g\left(\mathsf{x}_{\psi_C(i)}^A, \mathsf{x}_{\psi_C(i)}^B\right).
\end{align*}
From this, we obtain
\[ \theta\left(s_{C^i}\left[g^i(\mathsf{x}_{A^i}, \mathsf{x}_{B^i})/\mathsf{x}_{C^i}\right]'\right) \equiv \theta\left(s_{C^i}\right)\left[g\left(\mathsf{x}_{\psi_C(i)}^A, \mathsf{x}_{\psi_C(i)}^B\right)/\mathsf{x}_{\psi_C(i)}^C\right]. \] Finally, we have
\[ \T\left(M^i, \overline{\mathsf{x}^A_{\mathsf{Cod}(i)}}, \overline{\mathsf{x}^B_{\mathsf{Cod}(i)}}\right) \vdash g\left(\theta(s_{A^i}), \theta\left(s_{B^i}\right)\right) = \theta\left(s_{C^i}\right)\left[g\left(\mathsf{x}_{\psi_C(i)}^A, \mathsf{x}_{\psi_C(i)}^B\right)/\mathsf{x}_{\psi_C(i)}^C\right]. \] Now, we know that $\mathsf{x}_{\psi_A(i)}^A$ is the unique indeterminate occurring in $\theta(s_{A^i})$, and that $\mathsf{x}_{\psi_C(i)}^C$ is the unique indeterminate occurring in $\theta\left(s_{C^i}\right)$. Because of our assumption that $\psi_A(i) \neq \psi_C(i)$, it follows that $\mathsf{x}_{\psi_A(i)}^A \not\equiv \mathsf{x}_{\psi_C(i)}^A \in \Sigma\left(M^i, \overline{\mathsf{x}^A_{\mathsf{Cod}(i)}}\right)$. For the remainder of the argument, we will assume that $\T\left(M^i\right)$ is non-trivial for the sort $B$ as well, so that $i \in \ob\J^M_B$ and $\psi_B(i) : i \to i$ is defined, and we will also assume that $\psi_B(i) = \psi_C(i)$ (so that $\mathsf{x}_{\psi_B(i)}^B \equiv \mathsf{x}_{\psi_C(i)}^B$). Without these assumptions, the required argument is a simpler version of the one we are about to give.

So, in the above equation, we can substitute $\theta\left(\left(s_{A^i}^{-1}\right)^{\psi_A(i)}\right)$ for $\mathsf{x}_{\psi_A(i)}^A$ in the lefthand term, and $\theta\left(\left(s_{B^i}^{-1}\right)^{\psi_B(i)}\right)$ for $\mathsf{x}_{\psi_B(i)}^B \equiv \mathsf{x}_{\psi_C(i)}^B$ in both terms, and we obtain
\[ \T\left(M^i, \overline{\mathsf{x}^A_{\mathsf{Cod}(i)}}, \overline{\mathsf{x}^B_{\mathsf{Cod}(i)}}\right) \vdash g\left(\theta(s_{A^i})\left[\theta\left(\left(s_{A^i}^{-1}\right)^{\psi_A(i)}\right)/\mathsf{x}_{\psi_A(i)}^A\right], \theta\left(s_{B^i}\right)\left[\theta\left(\left(s_{B^i}^{-1}\right)^{\psi_B(i)}\right)/\mathsf{x}_{\psi_B(i)}^B\right]\right) \] \[ = \theta\left(s_{C^i}\right)\left[g\left(\mathsf{x}_{\psi_C(i)}^A, \theta\left(\left(s_{B^i}^{-1}\right)^{\psi_B(i)}\right)\right)/\mathsf{x}_{\psi_C(i)}^C\right] \] (note that $A \neq B$ implies $\mathsf{x}_{\psi_A(i)}^A \not\equiv \mathsf{x}_{\psi_B(i)}^B$). Earlier in the proof of the proposition, we saw that
\[ \T\left(M^i, \overline{\mathsf{x}^A_{\mathsf{Cod}(i)}}\right) \vdash \theta(s_{A^i})\left[\theta\left(\left(s_{A^i}^{-1}\right)^{\psi_A(i)}\right)/\mathsf{x}_{\psi_A(i)}^A\right] = \mathsf{x}_{\mathsf{id}_i}^A \] and
\[ \T\left(M^i, \overline{\mathsf{x}^B_{\mathsf{Cod}(i)}}\right) \vdash \theta\left(s_{B^i}\right)\left[\theta\left(\left(s_{B^i}^{-1}\right)^{\psi_B(i)}\right)/\mathsf{x}_{\psi_B(i)}^B\right] = \mathsf{x}_{\mathsf{id}_i}^B, \] so we obtain
\[ \T\left(M^i, \overline{\mathsf{x}^A_{\mathsf{Cod}(i)}}, \overline{\mathsf{x}^B_{\mathsf{Cod}(i)}}\right) \vdash g\left(\mathsf{x}_{\mathsf{id}_i}^A, \mathsf{x}_{\mathsf{id}_i}^B\right) = \theta\left(s_{C^i}\right)\left[g\left(\mathsf{x}_{\psi_C(i)}^A, \theta\left(\left(s_{B^i}^{-1}\right)^{\psi_B(i)}\right)\right)/\mathsf{x}_{\psi_C(i)}^C\right]. \]
We can also repeat the above argument to show that 
\[ \T\left(M^i, \overline{\mathsf{x}^A_{\mathsf{Cod}(i)}}, \overline{\mathsf{x}^B_{\mathsf{Cod}(i)}},  {\mathsf{x}_{\mathsf{id}_i}^A}'\right) \vdash g\left({\mathsf{x}_{\mathsf{id}_i}^A}', \mathsf{x}_{\mathsf{id}_i}^B\right) = \theta\left(s_{C^i}\right)\left[g\left(\mathsf{x}_{\psi_C(i)}^A, \theta\left(\left(s_{B^i}^{-1}\right)^{\psi_B(i)}\right)\right)/\mathsf{x}_{\psi_C(i)}^C\right], \] where ${\mathsf{x}_{\mathsf{id}_i}^A}' \notin \Sigma\left(M^i, \overline{\mathsf{x}^A_{\mathsf{Cod}(i)}}, \overline{\mathsf{x}^B_{\mathsf{Cod}(i)}}\right)$ is a new constant of sort $A$. So then we obtain
\[ \T\left(M^i, \overline{\mathsf{x}^A_{\mathsf{Cod}(i)}}, \overline{\mathsf{x}^B_{\mathsf{Cod}(i)}},  {\mathsf{x}_{\mathsf{id}_i}^A}'\right) \vdash g\left(\mathsf{x}_{\mathsf{id}_i}^A, \mathsf{x}_{\mathsf{id}_i}^B\right) = g\left({\mathsf{x}_{\mathsf{id}_i}^A}', \mathsf{x}_{\mathsf{id}_i}^B\right). \] By (a slight variation of) \cite[Lemma 5.1.30]{thesis}, it then follows that 
\[ \T\left(M^i, \x_{\id_i}^A, {\x_{\id_i}^B, \x_{\id_i}^A}'\right) \vdash g\left(\mathsf{x}_{\mathsf{id}_i}^A, \mathsf{x}_{\mathsf{id}_i}^B\right) = g\left({\mathsf{x}_{\mathsf{id}_i}^A}', \mathsf{x}_{\mathsf{id}_i}^B\right). \] By \cite[Theorem 10]{Horn} again, it follows that if $y, y' : A$ are distinct variables and $z : B$ is a variable, then $\T\left(M^i\right) \vdash^{y, y', z} g(y, z) = g(y', z)$, which contradicts the assumption that $g^{M^i}$ is non-degenerate in position $1$. This contradiction implies that we must have $\psi_A(i) = \psi_C(i) : i \to i$, as desired.
\end{proof} 
Hence, we may finally conclude that $\psi \in \mathsf{Aut}(\mathsf{Id}_\J)^M$, and therefore $(\gamma, \psi) \in \left(\prod_i G_\T\left(M^i\right)\right)^\J \times \mathsf{Aut}(\mathsf{Id}_\J)^M$. It now remains to show that $\beta_M(\gamma, \psi) = ([s_{C^j}])_{j, C} \in G_{\TJ}(M)$, which is not too difficult; we refer the reader to the end of the proof of \cite[Proposition 5.1.47]{thesis} for the verification. 
\end{proof}

One can now straightforwardly show as in \cite[Definition 5.1.55]{thesis} that the assignment
\[ M \mapsto \left(\prod_i G_\T\left(M^i\right)\right)^\J \times \mathsf{Aut}(\mathsf{Id}_\J)^M = \mathsf{lim}\left(G_\T \circ F^M\right) \times \mathsf{Aut}(\mathsf{Id}_\J)^M \in    \mathsf{Group} \] (see Lemma \ref{grouptheorylemma}) is functorial, i.e. that there is a canonical functor $G_{\TJ}^* : \TJmod \to \Group$ with this action on objects. We therefore obtain the following theorem, whose proof may be found in \cite[Theorem 5.1.56]{thesis}:
 
\begin{theo}
\label{functortheoriestheorem}
Let $\T$ be a quasi-equational theory with single-indeterminate isotropy and single-sorted non-total operations, and let $\J$ be a small index category. Then the family $\left(\beta_M\right)_{M \in \TJmod}$ of group isomorphisms is natural in $M$, i.e. we have a natural isomorphism $\beta : G_{\TJ}^* \xrightarrow{\sim} G_{\TJ} : \PTJmod \to \mathsf{Group}$. \qed
\end{theo}

\begin{rmk}
\label{notsingleindeterminateisotropy}
{\em
We used the assumption that $\T$ has single-indeterminate isotropy in order to show in Proposition \ref{betaissurjective} that the injective group homomorphism 
\[ \beta_M : \left(\prod_i G_\T\left(M^i\right)\right)^\J \times \mathsf{Aut}(\mathsf{Id}_\J)^M \to G_{\TJ}(M) \] is also \emph{surjective} for each small category $\J$ and $M \in \PTmod$. In \cite[Remark 5.1.58]{thesis} we provide an example of a theory $\T$ without single-indeterminate isotropy, for which there is a small category $\J$ and a model $M$ of $\TJ$ for which $\beta_M$ is \emph{not} surjective. So Proposition \ref{betaissurjective} does \emph{not} hold in general for theories without single-indeterminate isotropy. 

We also used the assumption that $\T$ has single-sorted non-total operations in order to show that $\beta_M$ is surjective for each small category $\J$ and $M \in \PTmod$. Specifically, given $([s_{C^j}])_{j, C} \in G_{\TJ}(M)$ for $M \in \PTJmod$, we produced $(\gamma, \psi) \in \left(\prod_i G_\T\left(M^i\right)\right)^\J \times \mathsf{Aut}(\mathsf{Id}_\J)^M$ with $\beta_M(\gamma, \psi) = ([s_{C^j}])_{j, C}$, and we used the assumption that $\T$ has single-sorted non-total operations to show that $\psi \in \mathsf{Aut}(\mathsf{Id}_\J)^M$; more specifically, we used this assumption to show that if $g : A_1 \times \ldots \times A_n \to A$ is a function symbol of $\Sigma$ with $n \geq 1$ and $i \in \ob\J$ and $g^{M^i}$ is non-degenerate in position $1 \leq m \leq n$, then $\psi_{A_m}(i) = \psi_A(i) : i \xrightarrow{\sim} i$. In \cite[Remark 5.1.59]{thesis} we show that there is a theory $\T$ without single-sorted non-total operations for which there is a small category $\J$ and a model $M$ of $\TJ$ such that $\beta_M$ is \emph{not} surjective. So Proposition \ref{betaissurjective} does \emph{not} hold in general for theories that have multi-sorted operations that are not totally defined. \qed
}   
\end{rmk}

\section{Categorical characterization}
\label{categorical}

In this section, we will deduce a categorical characterization of the covariant isotropy group $\Z_{\Tmod^\J}$ of $\Tmod^\J$ from the logical characterization of $G_{\TJ}$ given in the previous section. Unless otherwise stated, $\T$ is an arbitrary quasi-equational theory and $\J$ is an arbitrary small index category. 

Recall that if $F : \J \to \Tmod$, then $\Dom(F)$ is the class of morphisms in $\Tmod^\J$ with domain $F$.

\begin{defn}
\label{determinedbyfunctor}
{\em Let $F : \J \to \PTmod$, and let $\pi = \left(\pi_\mu : \mathsf{cod}(\mu) \to \mathsf{cod}(\mu)\right)_{\mu \in \mathsf{Dom}(F)}$ be a $\mathsf{Dom}(F)$-indexed family of endomorphisms in $\PTmod^\J$. For each $i \in \ob\J$, let \[ \phi^i = \left(\phi^i_g : \mathsf{cod}(g) \xrightarrow{\sim} \mathsf{cod}(g)\right)_{g \in \mathsf{Dom}(F(i))} \in \Z_\T(F(i)). \] Finally, let $\psi = (\psi_B)_{B \in \Sigma} \in \mathsf{Aut}(\mathsf{Id}_\J)^{M^F}$ (where $M^F \in \PTJmod$ is the model of $\TJ$ corresponding to the functor $F$ as in Proposition \ref{modelsarefunctors}). 

We say that $\pi$ is \emph{determined by} $\psi \in \mathsf{Aut}(\mathsf{Id}_\J)^{M^F}$ and the family $\left(\phi^i\right)_{i \in \J} \in \prod_{i \in \J} \Z_\T(F(i))$ if the following holds for any morphism $\mu : F \to G$ in $\PTmod^\J$, any $B \in \Sigma_{\mathsf{Sort}}$, and any $k \in \ob\J$:
\begin{itemize}
\item If $k \notin \J^{M^F}_B$, i.e. if $\T\left(F(k)\right)$ is trivial for the sort $B$, then $\pi_\mu(k)_B = \mathsf{id} : G(k)_B \to G(k)_B$.

\item If $k \in \J^{M^F}_B$, i.e. if $\T\left(F(k)\right)$ is non-trivial for the sort $B$, then \[ \pi_\mu(k)_B = G(\psi_B(k)) \circ \phi^k_{\mu(k)} : G(k)_B \to G(k)_B \]
(where $\mu(k) : F(k) \to G(k)$ is a $\Sigma$-morphism). \qed
\end{itemize}
}
\end{defn}

\noindent Note that in the corresponding \cite[Definition 5.2.7]{thesis}, the last equation of Definition \ref{determinedbyfunctor} instead has the form $\pi_\mu(k)_B = \left(\phi^k_{\mu(k) \circ F(\psi_B(k))}\right)_B \circ G(\psi_B(k))_B$, but this is equivalent to our equation because
\[ \phi^k_{\mu(k) \circ F(\psi_B(k))} \circ G(\psi_B(k)) = \phi^k_{G(\psi_B(k)) \circ \mu(k)} \circ G(\psi_B(k)) = G(\psi_B(k)) \circ \phi^k_{\mu(k)} \] by naturality of $\mu : F \to G$ and the fact that $\phi^k \in \Z_\T(F(k))$.   

\begin{defn}
\label{Fcompatible}
{\em Let $F : \J \to \PTmod$. For each $i \in \ob\J$, let \[ \phi^i = \left(\phi^i_g : \mathsf{cod}(g) \xrightarrow{\sim} \mathsf{cod}(g)\right)_{g \in \mathsf{Dom}(F(i))} \in \Z_\T(F(i)). \] 
We say that the family $\left(\phi^i\right)_{i \in \J} \in \prod_{i \in \J} \Z_\T(F(i))$ is \emph{compatible} if for every morphism $f : i \to j$ in $\J$ and every $g \in \mathsf{Dom}(F(j))$ we have
\[ \phi^j_g = \phi^i_{g \circ F(f)} : \mathsf{cod}(g) \xrightarrow{\sim} \mathsf{cod}(g). \]
In other words, we say that $\left(\phi^i\right)_{i \in \J}$ is compatible if $\left(\phi^i\right)_{i \in \J} \in \mathsf{lim}(\Z_\T \circ F) \in \mathsf{Group}$. \qed }
\end{defn}

\noindent We now have the following categorical characterization of the covariant isotropy group of $\PTmod^\J$, where we write $\Z_{\TJ}$ in place of $\Z_{\Tmod^\J}$; the proof may be found in \cite[Corollary 5.2.9]{thesis}.

\begin{cor}
\label{categoricalcharfunctors}
Let $\T$ be a quasi-equational theory with single-indeterminate isotropy and single-sorted non-total operations, and let $\J$ be a small index category.
Let $F : \J \to \PTmod$, and let \[ \pi = \left(\pi_\mu : \mathsf{cod}(\mu) \to \mathsf{cod}(\mu)\right)_{\mu \in \mathsf{Dom}(F)} \] be a $\mathsf{Dom}(F)$-indexed family of endomorphisms in $\PTmod^\J$. Then $\pi \in \Z_{\TJ}(F)$ iff there is a (uniquely determined) compatible family $\left(\phi^i\right)_{i \in \J} \in \prod_{i \in \J} \Z_\T(F(i))$ and a (uniquely determined) element $\psi = (\psi_B)_{B \in \Sigma} \in \mathsf{Aut}(\mathsf{Id}_\J)^{M^F}$ such that $\pi$ is determined by $\psi$ and $\left(\phi^{i}\right)_{i \in \J}$. \qed
\end{cor}

Before we give some specific applications of the general results that we have proven so far, we will first extract an important consequence of Theorem \ref{functortheoriestheorem} that does \emph{not} rely on the assumptions that $\T$ has single-indeterminate isotropy and single-sorted non-total operations, but which only applies to index categories $\J$ satisfying  a certain condition. Namely, we have the following consequence of (the proof of) Theorem \ref{functortheoriestheorem}, whose relatively straightforward proof can be found in \cite[Corollary 5.2.10]{thesis}:

\begin{cor}
\label{trivialisotropycorollary}
Let $\T$ be an \textbf{arbitrary} quasi-equational theory and $\J$ a small index category. If $M \in \PTJmod$ and $\J^M_B$ has only \textbf{trivial} endomorphisms for each $B \in \Sigma_{\mathsf{Sort}}$, then
\[ \beta_M : \left(\prod_i G_{\T}\left(M^i\right)\right)^\J \times \mathsf{Aut}(\mathsf{Id}_\J)^M \to G_{\TJ}(M) \] is a group isomorphism and in fact $G_{\TJ}(M) \cong \left(\prod_i G_{\T}\left(M^i\right)\right)^\J = \mathsf{lim}\left(G_\T \circ F^M\right)$. \qed
\end{cor}

\noindent We also deduce the following categorical version of the previous corollary:

\begin{cor}
\label{cattrivialisotropycorollary1}
Let $\T$ be an \textbf{arbitrary} quasi-equational theory and $\J$ a small index category with only \textbf{trivial} endomorphisms. Let $F : \J \to \PTmod$, and let $\pi = \left(\pi_\mu : \mathsf{cod}(\mu) \to \mathsf{cod}(\mu)\right)_{\mu \in \mathsf{Dom}(F)}$ be a $\mathsf{Dom}(F)$-indexed family of endomorphisms in $\PTmod^\J$. Then $\pi \in \Z_{\TJ}(F)$ iff there is a (uniquely determined) compatible family $\left(\phi^i\right)_{i \in \J} \in \prod_{i \in \J} \Z_\T(F(i))$ such that $\pi$ is determined by $\left(\phi^{i}\right)_{i \in \J}$, in the sense that
\[ \pi_\mu(k) = \phi^k_{\mu(k)} : G(k) \to G(k) \] for every morphism $\mu : F \to G$ in $\PTmod^\J$ and every $k \in \ob\J$. \qed
\end{cor}

\section{Applications}
\label{applications}

In this final section, we provide some specific applications of the general results proven in the preceding sections; in particular, we will provide an explicit characterization of the covariant isotropy group of $\Group^\J$ for any small category $\J$. 

First, for an \emph{arbitrary} quasi-equational theory $\T$, we deduce from Corollary \ref{trivialisotropycorollary} characterizations of the covariant isotropy group of $\TJ$ for certain common index categories $\J$ with only \emph{trivial} endomorphisms:

\begin{cor}
Let $\T$ be an \textbf{arbitrary} quasi-equational theory.

\begin{itemize}
\item Let $\J$ be any small discrete category. Then for any $M \in \PTJmod$, i.e. any family $\left(M^i\right)_{i \in \J}$ of $\T$-models, we have $\Z_{\TJ}(M) \cong \prod_{i \in \J} \Z_\T\left(M^i\right)$.

\item Let $\J$ be the category with two objects $i, j$ and two parallel morphisms $f, g : i \rightrightarrows j$. Then for any $M \in \PTJmod$, i.e. any parallel pair $f^M, g^M : M^i \rightrightarrows M^j$ in $\Tmod$, we have
\[ \Z_{\TJ}(M) \cong \mathsf{Eq}\left(\Z_\T\left(f^M\right), \Z_\T\left(g^M\right)\right) = \left\{ \gamma^i \in \Z_\T\left(M^i\right) : \Z_\T\left(f^M\right)\left(\gamma^i\right) = \Z_\T\left(g^M\right)\left(\gamma^i\right) \right\}, \] the equalizer of the group homomorphisms $\Z_\T\left(f^M\right), \Z_\T\left(g^M\right) : \Z_\T\left(M^i\right) \rightrightarrows \Z_\T\left(M^j\right)$. 

\item Let $\J$ be the category with three objects $i, j, k$ and two morphisms $f : i \to k$ and $g : j \to k$. Then for any $M \in \PTJmod$, i.e. any cospan $f^M : M^i \to M^k$, $g^M : M^j \to M^k$ in $\Tmod$, we have
\[ \Z_{\TJ}(M) \cong \Z_\T\left(M^i\right) \times_{\Z_\T\left(M^k\right)} \Z_\T\left(M^j\right) = \left\{ \left(\gamma^i, \gamma^j\right) \in \Z_\T\left(M^i\right) \times \Z_\T\left(M^j\right) : \Z_\T\left(f^M\right)\left(\gamma^i\right) = \Z_\T\left(g^M\right)\left(\gamma^j\right)\right\}, \] the pullback of the group homomorphisms $\Z_\T\left(f^M\right) : M^i \to M^k, \Z_\T\left(g^M\right) : M^j \to M^k$. \qed
\end{itemize}
\end{cor}

We now provide applications of our general results for quasi-equational theories $\T$ with single-indeterminate isotropy and single-sorted non-total operations, for which we can consider arbitrary small index categories. In fact, for simplicity we will restrict ourselves to \emph{single-sorted} quasi-equational theories $\T$ (which automatically have single-sorted non-total operations) with the property that $\T(M)$ is non-trivial (for the unique sort of $\T$) for every $M \in \Tmod$. A very large number of single-sorted quasi-equational theories $\T$ satisfy this latter property (including the theory of groups); indeed, if $M \in \Tmod$, then $\T(M)$ will be non-trivial as long as the $\T$-model $M\la \x \ra$ obtained by freely adjoining an indeterminate element $\x$ to $M$ has at least two elements, which is very often the case. One salient example of a single-sorted theory that does \emph{not} satisfy this property is the theory of unital rings, because the zero ring has no outgoing ring homomorphism to any non-zero unital ring and hence has trivial diagram theory. But this situation seems quite uncommon in practice (and of course, we can always restrict attention to functors $F : \J \to \Tmod$ for which each $F(i)$ has non-trivial diagram theory).    

So let $\T$ be a single-sorted quasi-equational theory with single-indeterminate isotropy such that every $M \in \Tmod$ has non-trivial diagram theory, and let $\J$ be any small index category. For any $M \in \TJmod$ we then clearly have $\mathsf{Aut}(\mathsf{Id}_\J)^M = \mathsf{Aut}(\mathsf{Id}_\J)$, and hence we obtain the following simplification of earlier results:

\begin{cor}
\label{singlesortedcorollary}
Let $\T$ be any single-sorted quasi-equational theory with single-indeterminate isotropy such that every $M \in \Tmod$ has non-trivial diagram theory, and let $\J$ be any small index category. 
\begin{itemize}
\item For any $M \in \PTJmod$ we have $\Z_{\TJ}(M) \cong \left(\prod_i \Z_\T\left(M^i\right)\right)^\J \times \mathsf{Aut}(\mathsf{Id}_\J)$, naturally in $M$.

\item Let $F : \J \to \PTmod$ and let $\pi = \left(\pi_\mu : \mathsf{cod}(\mu) \to \mathsf{cod}(\mu)\right)_{\mu \in \mathsf{Dom}(F)}$ be any $\mathsf{Dom}(F)$-indexed family of endomorphisms in $\PTmod^\J$. Then $\pi \in \Z_{\TJ}(F)$ iff there is a (uniquely determined) compatible family $\left(\phi^i\right)_{i \in \J} \in \prod_{i \in \J} \Z_\T(F(i))$ and a (uniquely determined) element $\psi \in \mathsf{Aut}(\mathsf{Id}_\J)$ such that $\pi$ is determined by $\left(\phi^i\right)_{i \in \J}$ and $\psi$, in the sense that
\[ \pi_\mu(k) = G(\psi(k)) \circ \phi^k_{\mu(k)} : G(k) \xrightarrow{\sim} G(k) \] for every morphism $\mu : F \to G$ in $\PTmod^\J$ and every $k \in \ob\J$. \qed
\end{itemize}
\end{cor}

We note that Corollary \ref{singlesortedcorollary} specializes to the characterization of covariant isotropy for presheaf categories $\Sets^\J$ obtained in \cite[Corollary 22]{FSCD}. Namely, if $\T$ is the single-sorted theory with no operations and no axioms, so that $\Tmod = \Sets$, then since $\T$ has trivial and hence single-indeterminate isotropy (see \cite[Corollary 3.1.6]{thesis}) and clearly satisfies the other assumptions of Corollary \ref{singlesortedcorollary}, it follows that for any presheaf $F : \J \to \Sets$ we have $\Z_{\Sets^\J}(F) \cong \Aut(\Id_\J)$, since $\Z_\Sets(F(i))$ is trivial for each $i \in \ob\J$. This is precisely what is stated in \cite[Corollary 22]{FSCD}.  

As promised, we now use Corollary \ref{singlesortedcorollary} to provide an explicit characterization of the covariant isotropy group of $\Group^\J$ for any small category $\J$. The quasi-equational theory $\T$ of groups is single-sorted, has single-indeterminate isotropy (see \cite[Example 4.1]{MFPS} or \cite[Proposition 3.2.7]{thesis}), and the diagram theory of any group is non-trivial (if $G$ is any group, then the group $G\la \x \ra$ is non-trivial). Now let $F : \J \to \Group$ be any functor. Then for any $i \in \ob\J$ and any $\Dom(F(i))$-indexed family of group endomorphisms $\phi^i = \left(\phi^i_f : \cod(f) \to \cod(f)\right)_{f \in \Dom(F(i))}$, we have $\phi^i \in \Z_\Group(F(i))$ iff there is a unique element $g_i \in F(i)$ such that for every outgoing group homomorphism $f : F(i) \to G$, the endomorphism $\phi^i_f : G \to G$ is the inner automorphism $\inn_{f(g_i)}$ induced by the element $f(g_i) \in G$ (see \cite[Theorem 1]{Bergman} or \cite[Corollary 3.2.8]{thesis}). We then have:

\begin{claim}
\label{compatibleclaim}
A family $\left(\phi^i\right)_{i \in \J} \in \prod_{i \in \J} \Z_\Group(F(i))$ is compatible in the sense of Definition \ref{Fcompatible} iff $F(h)(g_i) = g_j$ for any morphism $h : i \to j$ in $\J$, i.e. iff $(g_i)_{i \in \J} \in \limit \ F$. 
\end{claim} 

\begin{proof}
Suppose first that $\left(\phi^i\right)_{i \in \J} \in \prod_{i \in \J} \Z_\Group(F(i))$ is compatible, let $h : i \to j$ be a morphism of $\J$, and let us show $F(h)(g_i) = g_j$. We have the inclusion group homomorphism $\eta : F(j) \to F(j)\la \x \ra$, so by compatibility it follows that $\phi^j_{\eta} = \phi^i_{\eta \circ F(h)} : F(j)\la \x \ra \xrightarrow{\sim} F(j)\la \x \ra$, and hence $\inn_{\eta(g_j)} = \inn_{\eta(F(h)(g_i))} : F(j)\la \x \ra \xrightarrow{\sim} F(j)\la \x \ra$. We then obtain
\[ g_j \x g_j^{-1} = \inn_{\eta(g_j)}(\x) = \inn_{\eta(F(h)(g_i))}(\x) = F(h)(g_i) \x F(h)(g_i)^{-1} \] in $F(j)\la \x \ra$, which forces $F(h)(g_i) = g_j$, as desired. Conversely, if $(g_i)_{i \in \J} \in \limit \ F$ and $h : i \to j$ is a morphism of $\J$, then for any group homomorphism $f : F(j) \to G$ we have $\phi^j_f = \inn_{f(g_j)} = \inn_{(f \circ F(h))(g_i)} = \phi^i_{f \circ F(h)}$, as desired.   
\end{proof}

\noindent We now have the following application of Corollary \ref{singlesortedcorollary}:

\begin{cor}
\label{grouptheorycor}
Let $\J$ be any small index category. 
\begin{itemize}
\item For any $F : \J \to \Group$ we have $\Z_{\Group^\J}(F) \cong \limit \ F \times \mathsf{Aut}(\mathsf{Id}_\J)$, naturally in $F$.

\item Let $F : \J \to \Group$ and let $\pi = \left(\pi_\mu : \mathsf{cod}(\mu) \to \mathsf{cod}(\mu)\right)_{\mu \in \mathsf{Dom}(F)}$ be any $\mathsf{Dom}(F)$-indexed family of endomorphisms in $\Group^\J$. Then $\pi \in \Z_{\Group^\J}(F)$ iff there is a uniquely determined family $(g_i)_{i \in \J} \in \limit \ F$ and a uniquely determined element $\psi \in \mathsf{Aut}(\mathsf{Id}_\J)$ such that $\pi$ is determined by $\psi$ and $(g_i)_{i \in \J}$ in the sense that
\[ \pi_\mu(k) = G(\psi(k)) \circ \inn_{\mu_k(g_k)} : G(k) \xrightarrow{\sim} G(k) \] for every morphism $\mu : F \to G$ in $\Group^\J$ and every $k \in \ob\J$.

\item Thus, an automorphism $\pi : F \xrightarrow{\sim} F$ of $F : \J \to \Group$ is \textbf{inner} iff there are $(g_i)_{i \in \J} \in \limit \ F$ and $\psi \in \Aut(\Id_\J)$ such that $\pi(k) = F(\psi(k)) \circ \inn_{g_k} : F(k) \xrightarrow{\sim} F(k)$ for each $k \in \ob\J$.
\end{itemize}
\end{cor}

\begin{proof}
The first claim follows from the first claim of Corollary \ref{singlesortedcorollary} because $\left(\prod_i \Z_\T\left(F(i)\right)\right)^\J = \limit\left(\Z_\Group \circ F\right) \cong \limit \ F$, since $\Z_\Group \cong \Id$ by \cite[Theorem 2]{Bergman}. The second claim follows from the second claim of Corollary \ref{singlesortedcorollary} by the preceding observations (including Claim \ref{compatibleclaim}), and the last claim follows from the second and the definition of inner automorphisms in terms of $\Z_{\Group^\J}(F)$ (see the Introduction).  
\end{proof}

Thus, while Bergman's \cite[Theorem 1]{Bergman} shows that the categorical inner automorphisms in $\Group$ are precisely the conjugation-theoretic inner automorphisms, our Corollary \ref{grouptheorycor} shows that the categorical inner automorphisms in any category $\Group^\J$ of presheaves of groups can be characterized in terms of the conjugation-theoretic inner automorphisms of the component groups, together with natural automorphisms of $\Id_\J$. 

\section*{Acknowledgements}

\noindent The research in this paper was completed as part of the author's PhD thesis \cite{thesis}, and the author therefore wishes to acknowledge helpful discussions with his PhD supervisors Pieter Hofstra and Philip Scott concerning the material herein.

\section*{Competing Interests Statement}

\noindent The author has no competing interests to declare.


\section*{References}

\begin{thebibliography}{1}

\bibitem{Bergman}
G. Bergman, An inner automorphism is only an inner automorphism, but an inner endomorphism can be
something strange, 
Publicacions Matematiques 56 (2012) 91-126.

\bibitem{Funk}
J. Funk, P. Hofstra, B. Steinberg,  
Isotropy and crossed toposes, 
Theor. Appl. Cat. 26 (2012) 660-709.

\bibitem{MFPS}
P. Hofstra, J. Parker, P.J. Scott,
Isotropy of algebraic theories, 
Electron. Notes Theor. Comput. Sci. 341 (2018) 201-217.

\bibitem{FSCD}
P. Hofstra, J. Parker, P.J. Scott,
Polymorphic automorphisms and the Picard group. 
6th International Conference on Formal Structures for Computation and Deduction (FSCD 2021), N. Koyayashi, Ed.
Dagstuhl Publications LIPlcs, Vol. 195 (2021).

\bibitem{Horn}
E. Palmgren, S.J. Vickers, 
Partial Horn logic and cartesian categories,
Ann. Pure Appl. Log. 145 (2007) 314-353.

\bibitem{racks}
J. Parker,
Isotropy groups of free racks and quandles,
Journal of Algebra and its Applications (2021). https://doi.org/10.1142/S0219498822501638

\bibitem{thesis}
J. Parker, 
Isotropy groups of quasi-equational theories,
PhD Thesis, University of Ottawa (2020). http://dx.doi.org/10.20381/ruor-25256

\bibitem{Pettet}
M. Pettet, 
On inner automorphisms of finite groups, Proceedings of the American
Mathematical Society 106(1) (1989) 87-90.

\bibitem{Schupp}
P. Schupp, 
A characterization of inner automorphisms, Proceedings of the American
Mathematical Society 101(2) (1987) 226-228.

\end{thebibliography}
\end{document}